\documentclass[11pt, reqno]{amsart} 
\usepackage{amssymb,amscd,amsfonts,amsbsy}
\usepackage{latexsym}
\usepackage{exscale}
\usepackage{amsmath,amsthm,amsfonts}
\usepackage{mathrsfs}
\usepackage{xcolor} 
\usepackage[colorlinks=true,linkcolor=blue,citecolor=red,urlcolor=red, backref=page]{hyperref} 
\usepackage{esint} 

\usepackage[utf8]{inputenc}

\calclayout

\allowdisplaybreaks

\theoremstyle{plain}
\newtheorem{theorem}[equation]{Theorem}
\newtheorem{lemma}[equation]{Lemma}

\newtheorem{proposition}[equation]{Proposition}

\theoremstyle{definition}
\newtheorem{definition}[equation]{Definition}

\theoremstyle{remark}
\newtheorem{remark}[equation]{Remark}

\numberwithin{equation}{section} 

\def\diam{\operatorname{diam}}
\def\div{\operatorname{div}}
\def\dist{\operatorname{dist}}
\def\supp{\operatorname{supp}}
\def\BMO{\operatorname{BMO}}
\def\Lip{\operatorname{Lip}}
\def\loc{\operatorname{loc}}

\def\inter{\operatorname{int}}
\def\osc{\operatorname{osc}}

\def\Top{\operatorname{Top}}
\def\symmetric{\operatorname{sym}}
\def\osc{\operatorname{osc}}

\def\C{\mathcal{C}}
\def\N{\mathbb{N}}
\def\D{\mathbb{D}}

\def\F{\mathcal{F}}
\def\B{\mathbb{B}}
\def\G{\mathbb{G}}
\def\bbF{\mathbb{F}}

\def\Scal{\mathcal{S}}

\def\GG{\mathcal{G}}
\def\H{\mathcal{H}}
\def\P{\mathcal{P}}
\def\Z{\mathbb{Z}}
\def\S{\mathbf{S}}

\def\W{\mathcal{W}}
\def\R{\mathbb{R}}
\def\Rn{\mathbb{R}^n}
\def\re{\mathbb{R}}
\def\ree{\mathbb{R}^{n+1}}
\def\pom{\partial\Omega}

\def\w{\omega} 
\def\m{\mathfrak{m}}

\newcommand{\abs}[1]{\left\lvert #1 \right\rvert}

\renewcommand{\emptyset}{\text{\textup{\O}}}

\newcommand\restr[2]{\ensuremath{\left.#1\right|_{#2}}}

\def\Xint#1{\mathchoice
{\XXint\displaystyle\textstyle{#1}}%
{\XXint\textstyle\scriptstyle{#1}}%
{\XXint\scriptstyle\scriptscriptstyle{#1}}%
{\XXint\scriptscriptstyle\scriptscriptstyle{#1}}%
\!\int}
\def\XXint#1#2#3{{\setbox0=\hbox{$#1{#2#3}{\int}$ }
\vcenter{\hbox{$#2#3$ }}\kern-.585\wd0}}
\def\barint{\Xint-}

\newcommand{\bariint}{\barint\mkern-11.5mu\barint}

\begin{document}

\title[Carleson measure estimates, corona decompositions, and perturbation]
{Carleson measure estimates, corona decompositions, and perturbation of elliptic operators without connectivity}

\author{Mingming Cao}
\address{Mingming Cao\\
Instituto de Ciencias Matemáticas CSIC-UAM-UC3M-UCM\\
Con\-se\-jo Superior de Investigaciones Científicas\\
C/ Nicolás Cabrera, 13-15\\
E-28049 Ma\-drid, Spain} \email{mingming.cao@icmat.es}

\author{Pablo Hidalgo-Palencia}
\address{Pablo Hidalgo-Palencia\\
Instituto de Ciencias Matemáticas CSIC-UAM-UC3M-UCM\\
Con\-se\-jo Superior de Investigaciones Científicas\\
C/ Nicolás Cabrera, 13-15\\
E-28049 Ma\-drid, Spain
}
\address{and}
\address{Departamento de Análisis Matemático y Matemátia Aplicada
	\\
	Facultad de Matemáticas
	\\
	Universidad Complutense de Madrid
	\\
	Plaza de Ciencias, 3
	\\
	E-28040 Madrid, Spain
}
\email{pablo.hidalgo@icmat.es}

\author{José María Martell}
\address{José María Martell\\
Instituto de Ciencias Matemáticas CSIC-UAM-UC3M-UCM\\
Consejo Superior de Investigaciones Científicas\\
C/ Nicolás Cabrera, 13-15\\
E-28049 Ma\-drid, Spain} \email{chema.martell@icmat.es}

\thanks{The first and second authors are respectively supported by grants FJC2018-038526-I and CEX2019-000904-S-20-3, both funded by  MCIN/AEI/ 10.13039/501100011033. All the authors acknowledge financial support from MCIN/AEI/ 10.13039/501100011033 grants CEX2019-000904-S and PID2019-107914GB-I00. Part of this work was carried out while the three authors were visiting the Hausdorff Institute for Mathematics, Bonn (Germany). The authors express their gratitude to this institution.}

\date{\today}




\begin{abstract}
Let $\Omega \subset \mathbb{R}^{n+1}$, $n\ge 2$, be an open set with Ahlfors-David regular boundary satisfying the corkscrew condition. When $\Omega$ is connected in some quantitative form (more precisely, it satisfies the Harnack chain condition) one can establish that for any real elliptic operator with bounded coefficients, the quantitative absolute continuity of elliptic measures (i.e., its membership to the class $A_\infty$) is equivalent to the fact that all bounded null solutions satisfy Carleson measure estimates. In turn, in the same setting it is also known that these equivalent properties are stable under Fefferman-Kenig-Pipher perturbations. However, without connectivity, not much is known in general and, in particular, there is no Fefferman-Kenig-Pipher perturbation result available. 

In this paper, we work with a corona decomposition associated with the elliptic measure and show that it is equivalent to the fact that bounded null solutions satisfy partial/weak Carleson measure estimates, or to the fact that the Green function is comparable to the distance to the boundary in the corona sense. This characterization has profound consequences. First, we extend Fefferman-Kenig-Pipher's perturbation result to non-connected settings. Second, in the case of the Laplacian, these corona decompositions or, equivalently, the partial/weak Carleson measure estimates are meaningful enough to characterize the uniform rectifiability of the boundary of the open set. As a consequence, we obtain that the boundary of the set is uniformly rectifiable if bounded null solutions for any Fefferman-Kenig-Pipher perturbation of the Laplacian  satisfy (partial/weak) Carleson measure estimates. That is, there is a characterization of the uniform rectifiability of the boundary in terms of the properties of the bounded null solutions for operators whose coefficients may not possess any regularity. Third, for Kenig-Pipher operators any of the properties of the characterization is stable under transposition or under symmetrization of the matrices of coefficients. As a result, we obtain that Carleson measure estimates (or its partial/weak form) for bounded null-solutions of non-symmetric variable operators satisfying an $L^1$-Kenig-Pipher condition occur if and only if the boundary of the open set is uniformly rectifiable. Last, our results generalize previous work in settings where quantitative connectivity is assumed since, in that more topologically friendly settings, our conditions are equivalent to the fact that the elliptic measure is a Muckenhoupt weight.

\end{abstract}

\maketitle
\tableofcontents

\section{Introduction and Main results}\label{sec:intro} 
The study of the relationship between the geometry of an open set and the absolute continuity properties of its harmonic measure has attracted considerable interest over the last century, with remarkable progress over the last two decades. In summary, the emerging philosophy is that rectifiability along with suitable connectivity hypothesis is sufficient for absolute continuity of harmonic measure with respect to the surface measure, and that the rectifiability of the boundary is necessary. The first evidence comes from F. and M. Riesz \cite{RR}, who showed that for a simply connected domain in the complex plane, rectifiability of the boundary implies that harmonic measure is absolutely continuous with respect to arclength measure on the boundary. The converse was established  in \cite{AHMMMTV}, where it was proved that, in any dimension and in the absence of any connectivity condition, every piece of the boundary with finite surface measure is rectifiable, provided surface measure is absolutely continuous with respect to harmonic measure on that piece. 

On the other hand, some connectivity property seems to be necessary for absolute continuity to occur, since Bishop and Jones in \cite{BJ} constructed a uniformly rectifiable set $E$ on the plane and some subset of $F$ with zero arc-length which carries positive harmonic measure relative to the domain $\R^2 \setminus E$. Considering this point, a big effort has been recently made to understand in what open sets $\Omega \subset \ree$ and for what elliptic operators $L$ the elliptic measure $\omega_L$ is quantitatively absolutely continuous with respect to surface measure $\sigma:=\restr{\H^n}{\pom}$ on the boundary, which is formulated as $\omega_L \in A_{\infty}(\sigma)$  (cf.~Section~\ref {sec:1-sided-CAD}). One context where this theory has been satisfactorily developed is that of 1-sided chord-arc domains, that is, open sets which satisfy the interior corkscrew and Harnack chain conditions (quantitative openness and connectedness respectively) with an Ahlfors-David regular boundary. Under such background hypothesis, for the Laplacian, \cite{AHMNT, DJ, HM, HMU} give a characterization of $\w \in A_{\infty}(\sigma)$ in terms of uniform rectifiablility of the boundary (a quantitative version of rectifiability), which is also equivalent to the fact that $\Omega$ satisfies an exterior corkscrew condition and hence $\Omega$ is a chord-arc domain. Recently, in \cite{HMMTZ} (see also \cite{HMT1} and \cite{KP}), these characterizations have been extended to the so-called Kenig-Pipher operators, that is,  elliptic operators with variable coefficients whose gradient satisfies an $L^2$-Carleson condition. On the other hand, in the setting of 1-sided chord-arc domains it has been established that for general elliptic operators one can characterize $\omega_L \in A_{\infty}(\sigma)$ in terms of the fact that all bounded null solutions satisfy Carleson estimates (cf.~\cite{CHMT}). Additionally, the papers \cite{CHM, CHMT} extend the fundamental work of \cite{FKP} to the setting of 1-sided chord-arc domains and establish that the $A_\infty$ property is stable under Carleson perturbations.

However, in the absence of the Harnack chain condition, the quantitative absolute continuity of the harmonic or elliptic measure becomes delicate as can be seen from the Bishop-Jones counterexample. 
In \cite{HLMN} (see also \cite{HM4, MT}) it was shown that in an open set with Ahlfors-David regular boundary, if the associated harmonic measure satisfies the weak-$A_{\infty}$ property (a quantitative version of the absolute continuity) then the boundary is uniformly rectifiable. This latter property is in turn equivalent to the fact that all bounded harmonic functions satisfy Carleson measure estimates or are $\varepsilon$-approximable (see \cite{HMM, GMT} and also \cite{AGMT} for a class of elliptic operators with variable coefficients). Thus, even without strong connectivity assumptions (like in the Bishop-Jones counterexample), some estimates on harmonic functions ---which in
more topologically benign environments are equivalent to the $A_{\infty}$ property of harmonic measure--- remain valid.

The geometrical characterization of the weak-$A_{\infty}$ property of harmonic measure ---and hence the solvability of the $L^p$-Dirichlet problem for finite $p$, see for instance \cite{HLe}--- was obtained in \cite{AHMMT}. Assuming that an open set $\Omega$ satisfies natural (and optimal in a certain sense) background conditions (more precisely, interior corkscrew condition and Ahlfors-David regularity of the boundary), it is shown that harmonic measure belongs to weak-$A_{\infty}(\sigma)$ if and only if $\pom$ is uniformly rectifiable and $\Omega$ satisfies the weak local John condition, a condition that guarantees local non-tangential access to an ample portion of the boundary. Additionally, any of these properties is equivalent to the fact that $\Omega$ has ``interior big pieces of chord-arc domains''. In this regard we also mention Azzam \cite{A}, who obtained a similar characterization for the stronger $A_\infty$ condition.

The main goal of our present work is to consider open sets with no connectivity assumption and general elliptic operators. As explained above, the weak-$A_\infty$ property is not expected to hold, and we look to relax it so that the new condition has the following features. First, it is equivalent to the $A_\infty$ property in better settings. Second, it gives Carleson type estimates for bounded solutions. Third, in the case of the Laplacian or Kenig-Pipher operators allows one to characterize the uniform rectifiability of the boundary. Last but not least, it is stable under Fefferman-Kenig-Pipher Carleson perturbations, that is, if the disagreement of two matrices satisfy a Fefferman-Kenig-Pipher Carleson measure condition, then the desired condition should be transferable from one operator to the other. 

The condition that we consider here is a corona decomposition adapted to the elliptic measure that appeared implicit in \cite{HLMN}, playing a fundamental role in the proof of its main result, and was formalized in \cite{GMT} (see also  \cite{AGMT}) in a somehow weaker form. This corona decomposition is a partition of the collection of the dyadic cubes associated with the Ahlfors-David regular boundary in good cubes and bad cubes where the good cubes are organized in trees. The collection of bad cubes and the tops of the trees satisfy a packing condition. And, in each fixed tree, the averages of the elliptic measure, with some fixed pole at a scale above the top and normalized appropriately, is comparable to a constant (see Definition~\ref{def:corona} below). We show that any of the two versions of this corona decomposition is equivalent to the fact that the Green function is comparable to the distance to the boundary (with an appropriate normalization) in the corona sense (see Definition~\ref{def:corona}) or equivalent to  partial/weak Carleson measure estimates for bounded solutions, that is, Carleson measure estimates that only take into account one of the components of the associated Whitney region (see Definition~\ref{def:Carleson}). The precise result is as follows:

\begin{theorem}\label{thm:corona} 
Let $\Omega \subset \ree$, $n\ge 2$, be an open set with Ahlfors-David regular boundary satisfying the corkscrew condition, and let $Lu=-\div(A \nabla u)$ be a real (not necessarily symmetric) uniformly elliptic operator. Write $\omega_L$ and $G_L$ for the associated elliptic measure and Green function, respectively. Then the following statements are equivalent: 
\begin{list}{\rm (\theenumi)}{\usecounter{enumi}\leftmargin=1cm \labelwidth=1cm \itemsep=0.1cm \topsep=.2cm \renewcommand{\theenumi}{\alph{enumi}}}

\item\label{list:wL-strong} $\omega_L$ admits a strong corona decomposition (cf.~Definition~\ref{def:corona}).  

\item\label{list:wL} $\omega_L$ admits a corona decomposition (cf.~Definition~\ref{def:corona}).    

\item\label{list:GL} $G_L$ is comparable to the distance to the boundary in the corona sense (cf.~Definition~\ref{def:corona}).   

\item\label{list:CME} $L$ satisfies partial/weak Carleson measure estimates (cf.~Definition~\ref{def:Carleson}). 

\end{list}
\end{theorem}


Theorem~\ref{thm:corona} can be applied to consider Fefferman-Kenig-Pipher perturbations of elliptic operators extending \cite{FKP, MPT1, CHMT} to sets without any connectivity assumption. The result that we state next is, as far as we know, the first Fefferman-Kenig-Pipher perturbation in open sets which do not assume the Harnack chain condition: 

\begin{theorem}\label{thm:LL}
Let $\Omega \subset \ree$, $n \ge 2$, be an open set with Ahlfors-David regular boundary satisfying the corkscrew condition. Let $L_0u=-\div(A_0 \nabla u)$ and $L_1u=-\div(A_1 \nabla u)$ be real (not necessarily symmetric) uniformly elliptic operators so that $L_1$ is a Fefferman-Kenig-Pipher perturbation of $L_0$, that is, the following Carleson measure estimate holds:
\begin{equation}\label{carleson-perturb}
\sup _{\substack{x \in \partial \Omega \\ 0<r<\diam(\partial \Omega)} }
\frac{1}{\sigma(B(x, r) \cap \partial \Omega)} \iint_{B(x, r) \cap \Omega} 
\frac{\varrho(A_0, A_1)(X)^2}{\delta(X)} dX<\infty, 
\end{equation} 
where the disagreement between $A_0$ and $A_1$ in $\Omega$ is given by 
\begin{equation*}
\varrho(A_0, A_1)(X) :=\sup_{Y \in B(X, \delta(X)/2)} |A_0(Y) - A_1(Y)|, \quad X \in \Omega. 
\end{equation*} 
Then the following statements are equivalent: 
\begin{list}{\rm (\theenumi)}{\usecounter{enumi}\leftmargin=1cm \labelwidth=1cm \itemsep=0.1cm \topsep=.2cm \renewcommand{\theenumi}{\alph{enumi}}}
\item $\omega_{L_0}$ admits a (strong) corona decomposition (equivalently, $G_{L_0}$ is comparable to the distance to the boundary in the corona sense or $L_0$ satisfies partial/weak Carleson measure estimates).  
 
\item $\omega_{L_1}$ admits a (strong) corona decomposition (equivalently, $G_{L_1}$ is comparable to the distance to the boundary in the corona sense or $L_1$ satisfies partial/weak Carleson measure estimates). 
\end{list}
\end{theorem}

Much as in \cite{CHMT} we can see that any of the equivalent conditions in Theorem~\ref{thm:corona} can be transferred from $L$ to its transpose $L^{\top}$ or to its symmetric part $L^{\mathrm{sym}}$ when the matrix of coefficients are regular and some Carleson measure condition is assumed on its antisymmetric part: 

\begin{theorem}\label{thm:LLT}
Let $\Omega \subset \ree$, $n \ge 2$, be an open set with Ahlfors-David regular boundary satisfying the corkscrew condition. Let $Lu=-\div(A \nabla u)$ be a real (not necessarily symmetric) uniformly elliptic operator, let $L^{\top}$ denote the transpose of $L$, and let $L^{\symmetric}=\frac{L+L^{\top}}{2}$ be the symmetric part of $L$. Assume that $(A-A^{\top}) \in \Lip_{\loc}(\Omega)$ and that the following Carleson measure estimate holds 
\begin{equation}\label{carleson-div}
\sup_{\substack{x \in \partial \Omega \\ 0<r<\diam(\partial \Omega)} }
\frac{1}{\sigma(B(x, r) \cap \partial \Omega)} \iint_{B(x, r) \cap \Omega} 
|\div_C(A-A^{\top})(X)|^2 \delta(X) dX<\infty, 
\end{equation} 
where 
\begin{align*}
\div_C(A-A^{\top})(X)=\bigg(\sum_{i=1}^{n+1} \partial_i(a_{i,j}-a_{j,i})(X) \bigg)_{1 \leq j \leq n+1}, 
\quad X \in \Omega.  
\end{align*} 
Then the following statements are equivalent: 
\begin{list}{\rm (\theenumi)}{\usecounter{enumi}\leftmargin=1cm \labelwidth=1cm \itemsep=0.1cm \topsep=.2cm \renewcommand{\theenumi}{\alph{enumi}}}
\item $\omega_{L}$ admits a (strong) corona decomposition (equivalently, $G_{L}$ is comparable to the distance to the boundary in the corona sense or $L$ satisfies partial/weak Carleson measure estimates).  

\item $\omega_{L^{\top}}$ admits a (strong) corona decomposition (equivalently, $G_{L^{\top}}$  is comparable to the distance to the boundary in the corona sense or $L^{\top}$ satisfies partial/weak Carleson measure estimates). 
 
\item $\omega_{L^{\symmetric}}$ admits a (strong) corona decomposition (equivalently, $G_{L^{\symmetric}}$  is comparable to the distance to the boundary in the corona sense or $L^{\symmetric}$ satisfies partial/weak Carleson measure estimates).  
\end{list}
\end{theorem}

Back to the Laplacian, by means of Theorem \ref{thm:corona}, we can characterize the uniform rectifiability of the boundary of an open set with an Ahlfors-David regular boundary in terms of any of the preceding equivalent properties. In the case of (full) Carleson measure estimates, \cite[Theorem~1.1]{HMM} shows that any bounded harmonic function in $\ree\setminus E$, with $E$ uniformly rectifiable, satisfies (full) Carleson measure estimates. As a consequence, if $\Omega\subset\ree$ with $\pom$ is uniformly rectifiable then all bounded harmonic functions in $\Omega$ satisfy (full) Carleson measure estimates. The reciprocal is obtained in \cite[Theorem~1.3]{GMT} and the proof uses that uniform rectifiability is equivalent to the boundedness of the Riesz transform in $L^2$ (cf.~\cite{NTV}). Here we show that with the help of Theorem~\ref{thm:corona} the argument in  \cite{HLMN} can be easily adapted to obtain such a characterization. We would like to highlight that as a consequence of our result, we do not need full Carleson measure estimates, that is, partial/weak Carleson measure estimates suffice. Moreover, using some integration by parts argument from \cite{HMT1} and exploiting the fact that for ``good'' coefficients we can invoke Theorem~\ref{thm:LLT}, one can reduce matters to symmetric matrices and characterize the uniform rectifiability of the boundary of an open set with an Ahlfors-David regular boundary in terms of any of the preceding equivalent properties for a natural class of elliptic operators with variable coefficients. We would like to emphasize that our result extends \cite{HMT1}, where the Harnack chain condition was assumed, and, more notably, improves \cite{AGMT}, where the authors need a control of the oscillation of the matrix in points that might be in  different connected components (in our case the oscillations are always in a connected component of the set).

\begin{theorem}\label{thm:UR} 
	Let $\Omega \subset \ree$, $n\ge 2$, be an open set with Ahlfors-David regular boundary satisfying the corkscrew condition, and let $\w$ be the associated harmonic measure. Then $\pom$ is uniformly rectifiability if and only if $\w$ admits a (strong) corona decomposition. Moreover any of the previous conditions is equivalent to the fact that $G$, the associated  Green function,  is comparable to the distance to the boundary in the corona sense and/or all bounded harmonic functions satisfy partial/weak (or full)  Carleson measure estimates.  
	
	Furthermore, the same equivalences hold for  real (not necessarily symmetric) uniformly elliptic operators $Lu=-\div(A \nabla u)$ with $A$ satisfying one of the following conditions:
	
	\begin{list}{\rm (\theenumi)}{\usecounter{enumi}\leftmargin=1cm \labelwidth=1cm \itemsep=0.1cm \topsep=.2cm \renewcommand{\theenumi}{\alph{enumi}}}
		
		\item\label{thm:UR:a} $A \in \Lip_{\loc}(\Omega)$, $	\big\||\nabla A|\,\delta(\cdot)\big\|_{L^\infty(\Omega)}<\infty$ where $\delta(\cdot)=\dist(\cdot,\pom)$, and 
		\begin{equation}\label{carleson-KP}
		\sup_{\substack{x \in \partial \Omega \\ 0<r<\diam(\partial \Omega)} }
		\frac{1}{\sigma(B(x, r) \cap \partial \Omega)} \iint_{B(x, r) \cap \Omega} 
		|\nabla A (X)|dX<\infty.
		\end{equation} 
		
		\item\label{thm:UR:b} $A$ satisfies 
		\begin{equation}\label{carleson-KP:osc}
		\sup_{\substack{x \in \partial \Omega \\ 0<r<\diam(\partial \Omega)} }
		\frac{1}{\sigma(B(x, r) \cap \partial \Omega)} \iint_{B(x, r) \cap \Omega} 
		\frac{\osc (A, X)}{\delta(X)} dX<\infty,
		\end{equation} 
		where $\osc (A, X):=\sup\limits_{Y,Z\in B(X,\delta(X)/2)}|A (Y)-A(Z)|$, for $X\in\Omega$.
		 
	\item\label{thm:UR:c}  $A$ is a Fefferman-Kenig-Pipher perturbation (c.f. \eqref{carleson-perturb}) of the Laplacian or more generally any of the operators in \eqref{thm:UR:a} or \eqref{thm:UR:b}.
	\end{list}

\end{theorem}

Let us mention that in the previous result  the corkscrew condition cannot be removed, see Remark~\ref{remark:4-corner-CME} below.

Our last result studies the relationship between the corona decomposition associated with $\omega_L$ and the $A_\infty$ or weak-$A_\infty$ properties. Under strong connectivity, that is, in 1-sided chord-arc domains (see Section~\ref{sec:1-sided-CAD}), these notions turn out to be equivalent and, in view of \cite[Theorem~1.1]{CHMT}, we conclude that full and partial/weak Carleson measure estimates are equivalent properties.

\makeatletter 
\renewcommand\p@enumii{}
\makeatother

\begin{theorem}\label{thm:connected}
	Let $\Omega \subset \ree$, $n\ge 2$, be an open set with Ahlfors-David regular boundary satisfying the corkscrew condition and let $Lu=-\div(A \nabla u)$ be a real (not necessarily symmetric) uniformly elliptic operator. 
	
	\begin{list}{\rm (\theenumi)}{\usecounter{enumi}\leftmargin=1cm \labelwidth=1cm \itemsep=0.1cm \topsep=.2cm \renewcommand{\theenumi}{\roman{enumi}}}
		
		\item\label{1-sidedCAD:i} If $\omega_L\in A^{\rm weak}_\infty(\sigma)$ (cf.~Section~\ref{sec:1-sided-CAD}) then $\omega_L$ admits a (strong) corona decomposition. As a result \eqref{list:GL} and \eqref{list:CME} in Theorem \ref{thm:corona} hold. 
		
		\item\label{1-sidedCAD:ii} If  $\Omega$ is a  $1$-sided CAD (cf.~Section~\ref{sec:1-sided-CAD}), then  following are equivalent: 
		\begin{list}{\rm (\theenumii)}{\usecounter{enumii}\leftmargin=1cm \labelwidth=1cm \itemsep=0.1cm \topsep=.1cm \renewcommand{\theenumii}{\alph{enumii}}}
			
			\item\label{list:con-Ainfty} $\omega_L \in A_{\infty}(\sigma)$ (cf.~Section~\ref{sec:1-sided-CAD}).
			
			\item\label{list:con-corona}     $\omega_L$ admits a (strong) corona decomposition
			
			\item\label{list:con-Green} $G_L$ is comparable to the distance to the boundary in the corona sense.  
									
			\item\label{list:con-partial-CME} $L$ satisfies partial/weak Carleson measure estimates. 
			
			\item\label{list:con-full-CME} $L$ satisfies full Carleson measure estimates (cf.~Definition~\ref{def:Carleson}).   
			
		\end{list}

		\item\label{1-sidedCAD:iii} Let $\Omega$ be a  $1$-sided CAD and let $L$ be the  Laplacian; or a Kenig-Pipher operator, that is, an operator as in Theorem~\ref{thm:UR} part \eqref{thm:UR:a} but where \eqref{carleson-KP} is relaxed by replacing $|\nabla A|$  by $|\nabla A|^2\delta(\cdot)$; or an operator as in Theorem~\ref{thm:UR} part \eqref{thm:UR:b} but where \eqref{carleson-KP:osc} is relaxed by replacing $\osc (A,\cdot)$ by $\osc (A,\cdot)^2$; or a Fefferman-Kenig-Pipher perturbation of one of the previous  operators. Then any of the conditions \eqref{list:con-Ainfty}--\eqref{list:con-full-CME} is equivalent to the fact that $\Omega$ is a CAD or $\pom$ is uniformly rectifiable.  
		
\end{list}	
\end{theorem}

\makeatletter 
\renewcommand\p@enumii{\theenumi}
\makeatother

The paper is organized as follows. In Section \ref{sec:pre}, we present some preliminaries, definitions, and some background results that will be used throughout the paper. Sections \ref{sec:corona} is devoted to proving our main result, Theorem~\ref{thm:corona}. In Section \ref{sec:perturbation} we prove Theorems~\ref{thm:LL} and~\ref{thm:LLT}. These results turn out to be particular cases of a more general result, Theorem~\ref{thm:AAAD}, 
which is interesting on its own right. The proof of Theorem~\ref{thm:UR} is given in Section~\ref{section:UR}, the argument is just sketched because our results allow us to reduce matters to an argument essentially contained in \cite{HLMN}. Finally, in Section~\ref{sec:1-sided-CAD} we obtain Theorem~\ref{thm:connected}, where for most of the argument the open set satisfies the Harnack chain condition, in which case ``the change of pole formula'' is at our disposal.

\section{Preliminaries}\label{sec:pre}

\subsection{Notation and definitions}

\begin{list}{$\bullet$}{\leftmargin=0.4cm  \itemsep=0.2cm}
	
	\item  Our ambient space is $\ree$, $n\ge 2$.
	
	\item We use the letters $c$, $C$ to denote harmless positive constants, not necessarily the same at each occurrence, which depend only on dimension and the constants appearing in the hypotheses of the theorems (which we refer to as the ``allowable parameters''). We shall also sometimes write $a\lesssim b$ and $a\approx b$ to mean, respectively, that $a\leq C b$ and $0<c\leq a/b\leq C$, where the constants $c$ and $C$ are as above, unless explicitly noted to the contrary. Moreover, if $c$ and $C$ depend on some given parameter $\eta$, which is somehow relevant, we write $a\lesssim_\eta b$ and $a\approx_\eta b$. At times, we shall designate by $M$ a particular constant whose value will remain unchanged throughout the proof of a given lemma or proposition, but which may have a different value during the proof of a different lemma or proposition.
	
	\item Given $E \subset \R^{n+1}$ we write $\diam(E)=\sup_{x, y \in E}|x-y|$ to denote its diameter.

	\item Given an open set $\Omega\subset\re^{n+1}$, we shall use lower case letters $x,y,z$, etc., to denote points on $\partial\Omega$, and capital letters $X,Y,Z$, etc., to denote generic points in $\re^{n+1}$ (especially those in $\Omega$).
	
	
	\item The open $(n+1)$-dimensional Euclidean ball of radius $r$ will be denoted $B(x,r)$ when the center $x$ lies on $\partial\Omega$, or $B(X,r)$ when the center $X\in\re^{n+1}\setminus \partial\Omega$. A ``surface ball'' is denoted $\Delta(x,r):=B(x,r)\cap \partial\Omega$, and unless otherwise specified it is implicitly assumed that $x\in\partial\Omega$. Also if $\partial\Omega$ is bounded, we typically assume that $0<r\lesssim\diam(\partial\Omega)$, so that $\Delta=\partial\Omega$ if $\diam(\partial\Omega)<r\lesssim\diam(\partial\Omega)$.

	\item Given a Euclidean ball $B$ or surface ball $\Delta$, its radius will be denoted by $r_B$ or $r_\Delta$
	respectively.
	
	\item Given a Euclidean ball $B=B(X,r)$ or a surface ball $\Delta=\Delta(x,r)$, its concentric dilate by a factor of $\kappa>0$ will be denoted by $\kappa B=B(X,\kappa r)$ or $\kappa\Delta=\Delta(x,\kappa r)$.
	
	\item For $X\in\re^{n+1}$, we set $\delta(X):=\dist(X,\partial\Omega)$. 
	
	\item We let $\H^n$ denote the $n$-dimensional Hausdorff measure, and let $\sigma:=\H^n |_{\partial \Omega}$ denote the surface measure on $\partial \Omega$. 

	\item For a Borel set $A\subset\re^{n+1}$, we let $\inter(A)$ denote the interior of $A$, and $\overline{A}$ denote the closure of $A$. If $A\subset \partial\Omega$, $\inter(A)$ will denote the relative interior, i.e., the largest relatively open set in $\partial\Omega$ contained in $A$. Thus, for $A\subset \partial\Omega$, the boundary is then well defined by $\partial A:=\overline{A}\setminus\inter(A)$.
	
	\item For a Borel set $A\subset \partial\Omega$ with $0<\sigma(A)<\infty$, we write $\fint_{A}f\,d\sigma:=\sigma(A)^{-1}\int_A f\,d\sigma$.
	
	\item We shall use the letter $I$ (and sometimes $J$) to denote a closed $(n+1)$-dimensional Euclidean cube with 
	sides parallel to the coordinate axes, and we let $\ell(I)$ denote the side length of $I$. We use $Q$ to denote dyadic ``cubes'' on $\partial \Omega$. The latter exist, given that $\partial \Omega$ is Ahlfors-David regular (see \cite{DS1}, \cite{Ch}, and enjoy certain properties which we enumerate in Lemma \ref{lem:dyadic} below).
	
	\item We will use the symbol $\bigsqcup$ to denote a union comprised of pairwise disjoint sets.
	
\end{list}
\medskip


\begin{definition}[Ahlfors-David regular]\label{def:ADR}  
	We say that a closed set $E \subset \R^{n+1}$ is $n$-dimensional Ahlfors-David regular (or simply ADR) 
	if there is some uniform constant $C\ge 1$ such that
	\begin{equation*}
	C^{-1} r^n \leq \H^n(E \cap B(x, r)) \leq C r^n, \qquad \forall\,x \in E, \ r \in (0, 2\,\diam(E)). 
	\end{equation*} 
\end{definition}

\begin{definition}[Uniformly Rectifiable]   
A set $E \subset \R^{n+1}$ is uniformly rectifiable (or simply $UR$) if it is ADR and has big pieces of Lipschitz images of $\Rn$ (BPLI for short). The latter means that there exist $\theta, M > 0$ such that for every $x \in E$ and $r \in(0, \diam(E))$ there is a Lipschitz mapping $\rho=\rho_{x,r}: B_n(0, r) \subset \Rn \to \R^{n+1}$ with $\Lip(\rho) \leq M$ such that 
\begin{equation*}
\H^n (E \cap B(x, r) \cap \rho(B_n(0, r))) \geq \theta r^n. 
\end{equation*} 
When $E=\partial \Omega$, we shall sometimes simply say that ``$\Omega$ has the $UR$ property" to mean that $\partial \Omega$ is $UR$. 
\end{definition}

\begin{definition}[Corkscrew condition]\label{def:CKS} 
We say that an open set $\Omega \subset \R^{n+1}$ satisfies the corkscrew condition if for some uniform 
constant $c \in (0, 1)$, and for every surface ball $\Delta:=\Delta(x, r)$ with $x \in \partial \Omega$ and 
$0< r < \diam(\partial \Omega)$, there is a ball $B(X_{\Delta}, cr) \subset B(x, r) \cap \Omega$. The point 
$X_{\Delta} \in \Omega$ is called a ``corkscrew point'' relative to $\Delta$. We note that we may allow 
$r < C \diam(\partial \Omega)$ for any fixed $C$, simply by adjusting the constant $c$. 
\end{definition} 

We remark that the corkscrew condition is a quantitative, scale-invariant version of openness.

Once we have stated our main results, we give the precise definitions of all the previous concepts so to have a rigorous idea of what we are handling. First of all we proceed to define what we understand by the different corona decompositions present in Theorems~\ref{thm:corona} and \ref{thm:connected}.  

\begin{definition}\label{def:corona-basic}
	Let $\Omega \subset \ree$, $n\ge 2$, be an open set with Ahlfors-David regular boundary and let $\D:=\D(\pom)$ denote a dyadic grid on $\pom$ (cf.~Section~\ref{section:dyadic}). 
	
	\begin{list}{$\bullet$}{\leftmargin=.8cm \labelwidth=.8cm \itemsep=0.2cm \topsep=.2cm }
		\item A subcollection $\S \subset \D$ is semi-coherent if the following properties hold: 
		
		\begin{list}{\textup{(\theenumi)}}{\usecounter{enumi}\leftmargin=.8cm \labelwidth=.8cm \itemsep=0.2cm \topsep=.2cm \renewcommand{\theenumi}{\alph{enumi}}}
			\item $\S$ contains a unique maximal element, denoted by $\Top(\S)$, which contains all other elements of $\mathbf{S}$ as subsets, that is, if $Q\in\S$ then $Q\subset \Top(\S)$.
			
			\item If $Q \in \S$ and $Q \subset Q' \subset \Top(\S)$, then $Q' \in \S$.
		\end{list}
		
		\item A subcollection $\S \subset \D$ is coherent if it is semi-coherent and additionally satisfies that for every $Q\in\S$ either all of its children belong to $\S$, or else none of them do.

		\item A semi-coherent (resp. coherent) corona decomposition with constant $M_0\ge 1$ is a triple $(\B,\G,\bbF)$, where $\B$ and $\G$ are two subsets of $\D$ (the ``bad cubes'' and the ``good cubes'') and $\bbF$ is a family of subsets of $\G$ which satisfy the following conditions:
		\begin{list}{\textup{(\theenumi)}}{\usecounter{enumi}\leftmargin=.8cm \labelwidth=.8cm \itemsep=0.2cm \topsep=.2cm \renewcommand{\theenumi}{\alph{enumi}}}
			
			\item $\D=\G \bigsqcup \B$.
			
			\item $\G=\bigsqcup_{\S\in \bbF} \S$, where each $\S$ is semi-coherent (resp. coherent). 
			
			\item The collection $\B$ and the family of top cubes $\Top(\bbF):=\{\Top(\S):\S\in\bbF\}$ satisfy a Carleson packing condition:
			\begin{equation*}
			\sup_{Q \in \D} \Big(\frac1{\sigma(Q)}\sum_{Q' \in \B: Q' \subset Q} \sigma(Q') + \frac1{\sigma(Q)}\sum_{Q'\in \Top(\bbF): Q' \subset Q} \sigma(Q') 
			\Big)\le M_0.
			\end{equation*}
		\end{list} 
		
	\end{list}
	
\end{definition}

Let $(\B_0,\G_0,\bbF_0)$ be a semi-coherent corona decomposition with constant $M_0\ge 1$. As observed in \cite[pp.~56--57]{DS2}, there exists a different partition $\bbF_1$ of $\G_0$ such that $(\B_0,\G_0,\bbF_1)$ is a coherent corona decomposition with constant $M_1\ge 1$ depending on $n$, Ahlfors-David regularity, and $M_0$; and, additionally, for every $\S\in \bbF_1$, there exists $\S'\in\bbF_0$ such that $\S\subset \S'$.

Given $N\ge 0$ we say that two cubes $Q_1,Q_2\in\D$ are $2^N$-close if
\[
2^{-N}\,\ell(Q_1) \le \ell(Q_2)
\le
2^N\,\ell(Q_1),
\qquad
\dist(Q_1,Q_2)\le 2^N\,(\ell(Q_1)+\ell(Q_2)). 
\]
With this definition in mind, for every $N\ge 0$ we can invoke \cite[Lemma~3.26]{DS2} to find a new coherent corona decomposition $(\B,\G,\bbF)$ (depending on $N$) so that for every $\S\in\bbF$ there exists $\S'\in \bbF_1$ such that
\[
\S\subset \bigcup_{Q\in\S} \big\{Q'\in\D: \text{$Q$ and $Q'$ are $2^N$-close}\big\}\subset \S'.
\]
Moreover, the associated constant for the new corona decomposition depends on $n$, Ahlfors-David regularity, $N$, and  $M_1$. Putting everything together one has the following:

\begin{lemma}\label{lemma:corona-N-close}
	Let $(\B_0,\G_0,\bbF_0)$ be a semi-coherent corona decomposition with constant $M_0\ge 1$. For every $N\ge 0$, there exists a coherent corona decomposition $(\B,\G,\bbF)$ (depending on $N$) ---called the $2^N$-refinement of $(\B_0,\G_0,\bbF_0)$--- such that the associated constant depends on $n$, Ahlfors-David regularity, $N$, and  $M_0$; and, moreover, for every $\S\in\bbF$ there exists $\S'\in \bbF_0$ such that
	\[
	\S\subset \S(N):=\bigcup_{Q\in\S} \big\{Q'\in\D: \text{$Q$ and $Q'$ are $2^N$-close}\big\}\subset \S'.
	\]
\end{lemma}

\begin{definition}\label{def:corona}
	Let $\Omega \subset \ree$, $n\ge 2$, be an open set with Ahlfors-David regular boundary satisfying the corkscrew condition, and let $Lu=-\div(A\nabla u)$ be a real (non-necessarily symmetric) uniformly elliptic operator. Denote by $\omega_L$ and $G_L$ the associated elliptic measure and the Green function respectively. Let $\D:=\D(\pom)$ denote a dyadic grid on $\pom$.
	
	\begin{list}{\rm (\theenumi)}{\usecounter{enumi}\leftmargin=1cm \labelwidth=1cm \itemsep=0.1cm \topsep=.2cm \renewcommand{\theenumi}{\roman{enumi}}}
		
		\item We say that $\omega_L$ admits a (coherent/semi-coherent) strong corona decomposition if there exists a (coherent/semi-coherent) corona decomposition $(\B,\G,\bbF)$ such that for each $\S\in\bbF$ there exist $Q_{\S}\in\D$ and $X_{\S}\in\Omega$ with
		\begin{equation}\label{corona-w-strong:CKS-S}
		\Top(\S)\subset Q_\S,
		\qquad  
		\delta(X_{\S}) \approx \ell(Q_{\S}) \approx \dist(X_{\S}, Q_{\S}),
		\end{equation}
		and
		\begin{equation}\label{corona-w-strong:Aver-S}
		\frac{\omega_L^{X_{\S}}(Q_{\S})}{\sigma(Q_{\S})} 
		\lesssim \frac{\omega_L^{X_{\S}}(Q)}{\sigma(Q)} 
		\lesssim \bigg(\fint_Q (\mathcal{M} \omega_L^{X_{\S}})^{\frac12} d\sigma \bigg)^2
		\lesssim 
		\frac{\omega_L^{X_{\S}}(Q_{\S})}{\sigma(Q_{\S})} , \quad\forall Q \in \S. 
		\end{equation}

		\item We say that $\omega_L$ admits a (coherent/semi-coherent)  corona decomposition if there exists a (coherent/semi-coherent) corona decomposition $(\B,\G,\bbF)$ such that for each $\S\in\bbF$ there exist $Q_{\S}\in\D$ and $X_{\S}\in\Omega$ with
		\begin{equation}\label{corona-w:CKS-S}
		\Top(\S)\subset Q_\S,
		\qquad  
		\delta(X_{\S}) \approx \ell(Q_{\S}) \approx \dist(X_{\S}, Q_{\S}),
		\end{equation}
		and
		\begin{equation}\label{corona-w:Aver-S}
		\frac{\omega_L^{X_{\S}}(Q_{\S})}{\sigma(Q_{\S})} 
		\lesssim \frac{\omega_L^{X_{\S}}(Q)}{\sigma(Q)} 
		\lesssim \frac{\omega_L^{X_{\S}}(2\widetilde{\Delta}_Q)}{\sigma(2\widetilde{\Delta}_Q)} 
		\lesssim \frac{\omega_L^{X_{\S}}(Q_{\S})}{\sigma(Q_{\S})}, \quad\forall Q \in \S 
		\end{equation}
		(see \eqref{eq:BQ} for the definition of $\widetilde{\Delta}_Q$).

		\item We say that $G_L$ is comparable to the distance to the boundary in the corona sense if there exists a (coherent/semi-coherent) corona decomposition $(\B,\G,\bbF)$ such that for each $\S\in\bbF$ there exist $Q_{\S}\in\D$ and $X_{\S}\in\Omega$ with
		\begin{equation}\label{corona-G:CKS-S}
		\Top(\S)\subset Q_\S,
		\qquad  \delta(X_{\S}) \ge 4\Xi \ell(Q_\S), \qquad \dist(X_{\S}, Q_\S) \lesssim \ell(Q_\S),
		\end{equation}
		and
		\begin{equation}\label{corona-G:CFMS}
		\sup_{\substack{X \in 2\,\widetilde{B}_Q\cap\Omega \\ \delta(X) \ge c\,\ell(Q)}} 
		\frac{G_L(X_{\S}, X)}{\delta(X)} \approx \frac{\omega_L^{X_{\S}}(Q_\S)}{\sigma(Q_\S)}, \quad\forall Q \in \S,
		\end{equation}
		for some $c\in (0,\frac12)$ (see \eqref{eq:BQ} for the meaning of $\widetilde{B}_Q$ and the parameter $\Xi$).
	\end{list}
In the previous conditions it is understood that the implicit constants are all uniform, that is, the same constant in each estimate is valid for all $Q\in\S$ and for all $\S\in\bbF$.  
\end{definition}

\begin{remark}\label{remark:corona-G:CFMS:coherent}
	In the previous definitions we may always assume that the corona decompositions are formed by coherent subregimes. Indeed, if $(\B,\G,\bbF)$ is a semi-coherent corona decomposition with constant $M_0$ as in (i) (the cases (ii) and (iii) are treated identically) we can invoke \cite[pp.~56--57]{DS2}, to see that there is a different partition $\bbF'$ of $\G$ such that $(\B,\G,\bbF')$ is a coherent corona decomposition with constant $M_1\ge 1$ depending on $n$, Ahlfors-David regularity, and $M_0$; and, additionally, for every $\S\in \bbF'$, there exists $\widehat{\S}\in\bbF$ such that $\S\subset \widehat{\S}$. Given then $\S\in \bbF'$, we take $\widehat{\S}\in\bbF$ such that $\S\subset \widehat{\S}$. Set $Q_\S:=Q_{\widehat{\S}}$ and $X_\S:=X_{\widehat{\S}}$. Since  $\S\subset \widehat{\S}$ we have that $\Top(\S)\subset\Top(\widehat{\S})\subset Q_{\widehat{\S}}=Q_\S$ and the other two conditions in \eqref{corona-w-strong:CKS-S} follow automatically by construction. The same occurs with \eqref{corona-w-strong:Aver-S}, which holds for the cubes in $\widehat{\S}$, thus for those in $\S$. This show that $\omega_L$ admits a coherent strong corona decomposition.
	
	The same can be done with  (ii) and (iii) (details are left to the interested reader), hence from now on in the previous definitions we will drop the adjective ``coherent'' or ``semi-coherent'', with the understanding that, if needed, the corona decomposition we start with can be formed by coherent subregimes. 
\end{remark}

\begin{remark}\label{remark:corona-G:CFMS:N-close}
	One can also refine the corona decompositions in the previous definitions so that the required conditions not only hold for the cubes in the good sub-regimes but also in all the nearby cubes. More precisely, let $(\B_0,\G_0,\bbF_0)$ be a corona decomposition with constant $M_0$ as in (ii) (the cases (i) and (iii) are treated identically). Taking into account Lemma~\ref{lemma:corona-N-close} for every $N\ge 0$ we can then find $(\B,\G,\bbF)$, the $2^N$-refinement of $(\B_0,\G_0,\bbF_0)$, which is a coherent corona decomposition with associated constant depending on $n$, Ahlfors-David regularity, $M_0$, and $N$. Moreover, for every $\S\in\bbF$ there exists $\S'\in \bbF_0$ such that $\S\subset\S(N)\subset\S'$. Set $Q_{\S}:=Q_{\S'}$ and $X_{\S}:=X_{\S'}$. Write $\Top(\S)^{(N)}$ for the $N$-th dyadic ancestor of $Q$ (that is the unique dyadic cube containing $Q$ and with sidelenth $2^{N}\ell(Q)$). It is clear that $\Top(\S)^{(N)}$ and $\Top(\S)$ are $2^N$-close. Hence, $\Top(\S)^{(N)}\in\S(N)\subset \S'$ and $\Top(\S)\subset \Top(\S)^{(N)}\subset \Top(\S')\subset Q_{\S'}=Q_\S$. The other two conditions in \eqref{corona-w:CKS-S} follow by construction and we have
	\[
	\Top(\S)\subset \Top(\S)^{(N)}\subset Q_\S,
	\qquad  
	\delta(X_{\S}) \approx \ell(Q_{\S}) \approx \dist(X_{\S}, Q_{\S}),
	\]
	Finally, since $\S(N)\subset\S'$ and by construction $Q_{\S}=Q_{\S'}$ and $X_{\S}=X_{\S'}$, we have 
	\begin{align*}
	\frac{\omega_L^{X_{\S}}(Q_{\S})}{\sigma(Q_{\S})} 
	\lesssim 
	\frac{\omega_L^{X_{\S}}(Q)}{\sigma(Q)} 
	\lesssim 
	\frac{\omega_L^{X_{\S}}(2\widetilde{\Delta}_Q)}{\sigma(2\widetilde{\Delta}_Q)} 
	\lesssim 
	\frac{\omega_L^{X_{\S}}(Q_{\S})}{\sigma(Q_{\S})},
	\quad \forall\,Q \in \S(N).
	\end{align*}
	This means that we can refine the initial corona decomposition so that for every new good sub-regime $\S$ the desired property holds for every $Q\in\S(N)$, that is, for all the cubes that are $N$-close to the ones in $\S$. Also, $\ell(\Top(\S))\le 2^{-N}\,\ell(Q_S)$. Of course the same can be done with the other two definitions, the precise statements are left to the interested reader.  
\end{remark}

\begin{remark}\label{remark:Bourgain:0}
	The reader may wonder whether in (i), (ii), or (iii) it would have been reasonable or convenient to impose  Bourgain's estimate for the family of poles $\{X_\S\}_{\S\in\bbF}$, that is, whether for each $\S\in\bbF$ one would have needed to assume $\omega_L^{X_{\S}}(Q_\S)\gtrsim 1$. As we will show below, see Remark~\ref{remark:Bourgain-proof}, at the cost of possibly changing the given corona decomposition, we may always assume (in each of the items in the definition) that   $\omega_L^{X_{\S}}(Q_\S)\gtrsim 1$ for every $\S\in\bbF$. 
\end{remark}

\begin{definition}\label{def:Carleson} 
	Let $\Omega \subset \ree$, $n\ge 2$, be an open set with Ahlfors-David regular boundary satisfying the corkscrew condition, and let $Lu=-\div(A\nabla u)$ be a real (non-necessarily symmetric) uniformly elliptic operator. Let $\D:=\D(\pom)$ denote a dyadic grid on $\pom$.
	
	\begin{list}{\rm (\theenumi)}{\usecounter{enumi}\leftmargin=1cm \labelwidth=1cm \itemsep=0.1cm \topsep=.2cm \renewcommand{\theenumi}{\roman{enumi}}}
		
		\item We say that $L$ satisfies full Carleson measure estimates if 
		\begin{equation*}
		\sup_{\substack{x \in \pom \\ 0<r<\infty}} \frac{1}{r^n}  
		\iint_{B(x,r) \cap \Omega} |\nabla u(X)|^2 \delta(X) \, dX 
		\lesssim \|u\|_{L^{\infty}(\Omega)}^2,  
		\end{equation*}
		for every bounded weak solution $u \in W^{1,2}_{\loc}(\Omega) \cap L^{\infty}(\Omega)$ of $Lu=0$ in $\Omega$.  
		
		\item We say that $L$ satisfies partial/weak Carleson measure estimates (with parameter $\tau\in (0,\frac12)$)\footnote[2]{Regarding Theorem~\ref{thm:corona}, in $\eqref{list:GL} \Longrightarrow \eqref{list:CME}$ we actually show that partial/weak Carleson measure estimates hold for every $\tau\in (0,\frac12)$, and for $\eqref{list:CME} \Longrightarrow \eqref{list:wL-strong}$ one just needs some fixed $\tau$, provided it is small enough depending on $n$, Ahlfors-David regularity, and ellipticity, namely, $0<\tau<\tau_0$ with $\tau_0$ from Lemma~\ref{lemma:AGMT}. In turn, Theorem~\ref{thm:corona} shows a posteriori that partial/weak Carleson measure estimates for some fixed parameter $\tau_1$ implies the same property for all values of $\tau$ provided $\tau_1$ is sufficiently small depending on the allowable parameters.} if for every $Q \in \D$ there exists $P_Q \in \Omega$ with $\delta(P_Q) \approx \ell(Q)\approx \dist(P_Q,Q)$ such that 
		\begin{equation*}
		\sup_{Q_0 \in \D} \frac{1}{\sigma(Q_0)} \sum_{Q \in \D_{Q_0}} 
		\iint_{B(P_Q, (1-\tau)\delta(P_Q))} |\nabla u(X)|^2 \delta(X) \, dX 
		\lesssim_\tau \|u\|_{L^{\infty}(\Omega)}^2,  
		\end{equation*}
		for every bounded weak solution $u \in W^{1,2}_{\loc}(\Omega) \cap L^{\infty}(\Omega)$ of $Lu=0$ in $\Omega$.  
	\end{list} 
\end{definition}

\begin{remark}\label{remark:Bourgain:1}
	Much as in Remark~\ref{remark:Bourgain:0}, at the cost of possibly changing the collection of points $\{P_Q\}_{Q\in\D}$, we may always assume that $\omega_L^{P_Q}(Q)\gtrsim 1$ for every $Q\in\D$, see  Remark~\ref{remark:Bourgain-proof}. 
\end{remark}

\subsection{Dyadic grids and sawtooths}\label{section:dyadic}
We give a lemma concerning the existence of a ``dyadic grid'', which was proved in \cite{DS1, DS2, Ch}.  

\begin{lemma}\label{lem:dyadic}
Suppose that $E \subset \R^{n+1}$ is an $n$-dimensional $ADR$ set. Then there exist constants
$C_1\ge 1$ and $\gamma>0$ depending only on $n$ and the $ADR$ constant 
such that, for each $k \in \Z$, there is a collection of Borel sets (cubes)
\begin{align*}
\D_k =\{Q_j^k \subset E: j \in \mathfrak{J}_k\}
\end{align*}
where $\mathfrak{J}_k$ denotes some (possibly finite) index set depending on $k$, satisfying: 
\begin{list}{$(\theenumi)$}{\usecounter{enumi}\leftmargin=1cm \labelwidth=1cm \itemsep=0.1cm \topsep=.2cm \renewcommand{\theenumi}{\alph{enumi}}}

\item $E=\bigcup_j Q_j^k$, for each $k \in \Z$. 

\item If $m \leq k$, then either $Q_i^k \subset Q_j^m$ or $Q_i^k \cap Q_j^m = \emptyset$. 

\item For each $k\in\Z$, $j\in \mathfrak{J}_k$, and each $m<k$, there is a unique $i\in \mathfrak{J}_k$ such that $Q_j^k \subset Q_i^m$. 

\item For each $k\in\Z$, $j\in \mathfrak{J}_k$, there is $x_j^k\in E$ such that 
\[
B(x_j^k, C_1^{-1} 2^{-k}) \cap E\subset Q_j^k \subset B(x_j^k, C_1 2^{-k}) \cap E.
\]

\item $\H^n(\{x \in Q_j^k: \dist(x, E \backslash Q_j^k) \leq 2^{-k} \tau\}) \leq C_1 \tau^{\gamma} \H^n(Q_j^k)$ 
for all  $k\in\Z$, $j\in \mathfrak{J}_k$, and $\tau \in (0, 1)$.
\end{list}
\end{lemma}

A few remarks are in order concerning this lemma. 
\begin{list}{$\bullet$}{\leftmargin=.8cm \labelwidth=.8cm \itemsep=0.2cm \topsep=.2cm }

\item In the setting of a general space of homogeneous type, this lemma has been proved by Christ \cite{Ch}, 
with the dyadic parameter $1/2$ replaced by some constant $\delta \in (0, 1)$. In fact, one may always take $\delta=1/2$ 
(cf. \cite[Proof of Proposition 2.12]{HMMM}). In the presence of the Ahlfors-David property, the result already appears in 
\cite{DS1, DS2}. 

\item Since in the present scenario we have that $\diam(\Delta(x,r))\approx r$ we clearly have that if $E$ is bounded and $k\in\Z$ is such that $\diam(E) < C_1\,2^{-k}$, then there cannot be two distinct cubes in $\D_k$. Thus, $\D_k=\{Q^k\}$ with $Q^k=E$. Therefore, for our purposes, we may ignore those $k \in \Z$ such that $2^{-k} \gtrsim \diam(E)$, in the case that the latter is finite. Hence, we shall denote by $\D(E)$ the collection of all relevant $Q^k_j$, i.e., 
\begin{align*}
\D(E) := \bigcup_{k \in \Z} \D_k,
\end{align*}
where, if $\diam(E)$ is finite, the union runs over $k\ge -k_0$ with $2^{k_0} \approx \diam(E)$ and there exits $Q\in\D_{-k_0}$  so that $E=Q$. 

\item For a dyadic cube $Q \in \D_k$, we shall set $\ell(Q)=2^{-k}$, and we shall refer to this quantity as the 
``length'' of $Q$. Evidently, $\ell(Q) \approx \diam(Q)$. We set $k(Q)=k$ to be the dyadic generation to which $Q$ 
belongs if $Q \in \D_k$; thus, $\ell(Q)=2^{-k(Q)}$. One can easily see that if $Q\in \D_{k}$ and $Q'\in \D_{k'}$ with $Q\subsetneq Q'$ then necessarily $k'<k$. However, it is possible to have two cubes $Q\in\D_k$ and $Q'\in\D_{k'}$, such that $k\neq k'$ and $Q=Q'$. In that case the ADR condition implies that $k\approx k'$. To avoid some technicalities whenever we write $Q\subset Q'$ we will understand (unless otherwise is specified) that matters are organized so that $Q\in \D_{k}$, $Q'\in \D_{k'}$, and $k'<k$.

\item Write $\Xi=2 C_1^2$. Property $(d)$ implies that for each cube $Q \in \D$, there is a point $x_Q \in E$, a Euclidean ball 
$B(x_Q, r_Q)$ and a surface ball $\Delta(x_Q, r_Q) := B(x_Q, r_Q) \cap E$, with $\Xi^{-1}\ell(Q) \leq r_Q \leq \ell(Q)$ (indeed $r_Q=(2C_1)^{-1}\ell(Q))$, such that 
\begin{align}\label{eq:Q-DQ} 
\Delta(x_Q, 2r_Q) \subset Q \subset \Delta(x_Q, \Xi r_Q), 
\end{align}
We shall write 
\begin{align}\label{eq:BQ} 
B_Q := B(x_Q, r_Q),
\quad 
\widetilde{B}_Q := B(x_Q, \Xi r_Q),
\quad 
\Delta_Q := \Delta(x_Q, r_Q), 
\quad 
\widetilde{\Delta}_Q := \Delta(x_Q, \Xi r_Q), 
\end{align}
and we shall refer to the point $x_Q$ as the ``center'' of $Q$.

\item Let $Q\in\D_k$ and consider  the family of its dyadic children $\{Q'\in \D_{k+1}: Q'\subset Q\}$. Note that for any two distinct children $Q', Q''$, one has $|x_{Q'}-x_{Q''}|\ge r_{Q'}=r_{Q''}=r_Q/2$, otherwise $x_{Q''}\in Q''\cap \Delta_{Q'}\subset Q''\cap Q'$, contradicting the fact that $Q'$ and $Q''$ are disjoint. Also $x_{Q'}, x_{Q''}\in Q\subset \Delta(x_Q,r_Q)$, hence by the geometric doubling property we have a purely dimensional bound for the number of such $x_{Q'}$ and hence the number of dyadic children of a given dyadic cube is uniformly bounded.

\end{list}

We next introduce the notation of ``Carleson region'' and ``discretized sawtooth'' from \cite[Section 3]{HM}. 
Given a dyadic cube $Q \in \D(E)$, the ``discretized Carleson region'' $\D_Q$ relative to $Q$ is defined by 
\[ \D_Q := \{Q' \in \D(E) : Q' \subset Q\}. \] 
Let $\F=\{Q_j\} \subset \D(E)$ be a family of pairwise disjoint cubes. The ``global discretized sawtooth'' relative to $\F$ 
is the collection of cubes $Q \in \D(E)$ that are not contained in any $Q_j \in \F$, that is, 
\begin{equation*}
\D_{\F} :=\D(E) \setminus \bigcup_{Q_j \in \F} \D_{Q_j}.  
\end{equation*}
For a given cube $Q \in \D(E)$, we define the ``local discretized sawtooth'' relative to $\F$ is the collection of cubes in 
$\D_Q$ that are not contained in any $Q_j \in \F$ of, equivalently, 
\begin{equation*}
\D_{\F, Q} :=\D_Q \setminus \bigcup_{Q_j \in \F} \D_{Q_j}=\D_{\F} \cap \D_Q. 
\end{equation*} 

We also introduce the ``geometric'' Carleson regions and sawtooths. In the sequel, $\Omega \subset \R^{n+1}$,  
$n \geq 2$, is an open set with ADR boundary and satisfying the corkscrew condition.  Given $Q \in \D:=\D(\partial \Omega)$, we define the ``corkscrew point relative to $Q$'' as 
$X_Q:=X_{\Delta_Q}$. We then note that 
\[ 
\delta(X_Q) \approx \dist(X_Q, Q) \approx \diam(Q). 
\] 

Our next goal is to define some associated regions which inherit the good properties of $\Omega$. Let $\W=\W(\Omega)$ denote a collection of (closed) dyadic Whitney cubes 
of $\Omega$, so that the cubes in $\W$ form a covering of $\Omega$ with non-overlapping interiors, which satisfy 
\begin{equation*}
4 \diam(I) \leq \dist(4I, \partial \Omega) \leq \dist(I, \partial \Omega) \leq 40 \diam(I), \quad \forall\,I \in \W, 
\end{equation*}
and also 
\begin{equation*}
(1/4) \diam(I_1) \leq \diam(I_2) \leq 4 \diam(I_1), \quad\text{whenever $I_1$ and $I_2$ touch}.
\end{equation*}
Let $X(I)$ be the center of $I$ and $\ell(I)$ denote the sidelength of $I$. 

Given $0<\lambda<1$ and $I \in \W$, we write $I^*=(1+\lambda)I$ for the ``fattening'' of $I$. 
By taking $\lambda$ small enough, we can arrange matters, so that for any $I, J \in \W$, 
\begin{equation*}
\begin{aligned} 
\dist(I^*, J^*) & \approx \dist(I, J), \\ 
\operatorname{int}(I^*) \cap \operatorname{int}(J^*) \neq \emptyset & 
\Longleftrightarrow \partial I \cap \partial J \neq \emptyset.  
\end{aligned}
\end{equation*}
(The fattening thus ensures overlap of $I^*$ and $J^*$ for any pair $I, J \in \W$ whose boundaries touch, so that the 
Harnack chain property then holds locally, with constants depending upon $\lambda$, in $I^* \cap J^*$.) By choosing 
$\lambda$ sufficiently small, say $0<\lambda<\lambda_0$, we may also suppose that there is a $\tau \in (1/2, 1)$ such 
that for distinct $I, J \in \W$, we have that $\tau J \cap I^{*}=\emptyset$.  
In what follows we will need to work with the dilations $I^{**}=(1+2\lambda)I$ or $I^{***}=(1+4\lambda)I$, and in 
order to ensure that the same properties hold we further assume that $0<\lambda<\lambda_0/4$.

Given $\vartheta\in\mathbb{N}$,  for every cube $Q \in \D$ we set  
\begin{equation}\label{eq:WQ}
\W_Q^\vartheta :=\left\{I \in \W: 2^{-\vartheta}\ell(Q) \leq \ell(I) \leq 2^\vartheta\ell(Q), \text { and } \dist(I, Q) \leq 2^\vartheta \ell(Q) \right\}.
\end{equation}
We will choose $\vartheta\ge \vartheta_0$,  with $\vartheta_0\ge 6+\log_2 n$ large enough depending on the constants of the corkscrew condition (cf. Definition  \ref{def:CKS})  and in the dyadic cube construction
(cf. Lemma \ref{lem:dyadic}), so that $X_Q \in I$ for some $I \in \W_Q^\vartheta$, and for each dyadic child $Q^j$ of $Q$, the respective corkscrew points $X_{Q^j}\in I^j$ for some $I^j \in \W_Q^\vartheta$. 
Given $I\in\W$ with $\ell(I)\lesssim\diam(\pom)$ and let $Q_I^*$ be one of the nearest dyadic cubes to $I$ so that $\ell(I)=\ell(Q_I^*)$.  Clearly, $\dist(I,Q_I^*)\le 40\sqrt{n+1} \ell(I)=40\sqrt{n+1}\ell(Q_I^*)$. Hence $I\in  \W_{Q_I^*}^\vartheta$ since $\vartheta\ge \vartheta_0$.


Given $Q\in\D$ we define its associated Whitney regions $U_Q^\vartheta$ and $\widehat{U}_Q^\vartheta$ (not necessarily connected) as 
\begin{equation*}
	U_{Q}^\vartheta :=\bigcup_{I \in \W_{Q}^{\vartheta}}I^*,
	\qquad
	U_{Q}^{\vartheta,*} :=\bigcup_{I \in \W_{Q}^{\vartheta}}I^{**},
\end{equation*}
For a given $Q \in \D$, the ``Carleson boxes'' relative to $Q$ are defined by  
\begin{equation*}
	T_{Q}^\vartheta :=\operatorname{int}\bigg(\bigcup_{Q' \in \D_Q} U_{Q'}^\vartheta\bigg),
	\qquad
	T_{Q}^{\vartheta,*} :=\operatorname{int}\bigg(\bigcup_{Q' \in \D_Q} U_{Q'}^{\vartheta,*}\bigg).
\end{equation*}  
For a given family $\F=\{Q_j\}$ of pairwise disjoint cubes and a given $Q \in \D$, we define the 
``local sawtooth regions'' relative to $\F$ by 
\begin{equation*}
	\Omega_{\F, Q}^\vartheta :=\inter\bigg(\bigcup_{Q' \in \D_{\F, Q}} U_{Q'}^\vartheta\bigg), 
	\qquad
		\Omega_{\F, Q}^{\vartheta,*} :=\inter\bigg(\bigcup_{Q' \in \D_{\F, Q}} U_{Q'}^{\vartheta,*}\bigg). 
\end{equation*}

Following \cite{HM}, one can easily see that there exist constants $0<\kappa_1<1$ and 
$\kappa_0 \geq 16\Xi$ (with $\Xi$ the constant in \eqref{eq:Q-DQ}), depending only on the allowable parameters and on $\vartheta$, so that
\begin{align}\label{eq:kappa}
	\kappa_1B_Q \cap \Omega \subset T_Q^\vartheta \subset T^{\vartheta,*}_Q \subset T^{\vartheta,**}_Q \subset \overline{T^{\vartheta,**}_Q} 
	\subset \kappa_0 B_Q \cap \overline{\Omega} =: \frac12 B^*_Q \cap \overline{\Omega}, 
\end{align}
where $B_Q$ is defined as in \eqref{eq:BQ}.

We recall that $\Omega$ is an open set with ADR boundary and this allows us to define  the open set $\Omega':=\ree\setminus \pom$ whose boundary is $\pom'=\pom$. One can then proceed as above and define the associated local sawtooth regions $(\Omega')_{\F, Q}^\vartheta$ with respect to $\Omega'$ (that is, we now have some underlying Whitney decomposition for $\Omega'$ which agrees with $\W(\Omega)$ when restricted to $\Omega$). In \cite[Proposition A.2]{HMM} it was shown that all $(\Omega')_{\F, Q}^\vartheta$ have ADR boundaries and the constants are uniform and depend on $n$, the ADR constant of $\pom'=\pom$, and $\vartheta$. One can easily see that $\partial \Omega_{\F, Q}^\vartheta\subset \partial(\Omega')_{\F, Q}^\vartheta$ since $\Omega_{\F, Q}^\vartheta =(\Omega')_{\F, Q}^\vartheta\cap \Omega$. Thus, it is trivial to obtain that all local sawtooth regions $\Omega_{\F, Q}^\vartheta$ have boundary satisfying the upper ADR condition (that is only the upper estimate) albeit with bounds that are uniform and depend on $n$, the ADR constant of $\pom$, and $\vartheta$. In short we have the following:

\begin{lemma}\label{lemma:upper-ADR-sawtooth}
Let $\Omega \subset \ree$, $n\ge 2$, be an open set with Ahlfors-David regular boundary satisfying the corkscrew condition.  For every $\vartheta\ge \vartheta_0$, all sawtooth domains $\Omega_{\F, Q}^\vartheta$ and $\Omega_{\F, Q}^{\vartheta,*}$ have upper ADR boundary (that is, they satisfy the upper bound in Definition \ref{def:ADR}). Moreover, the implicit constants are uniform and depend only on
dimension, the ADR constant of $\pom$ and the parameter $\vartheta$.
\end{lemma}

\begin{lemma}\label{lemma:VQ-overlap}
	Given $c,\tau\in (0,\frac12)$ there exist $\vartheta, C \gg 1$ (depending on $n$, ADR, $\tau$, and $c$) with the following significance: 
	for every $Q_0\in\D$, every pairwise disjoint collection $\F\subset\D_{Q_0}$, and every sequence $\{P_Q\}_{Q\in\D_{\F,Q_0}}$ so that $P_Q\in 2\,\widetilde{B}_Q\cap\Omega$ with $\delta(P_Q) \ge c\,\ell(Q)$ there hold
	\[
	\bigcup_{Q\in\D_{\F,Q_0}} B(P_Q, (1-\tau)\delta(P_Q))
	\subset
	\bigcup_{Q\in\D_{\F,Q_0}} \bigcup_{I\in \W_{Q}^{\vartheta}} I
	\subset
	\inter\Big(	\bigcup_{Q\in\D_{\F,Q_0}} U_Q^{\vartheta}\Big)
	=
	\Omega^{\vartheta}_{\F,Q_0}.	
	\] 
and
	\[
\sum_{Q\in\D_{\F,Q_0}} \mathbf{1}_{B(P_Q, (1-\tau)\delta(P_Q))}
\le 
C\, \mathbf{1}_{\Omega^{\vartheta}_{\F,Q_0}}.
\]
\end{lemma}

\begin{proof}
For each $Q\in\D_{\F,Q_0}$ write $V_Q:= B(P_Q, (1-\tau)\delta(P_Q))$. Let $I\in\W$ be such that $I\cap V_Q\neq\emptyset$. Taking $Z_Q\in I\cap V_Q$ we note that 
\[
\ell(I)
\approx
\delta (Z_Q)
\approx_\tau
\delta (P_Q)
\approx_c 
\ell(Q),
\]
where the implicit constants depend on $\tau$ and the parameter $c$, and the other allowable parameters. Besides,
\[
\dist(I,Q)
\le
|Z_Q-P_Q|+|P_Q-x_Q|
\le
(1-\tau)\delta(P_Q)+2\Xi r_Q
\lesssim
\ell(Q)
\]
where the implicit constant depends on the allowable parameters. This means that we can find $\vartheta\gg 1$ (depending on $n$, ADR, $\tau$, and $c$) such that $I\in \W_{Q}^{\vartheta}$ (cf. \eqref{eq:WQ}). As a consequence, 
\[
\bigcup_{Q\in\D_{\F,Q_0}} V_Q
\subset 
\bigcup_{Q\in\D_{\F,Q_0}} \bigcup_{I\in\W : I\cap V_Q\neq\emptyset} I
\subset
\bigcup_{Q\in\D_{\F,Q_0}} \bigcup_{I\in \W_{Q}^{\vartheta}} I
\subset
\inter\Big(	\bigcup_{Q\in\D_{\F,Q_0}} U_Q^{\vartheta}\Big)
=
\Omega^{\vartheta}_{\F,Q_0}.	
\] 
To complete the proof let $X\in\bigcup_{Q\in\D_{\F,Q_0}} V_Q$ and pick  $Q\in \D_{\F,Q_0}$ so that $X\in V_Q$ with $Q\in \D_{\F,Q_0}$. Note that if 
$X\in V_Q'$ for $Q'\in \D_{\F,Q_0}$, then $\ell(Q)\approx_\tau \delta(X)\approx_\tau \ell(Q')$ and
\[
\dist(Q,Q')
\le
|x_Q-P_Q|+|P_Q-X|+|X-P_{Q'}|+|P_{Q'}-x_{Q'}|
\lesssim_\tau
\ell(Q)+\ell(Q'),
\]
thus $Q$ and $Q'$ are $2^N$-close with a uniform constant $N$ depending just on $n$, ADR, $\tau$, and $c$. As a result, 
\[
\sum_{Q'\in\D_{\F,Q_0}} \mathbf{1}_{V_{Q'}}(X)
=
\#\{Q'\in\D_{\F,Q_0}: V_{Q'}\ni X\}
\le
\#\{Q'\in\D: \text{$Q'$ and $Q_0$ are $2^N$-close}\}
\lesssim_N 1.
\]
This completes the proof. 
\end{proof}

We need the following auxiliary which adapts \cite[Lemma~4.44]{HMT1} and \cite[Lemma~6.4]{CDMT} to our current setting. The proof is the same as that of \cite[Lemma~6.4]{CDMT} and details are left to the interested reader. 

\begin{lemma}\label{lem:approx}
Let $\Omega \subset \ree$, $n\ge 2$, be an open set with Ahlfors-David regular boundary satisfying the corkscrew condition. Given $Q_0 \in \D$, a pairwise disjoint collection $\F \subset \D_{Q_0}$, and $M\ge 4$ let $\F_M$ be the family of maximal cubes of the collection ${\F}$ augmented by adding all the cubes $Q \in \D_{Q_0}$ such that $\ell(Q) \leq 2^{-M} \ell(Q_0)$. There exist $\Psi_M \in \mathscr{C}_c^{\infty}(\R^{n+1})$ and a constant $C\ge 1$ depending only on dimension $n$, the corkscrew and ADR constants, but independent of $M$, $\F$, and $Q_0$ such that the following hold:
	
\begin{list}{$(\theenumi)$}{\usecounter{enumi}\leftmargin=.8cm
			\labelwidth=.8cm\itemsep=0.2cm\topsep=.1cm
			\renewcommand{\theenumi}{\roman{enumi}}}
		
\item $C^{-1}\,\mathbf{1}_{\Omega^\vartheta_{\F_M, Q_0}} \le \Psi_M \le \mathbf{1}_{\Omega_{\F_M, Q_0}^{\vartheta, *}}$.
		
\item $\sup_{X\in \Omega} |\nabla \Psi_M(X)|\, \delta(X) \le C$.
		
\item Setting
\begin{equation*}
\W_M^\vartheta:=\bigcup_{Q \in\D_{\F_{M},Q_0}} \W_Q^{\vartheta}, \qquad
\W_M^{\vartheta,\Sigma}:= \big\{I\in \W_M^\vartheta:\, \exists\,J\in \W \setminus \W_M^\vartheta\ \text{with}\ \partial I \cap \partial J \neq \emptyset	\big\},
\end{equation*}
\end{list}
one has
\begin{equation*}
\nabla \Psi_M \equiv 0 \quad \mbox{in}	\quad 
\bigcup_{I\in \W_M^\vartheta \setminus \W_M^{\vartheta, \Sigma} }I^{**},
\end{equation*}
and there exists a family $\{\widehat{Q}_I\}_{I \in \W_M^{\vartheta,\Sigma}}$ so that
\begin{equation*}
C^{-1}\,\ell(I)\le \ell(\widehat{Q}_I)\le C\,\ell(I), \qquad
\dist(I, \widehat{Q}_I)\le C\,\ell(I), \qquad
\sum_{I \in \W_M^{\vartheta,\Sigma}} \mathbf{1}_{\widehat{Q}_I} \le C.
\end{equation*}
\end{lemma}

\subsection{PDE estimates} 
Now we recall several facts concerning the elliptic measures and the Green functions. For our first results 
we will only assume that $\Omega \subset \R^{n+1}$, $n \geq 2$, is an open set, not necessarily connected, 
with $\partial \Omega$ being ADR. Later we will focus on the case where $\Omega$ is a 
$1$-sided CAD. 

Let $Lu = - \div(A \nabla u)$ be a variable coefficient second order divergence form operator with 
$A(X)=(a_{i, j}(X))_{i,j=1}^{n+1}$ being a real (not necessarily symmetric) matrix with 
$a_{i, j} \in L^{\infty}(\Omega)$ for $1 \leq i, j \leq n+1$, and $A$ uniformly elliptic, that is, there 
exists $\Lambda \geq 1$ such that 
\begin{align*}
\Lambda^{-1} |\xi|^2 \leq A(X) \xi \cdot \xi \quad\text{ and }\quad |A(X) \xi \cdot \eta| \leq \Lambda |\xi| |\eta|, 
\end{align*}
for all $\xi, \eta \in \R^{n+1}$ and for almost every $X \in \Omega$. 

In what follows we will only be working with this kind of operators, we will refer to them as ``elliptic operators'' 
for the sake of simplicity. We write $L^{\top}$ to denote the transpose of $L$, or, in other words, 
$L^{\top}u = -\div(A^{\top}\nabla u)$ with $A^{\top}$ being the transpose matrix of $A$. 

We say that a function $u \in W_{\loc}^{1,2}(\Omega)$ is a weak solution of $Lu=0$ in $\Omega$, or that
$Lu=0$ in the weak sense, if 
\begin{align*}
\iint_{\Omega} A(X) \nabla u(X) \cdot \nabla \Phi(X)=0, \quad \forall\,\Phi \in \mathscr{C}_c^{\infty}(\Omega). 
\end{align*}
Here and elsewhere $\mathscr{C}_c^{\infty}(\Omega)$ stands for the set of compactly supported smooth functions with all derivatives of all orders being continuous.

Associated with the operators $L$ and $L^{\top}$,  one can respectively construct the elliptic measures 
$\{\omega_L^X\}_{X \in \Omega}$ and $\{\omega_{L^\top}^X\}_{X \in \Omega}$, and the Green functions 
$G_L$ and $G_{L^{\top}}$ (see \cite{HMT2} for full details). We next present some definitions and properties 
that will be used throughout this paper.

The following lemmas can be found in \cite{HMT2}.

\begin{lemma}
Suppose that $\Omega \subset  \R^{n+1}$, $n\ge 2$, is an open set such that $\partial \Omega$ is 
ADR. Given an elliptic operator $L$, there exist $C>1$ (depending only on dimension 
and on the ellipticity of $L$) and $c_{\theta}>0$ (depending on the above parameters and on 
$\theta \in (0, 1)$) such that $G_L$, the Green function associated with $L$, satisfies 
\begin{align}
G_L(X, Y) &\leq C|X-Y|^{1-n}; \\[.2cm] 
c_{\theta} |X-Y|^{1-n} \leq G_L(X, Y), \quad &\text{if } |X-Y| \leq \theta \delta(X), \theta \in(0,1); \\[.2cm] 
G_L(\cdot, Y) \in C(\overline{\Omega} \setminus \{Y\}) &\text{ and } 
G_L(\cdot, Y)|_{\partial \Omega} \equiv 0, \forall\,Y \in \Omega;  \\[.2cm] 
G_L(X, Y) \geq 0, &\quad \forall\,X, Y \in \Omega, X \neq Y;  \\[.2cm] 
G_L(X, Y)=G_{L^{\top}}(Y, X), &\quad \forall\,X, Y \in \Omega, X \neq Y. 
\end{align}
Moreover, $G_L(\cdot, Y) \in W^{1,2}_{\loc}(\Omega \setminus \{Y\})$ for every $Y \in \Omega$, 
and satisfies $LG_L(\cdot, Y)=\delta_Y$ in the weak sense in $\Omega$, that is,  
\begin{equation}
\iint_{\Omega} A(X) \nabla_{X} G_L(X, Y) \cdot \nabla \Phi(X) dX = \Phi(Y), \quad \forall\,\Phi \in \mathscr{C}_c^{\infty}(\Omega). 
\end{equation} 
Finally, the following Riesz formula holds 
\begin{equation}\label{eq:Riesz}
\iint_{\Omega} A^{\top}(Y) \nabla_{Y} G_{L^{\top}}(Y, X) \cdot \nabla \Phi(Y) dY 
= \Phi(X) - \int_{\partial \Omega} \Phi d\omega_L^X, 
\end{equation} 
for a.e.~$X \in \Omega$ and for every $\Phi \in \mathscr{C}_c^{\infty} (\R^{n+1})$.  
\end{lemma}

\begin{lemma}\label{lem:Bourgain-CFMS}
Suppose that $\Omega \subset  \R^{n+1}$, $n\ge 2$, is an open set with ADR boundary. Let $L$ be an elliptic operator. There exists a constant $C>1$ (depending only on the dimension, the ADR constant and the ellipticity of $L$) such that  we 
have the following properties: 

\begin{list}{\textup{(\theenumi)}}{\usecounter{enumi}\leftmargin=1cm \labelwidth=1cm \itemsep=0.2cm \topsep=.2cm \renewcommand{\theenumi}{\alph{enumi}}}

\item  For every $x \in \partial \Omega$ and $0<r<\diam(\partial \Omega)$ there holds
\begin{align}\label{Bourgain}
\omega_L^Y(\Delta(x,r)) \geq 1/C, \quad \forall\,Y \in \Omega \cap B(x, C^{-1}r). 
\end{align}

\item  Given $X,Y\in\Omega$ such that $|X-Y|\ge \delta(Y)/2$, then
\begin{align}\label{eq:CFMS:1} 
\frac{G_L(X,Y)}{\delta(Y)}
\le 
C\,
\frac{\omega_L^X(\Delta(\widehat{y},2\delta(Y)))}{\sigma(\Delta(\widehat{y},2\,\delta(Y)))}
\end{align}
with $\widehat{y}\in\Omega$ such that $|Y-\widehat{y}|=\delta(Y)$. 

\item If $0 \le u \in W^{1,2}_{\loc}(B_0 \cap \Omega) \cap \mathscr{C}(\overline{B_0 \cap \Omega})$ satisfies $Lu=0$ in the weak-sense in $B_0 \cap \Omega$ and $u\equiv 0$ in $\Delta_0$ then
\[
u(X) \le C\, \bigg(\frac{|X-x_0|}{r_0}\bigg)^{\gamma} \sup_{Y \in \overline{B_0 \cap \Omega}} u(Y), \qquad\forall X \in \frac12 B_0 \cap \Omega.
\]
		\end{list}
\end{lemma}

\begin{lemma}[{\cite[Lemma~3.3]{AGMT}}]\label{lemma:AGMT}
There exists a small enough parameter $\tau_0\in (0,\frac12)$ (depending on $n$, ADR and ellipticity) such that given $0<\tau<\tau_0$ there exists $\epsilon_\tau>0$ small enough (depending on the same parameters and additionally on $\tau$), so that for each $Q\in\D$, $\epsilon\in (0,\epsilon_\tau]$, $Y_Q\in\frac12 B_Q\cap\Omega$ (cf.~\eqref{eq:BQ}) with $\delta(Y_Q)\approx \epsilon \ell(Q)$, and for each Borel set $E_Q\subset Q$ satisfying 
\[
\omega_L^{Y_Q}(E_Q) \ge (1-\epsilon)\, \omega_L^{Y_Q}(Q),
\]
there exists a  Borel function  $0\le f_Q\lesssim  \mathbf{1}_{E_Q}$ so that 
\[
u_Q(X):=\int_{E_Q} f_Q\,d\omega_L^X, \qquad X\in\Omega,
\]
satisfies $u_Q\in W^{1,2}_{\loc}(\Omega)$, $Lu_Q=0$ in the weak sense in $\Omega$, and
\[
\iint_{B(Y_Q, (1-\tau)\delta(Y_Q))} |\nabla u_Q(X)|^2\,\delta(X)\gtrsim_{\tau,\epsilon} \sigma(Q),
\]
where the implicit constants depend on $n$, ADR, ellipticity, $\tau$, and $\epsilon$.

\end{lemma}

\subsection{Auxiliary results}

Much as in Definition \ref{def:corona-basic}, and fixed $Q_0\in\D$ a semi-coherent corona decomposition relative to $Q_0$ with constant $M_0\ge 1$ is a triple $(\B_{Q_0}, \G_{Q_0}, \bbF_{Q_0})$ where $\D_{Q_0}=\G_{Q_0} \sqcup \B_{Q_0}$, the ``good'' collection $\G_{Q_0}$ is further subdivided so that $\G_{Q_0}=\bigsqcup_{\S\in \bbF_{Q_0}} \S$, with each $\S$ being semi-coherent, and the ``bad'' collection $\B_{Q_0}$ and the family of top cubes $\Top(\bbF_{Q_0}):=\{\Top(\S):\S\in\bbF_{Q_0}\}$ satisfy the Carleson packing condition:
\begin{equation*}
\sup_{Q \in \D_{Q_0}}	\Big(\frac1{ \sigma(Q)}\sum_{Q' \in \B_{Q_0}\cap \D_Q} \sigma(Q') + \frac1{ \sigma(Q)}\sum_{Q'\in \Top(\bbF_{Q_0})\cap\D_{Q}} \sigma(Q') \Big)
	\le M_0.
\end{equation*}

\begin{proposition}\label{pro:global} 
Suppose that there exists $M_0\ge 1$ such that for every $Q\in \D$ there is $(\B_{Q}, \G_{Q}, \bbF_{Q})$, a (coherent/semi-coherent) corona decomposition relative to $Q$ with constant $M_0\ge 1$. Then, there exists a (coherent/semi-coherent) corona decomposition $(\B, \G, \bbF)$ with constant $M_0+C$ (where $C$ depends on dimension and the ADR constant) so that it is compatible with the given ones, that is, any coherent/semi-coherent sub-regime $\S\in \bbF$ belongs to $\bbF_{Q}$ for some $Q\in\D$. 
\end{proposition}

\begin{proof}
If $\pom$ is bounded, as observed above we have that $\D=\D_{Q_0}$ for some $Q_0\in\D$ with $\ell(Q_0)\approx \diam(\pom)$. Thus, by hypothesis there exists a (coherent/semi-coherent) corona decomposition,  $(\B_{Q_0}, \G_{Q_0}, \bbF_{Q_0})$, relative to $Q_0$ and with constant $M_0\ge 1$. The desired conclusion follows at once by setting $\G:=\G_{Q_0}$, $\B:=\B_{Q_0}$, and  $\bbF:=\bbF_{Q_0}$.

We are now going to deal with the case $\pom$ unbounded, which requires a bit more work. Fix $x_0 \in \pom$ and some $Q_0 \in \D$ such that $x_0 \in Q_0$ and $\ell(Q_0) = 1$. For $k \geq 0$, call $Q_0^{(k)} \in \D$ the $k$-th dyadic ancestor of $Q_0$, that is $Q_0^{(k)}$ is the unique cube in $\D_{-k}$ containing $Q_0$. 
Recall that there exists $N\ge 1$ so that the number of dyadic children of a given dyadic cube is at most $N+1$. For any $k \geq 0$, adding the null set if needed, we can label the dyadic children of  $Q_0^{(k+1)}$ as  $(Q_0^{(k)})^0, (Q_0^{(k)})^1, \dots, (Q_0^{(k)})^N$ with $(Q_0^{(k)})^0=Q_0^{(k)}$. In particular, $(Q_0^{(k)})^1, \dots, (Q_0^{(k)})^N$ are the dyadic ``siblings'' of $Q_0^{(k)}$, that is, the dyadic children of $Q_0^{(k+1)}$ which are not $Q_0^{(k)}$ itself. 

To simplify the notation, let $\{R_j\}_{j=1}^\infty$ be a enumeration of $\{(Q_0^{(k)})^i: k\ge 0, 1\le i\le N\}$ (in this enumeration we drop the null sets that we may have added). Set also $R_0:=Q_0$. By construction it follows easily that we have a disjoint decomposition 
\begin{equation*}
	\bigcup_{k = 0}^\infty \D_{Q_0^{(k)}} 
	= 
	\Big(\bigsqcup_{j=0}^\infty \D_{R_j}\Big) \bigsqcup \Big(\bigsqcup_{k = 1}^\infty \{Q_0^{(k)}\}\Big).
\end{equation*} 
By hypothesis, for each $R_j$ we let  $(\B_{R_j},\G_{R_j},\bbF_{R_j})$ be the associated (coherent/semi-coherent) corona decomposition. This allows us to define a decomposition into good and bad cubes for the full family of cubes we are considering. Indeed, if we define
\begin{equation}\label{eq:GBQ}
	\G^{Q_0} := \bigsqcup_{j=0}^\infty \G_{R_j}
	\qquad 
	\text{ and }
	\qquad 
	\B^{Q_0} := \Big(\bigsqcup_{j=0}^\infty \B_{R_j}\Big) \bigsqcup \Big(\bigsqcup_{k = 1}^\infty \{Q_0^{(k)}\}\Big),
\end{equation}
we have 
\begin{equation*}
\D^{Q_0}:=\bigcup_{k = 0}^\infty \D_{Q_0^{(k)}} = \G^{Q_0} \sqcup \B^{Q_0}.
\end{equation*}
If we also set
\[
\bbF^{Q_0} := \bigsqcup_{j=0}^\infty \bbF_{R_j}
\]
we have  by construction,
\[
\G^{Q_0}
=
\Big( \bigsqcup_{\S\in \bbF^{Q_0}} \S \Big)
=
\bigsqcup_{j=0}^\infty \Big( \bigsqcup_{\S\in \bbF_{R_j}} \S \Big)
\]
where each $\S$ in the previous unions is coherent/semi-coherent.

Fix next  $Q \in\D^{Q_0}$ and note that if $Q_0^{(k)}\subset Q$ for some $k\ge 1$ then $Q=Q_0^{(k_0)}$ with $k\le k_0$. Hence, 
\[
\sum_{Q_0^{(k)}\subset Q} \sigma(Q_0^{(k)}) 
=
\sum_{k=1}^{k_0}\sigma(Q_0^{(k)}) 
\le
C\, \sigma(Q_0^{(k_0)}) \sum_{k=0}^{k_0} 2^{(k-k_0)n}
\le
C\, \sigma(Q_0^{(k_0)})
=
C\, \sigma(Q)
, 
\]
where $C$ depends on $n$ and the ADR constants.
As a result,  by \eqref{eq:GBQ}, the packing Carleson condition of $\B_{R_j}$ and the top cubes in $\bbF_{R_j}$, and the fact that the family $\{R_j\}_j$ is pairwise disjoint we obtain 
\begin{multline*}
		\sum_{R \in \Top(\bbF^{Q_0})\cap\D_Q } \sigma(R) +	\sum_{R \in \B^{Q_0}\cap \D_{Q}} \sigma(R) 
	\\
	=
	\sum_{j=0}^\infty \Big(\sum_{R \in \Top(\bbF_{R_j})\cap\D_Q } \sigma(R)+\sum_{R \in \B_{R_j}\cap \D_{Q}} \sigma(R) \Big)
	+ \sum_{Q_0^{(k)}\subset Q} \sigma(Q_0^{(k)}) 
	\\
	\le
	M_0 \sum_{j=0}^\infty \sigma(R_j\cap Q) 
	+
	C\, \sigma(Q)
	\le
	(M_0+C)	\,\sigma(Q).
\end{multline*}
We conclude as desired that the bad cubes and the top cubes of the good subregimes satisfy a packing condition with constant $M_0+C$.

To complete the proof we are going to proceed inductively. Suppose that we have already constructed $Q_0, Q_1,\dots,Q_N\in\D$ so that
$\ell(Q_j)=1$ for very $1\le j\le N$; $\{\D^{Q_j}\}_{j=1}^N$ is a pairwise disjoint family (of families of dyadic cubes); and if we set $\pom_j=\bigcup_{k=0}^\infty Q_j^{(k)}$, one has that $\{\pom_j\}_{j=1}^N$ is a pairwise disjoint family of subsets of $\pom$. Let us explain how to take the next step. If $\D = \bigcup_{k=0}^N \D^{Q_k}$ then we stop the construction. Otherwise, we can pick $Q\in \D\setminus\bigcup_{k=0}^N \D^{Q_k}$.

If $\ell(Q)\le 1$ let $Q_{N+1}\supset Q$ be so that $\ell(Q_{N+1}) = 1$ and note that clearly $Q_{N+1}\notin \bigcup_{k=0}^N \D^{Q_k}$. On the other hand, if $\ell(Q)> 1$ we select some $Q_{N+1}\subset Q$ with $\ell(Q_{N+1}) = 1$. We claim that in this case we also have $Q_{N+1}\notin \bigcup_{k=0}^N \D^{Q_k}$. Otherwise, $Q_{N+1}\subset Q_k^{(i)}$ for some $i\ge 0$ and $0 \leq k \leq N$. In particular, $Q\cap Q_k^{(i)}\neq \emptyset $. Then, either $Q\subset  Q_k^{(i)}$, in which case $Q\in\D^{Q_k}$ and we have reached a contradiction;  or
$Q_k^{(i)}\subsetneq Q$, in which case $Q=Q_k^{(i')}$ for some $i'\ge 1$, and again  $Q\in\D^{Q_k}$ leading to a contradiction.
In either scenario we have found $Q_{N+1}\notin \bigcup_{k=0}^N \D^{Q_k}$ with $\ell(Q_{N+1})=1$.  

We next show that $\D^{Q_j}\cap \D^{Q_{N+1}}=\emptyset$ for every $0\le j\le N$. Assume otherwise that there exits $Q\in \D_{Q_j^{(k_0)}}\cap \D_{Q_{N+1}^{(k_1)}}$. In particular, $Q_j^{(k_0)}$ and $Q_{N+1}^{(k_1)}$ meet. This implies that either $Q_j^{(k_0)}\subset Q_{N+1}^{(k_1)}$or 
$Q_{N+1}^{(k_1)}\subsetneq Q_j^{(k_0)}$. In both cases we conclude that $Q_{N+1}\in \D^{Q_j}$, a contradiction.

Set $\pom_{N+1}=\bigcup_{k=0}^\infty Q_{N+1}^{(k)}$ and we claim that $\pom_j\cap\pom_{N+1}=\emptyset$ for every $0\le j\le N$. Indeed, if this were not the case, we could pick $y \in Q_j^{(k_0)}\cap Q_{N+1}^{(k_1)}$ for some $k_0, k_1 \geq 0$. Take then $Q\ni y$ with $\ell(Q)=1$. Then, $Q\subset Q_j^{(k_0)}\cap Q_{N+1}^{(k_1)}$ and hence $Q\in \D^{Q_j}\cap \D^{Q_{N+1}}$, a contradiction.

The previous argument leads then to a (possible infinite) collection dyadic cubes $\{Q_j\}_{j\in \mathfrak{N}}$ such that $\ell(Q_j)=1$ for every $j\in \mathfrak{N}$; $\D^{Q_j}\cap \D^{Q_{j'}}=\emptyset$ for every $j,j'\in \mathfrak{N}$; $\pom_j\cap\pom_{j'}=\emptyset$ for every $j,j'\in \mathfrak{N}$; and $\D=\sqcup_{j\in \mathfrak{N}} \D^{Q_j}$. For every $j\in \mathfrak{N}$ we can then repeat the argument above with $Q_j$ in place of $Q_0$ to write  $\D^{Q_j}=\G^{Q_j}\sqcup \B^{Q_j}$, so that $\G^{Q_j}$ splits further into coherent/semi-coherent subregimes, and  the bad cubes and the top cubes of the good subregimes satisfy a  Carleson  packing condition with constant $M_0+C$. To complete the proof we set $\G:=\bigsqcup_{j\in \mathfrak{N}} \G^{Q_j}$, $\B:=\bigsqcup_{j\in \mathfrak{N}} \B^{Q_j}$, and $\bbF:=\bigsqcup_{j\in \mathfrak{N}} \bbF^{Q_j}$,  and the reader can easily check that $\D=\G\sqcup \B$, so that $\G$ splits further into the semi-coherent subregimes in $\bbF$. Observe that for any $Q\in\D$ there exists a unique cube $Q_j$ with $j\in\mathfrak{N}$ so that $\D_Q\subset \D^{Q_j}$, hence when proving the Carleson packing we have
\begin{multline*}
\sup_{Q\in \D} \frac1{\sigma(Q)} 	\sum_{R\in\Top(\bbF)\cap\D_Q} \sigma(R)+ \frac1{\sigma(Q)} \sum_{R \in \B \cap \D_Q} \sigma(R) 
\\
=
\sup_{j\in\mathfrak{N}}
\sup_{Q\in \D^{Q_j}} \frac1{\sigma(Q)} 	\sum_{R\in\Top(\bbF^{Q_j})\cap\D_Q} \sigma(R)+ \frac1{\sigma(Q)} \sum_{R \in \B^{Q_j} \cap \D_Q} \sigma(R) 
\le
M_0+C.
\end{multline*} 
The fact that any (coherent/semi-coherent) subregime $\S \in \bbF$ belongs to $\G_{Q}$ for some $Q\in\D$ is clear from the construction and this completes the proof. 
\end{proof}

\section{Proof of Theorem \ref{thm:corona}}\label{sec:corona}
This section is devoted to showing Theorem \ref{thm:corona}. We will follow the scheme 
\[
\eqref{list:wL-strong} \Longrightarrow \eqref{list:wL} \Longrightarrow \eqref{list:GL} \Longrightarrow \eqref{list:CME} \Longrightarrow \eqref{list:wL-strong}.
\]  
Since $\eqref{list:wL-strong} \Longrightarrow \eqref{list:wL}$ is trivial, it suffices to show the other implications.

\subsection{Proof of \texorpdfstring{$\eqref{list:wL} \Longrightarrow \eqref{list:GL}$}{(b) implies (c)}: \texorpdfstring{$\omega_L$}{the elliptic measure} admits a corona decomposition implies that \texorpdfstring{$G_L$}{the Green function} is comparable to the distance to the boundary in the corona sense}\label{sec:wL-GL} 
Let $(\B_0, \G_0, \bbF_0)$ be the assumed corona decomposition associated with $\omega_L$. Given $N\ge 0$, large enough to be chosen momentarily, we may proceed as in Remark~\ref{remark:corona-G:CFMS:N-close} and pick $(\B, \G, \bbF)$ the $2^N$-refinement of $(\B_0, \G_0, \bbF_0)$. Recall that this implies that for every $\S\in\bbF$ there exist $Q_\S\in\D$ and $X_\S \in \Omega$ so that 
\[
\Top(\S)\subset \Top(\S)^{(N)}\subset Q_\S,
\qquad  
\delta(X_{\S}) \approx \ell(Q_{\S}) \approx \dist(X_{\S}, Q_{\S}),
\]
and
\begin{equation}\label{qfeawgvev}
\frac{\omega_L^{X_{\S}}(Q_{\S})}{\sigma(Q_{\S})} 
\lesssim 
\frac{\omega_L^{X_{\S}}(Q)}{\sigma(Q)} 
\lesssim 
\frac{\omega_L^{X_{\S}}(2\widetilde{\Delta}_Q)}{\sigma(2\widetilde{\Delta}_Q)} 
\lesssim 
\frac{\omega_L^{X_{\S}}(Q_{\S})}{\sigma(Q_{\S})},
\quad \forall\,Q \in \S(N).
\end{equation}
Fix $\S \in \bbF$ and write
\begin{equation} \label{eq:wG-normalize}
\mu := \sigma(Q_\S) \omega_L^{X_\S} 
\quad\text{ and }\quad 
\GG := \sigma(Q_\S) G_L(X_\S, \cdot). 
\end{equation}
		
\begin{lemma}\label{lem:gradient}
There exists $N\gg 1$, depending on the allowable parameters and on the implicit constants in \eqref{corona-w:CKS-S},
such that if $Q \in \S\in\bbF$, then there is a set $E_Q\subset  \{Y\in 2\widetilde{B}_Q\cap\Omega: \delta(Y)\ge 2^{-N}\ell(Q)\}$ with $|E_Q|>0$ so that 
\begin{equation}\label{eq:GG}
\frac{\GG(Y)}{\delta(Y)} \gtrsim \frac{\mu(Q)}{\sigma(Q)},
\quad\text{for every }Y\in E_Q,
\qquad \text{and}\qquad 
\|\nabla\GG\|_{L^\infty(E_Q)}\gtrsim \frac{\mu(Q)}{\sigma(Q)} .
\end{equation}
\end{lemma}

\begin{proof}
	We use some ideas from \cite[Lemma~4.24]{HLMN}.
Let $Q \in \S$ with and note that since $Q\subset \Top(\S)\subset \Top(\S)^{(N)}\subset Q_\S$ it follows that $\ell(Q) \lesssim 2^{-N}\ell(Q_\S)$. Let $\widetilde{B}_Q$ and $\widetilde{\Delta}_Q$ be as in \eqref{eq:BQ} so that $Q \subset \widetilde{\Delta}_Q $. Choose a cut-off function $\Phi_Q \in \mathscr{C}^\infty_c(\R^{n+1})$ satisfying $\mathbf{1}_{\widetilde{B}_Q} \leq \Phi_Q \leq \mathbf{1}_{\frac54\widetilde{B}_Q}$ and $\|\nabla \Phi_Q\|_{L^{\infty}(\R^{n+1})} \lesssim \ell(Q)^{-1}$. Let $N>4$ be large enough to be chosen and note that by assumption 
\[
2^{N} \ell(Q)
\lesssim
\ell(Q_\S)
\approx
\delta(X_\S) 
\le
|X_\S-x_Q|.
\] 
Thus, taking $N$ large enough one can guarantee that $X_\S \notin 4\widetilde{B}_Q$. Then, $\Phi_Q(X_\S) = 0$, and by \eqref{eq:Riesz}, 
\begin{multline}\label{aVRGVR}
\ell(Q) \mu(Q)
 \leq \ell(Q) \int_{\pom} \Phi_Q \, d\mu 
= - \ell(Q) \iint_\Omega A^{\top} \nabla \GG \cdot \nabla \Phi_Q \, dX
\lesssim \iint_{\frac54\widetilde{B}_Q \cap \Omega} |\nabla \GG| \, dX
\\
= \iint_{\frac54\widetilde{B}_Q \cap \{X\in\Omega:\delta(X) \geq 2^{-N+1}\ell(Q)\}} |\nabla \GG| \, dX
+ \iint_{\frac54\widetilde{B}_Q \cap \{X\in\Omega:\delta(X) \leq 2^{-N+1}\ell(Q)\}} |\nabla \GG| \, dX
=: {\rm I} + {\rm II}.
\end{multline}
Let us first estimate ${\rm II}$, the term closer to the boundary. Picking $N$ big enough, we have 
\begin{align*}
{\rm II} \leq \sum_{\substack{I \in \W \\ I \cap (\frac54\widetilde{B}_Q \cap \{X\in\Omega:\delta(X) \leq 2^{-N+1}\ell(Q)\}) \neq \emptyset}} \iint_I |\nabla \GG|\, dX
\le \sum_{\substack{I \in \W: I \subset \frac32 \widetilde{B}_Q \\ \ell(I) \lesssim 2^{-N} \ell(Q)}} \iint_I |\nabla \GG|\, dX. 
\end{align*}
Given a cube $I$ in the summation above, let $\widehat{x}_I \in \pom$ be such that $\abs{\widehat{x}_I - X(I)} = \delta(X(I))$. Then, assuming that $N$ is large enough,  
\begin{multline*}
\delta(X_\S)
\le
|X_\S -X(I)|+\delta(X(I))
\le
|X_\S -X(I)|+C\,\ell(I)
\le
|X_\S -X(I)|+C\,2^{-N+1}\,\ell(Q)
\\
\le
|X_\S -X(I)|+C\,2^{-2\,N+1}\,\ell(Q_\S)
\le
|X_\S -X(I)|+ 2^{-1}\,\delta(X_\S)
\end{multline*}
and, as a result, 
\[
|X_\S -X(I)|\ge 2^{-1}\delta(X_\S) \approx \ell(Q_\S) \gtrsim 
2^N\ell(Q)
\gtrsim
4^{N}\, \ell(I)
\approx 
4^{N}\, \delta(X(I)). 
\]
By Caccioppoli's and Harnack's inequalities, and Lemma~\ref{lem:Bourgain-CFMS} we obtain 
\begin{multline*}
{\rm II}  \leq 
\sum_{\substack{I \in \W: I \subset \frac32 \widetilde{B}_Q \\ \ell(I) \lesssim 2^{-N} \ell(Q)}} \abs{I} \bigg(\bariint_I |\nabla \GG|^2\, dX\bigg)^{\frac12} 
\lesssim \sum_{\substack{I \in \W: I \subset \frac32 \widetilde{B}_Q \\ \ell(I) \lesssim 2^{-N} \ell(Q)}}  \ell(I)^n \bigg(\bariint_{I^*} \abs{\GG}^2 \, dX\bigg)^{\frac12}
\\  \approx 
\sum_{\substack{I \in \W: I \subset \frac32 \widetilde{B}_Q \\ \ell(I) \lesssim 2^{-N} \ell(Q)}} \ell(I)^n \GG(X(I)) 
\lesssim
\sum_{\substack{I \in \W: I \subset \frac32 \widetilde{B}_Q \\ \ell(I) \lesssim 2^{-N} \ell(Q)}} \ell(I) \,  \mu(\Delta(\widehat{x}_I, 2\delta(X(I)))). 
\end{multline*}
To proceed, we observe that the family $\{\Delta(\widehat{x}_I, 2\delta(X(I)))): I \in \W, \ell(I) = 2^{-k} \ell(Q)\}$ has bounded overlap, and that each such surface ball is contained in $2\widetilde{\Delta}_Q$ provided $N$ is large enough, $I\subset\frac32\widetilde{B}_Q$, and $k\gtrsim N$. Thus, it follows from the corona for $\omega_L$ (more precisely from \eqref{qfeawgvev} and the fact that $\S\subset\S(N)$) and the ADR property that 
\begin{multline*}
{\rm II}  \lesssim 
\sum_{k \gtrsim N} 2^{-k} \ell(Q) \sum_{\substack{I \in \W: I \subset \frac32 \widetilde{B}_Q \\ \ell(I) = 2^{-k} \ell(Q)}} \mu(\Delta(\widehat{x}_I, 2\delta(X(I)))) 
\\  
\lesssim 
\sum_{k \gtrsim N} 2^{-k} \ell(Q) \, \mu \bigg( \bigcup_{\substack{I \in \W: I \subset \frac32 \widetilde{B}_Q \\ \ell(I) = 2^{-k} \ell(Q)}} \Delta(\widehat{x}_I, 2\delta(X(I)))) \bigg) 
\\ 
 \leq 
 \sum_{k \gtrsim N} 2^{-k} \ell(Q)  \, \mu (2\widetilde{\Delta}_Q) 
\lesssim 
2^{-N} \ell(Q) \, \mu (2\widetilde{\Delta}_Q) 
\le  \frac12 \ell(Q) \mu(Q),  
\end{multline*}
provided $N$ is large enough. As a result, in \eqref{aVRGVR} we can hide ${\rm II}$ to arrive at
\[
\frac12 \ell(Q) \mu(Q) \leq {\rm I} 
= \iint_{\frac54\widetilde{B}_Q \cap \{X\in\Omega:\delta(X) \geq 2^{-N+1}\ell(Q)\}} |\nabla \GG| \, dX.
\]
Cover $\frac54\widetilde{B}_Q $ with  a family of balls $\{B_k\}_{k=1}^K$ with $B_k=B(X_k,2^{-N-1}\,\ell(Q))$, $X_k\in\widetilde{B}_Q$,  and where $K$ is uniformly bounded and depends on $n$, ADR, and $N$. Then, 
\begin{multline*}
\frac12 \ell(Q) \mu(Q) 
\leq	
{\rm I} 
	\le
	\sum_{k=1}^K \iint_{B_k \cap \{X\in\Omega:\delta(X) \geq 2^{-N+1}\ell(Q)\}} |\nabla \GG| \, dX
	\\
	\lesssim_N
	\max_{1\le k\le K} \iint_{B_k \cap \{X\in\Omega:\delta(X) \geq 2^{-N+1}\ell(Q)\}} |\nabla \GG| \, dX.
	\end{multline*}
Take $1\le k_Q\le K$ such that the maximum is attained.  Set $E_Q:=B_{k_Q} \cap \{X\in\Omega:\delta(X) \geq 2^{-N+1}\ell(Q)\}$ and note that 
\[
\frac12 \ell(Q) \mu(Q) 
\lesssim_N
\iint_{E_{Q}} |\nabla \GG| \, dX
\lesssim
(2^{-N}\,\ell(Q))^{n+1}\|\nabla \GG\|_{L^\infty(E_Q)}
\lesssim \ell(Q)\sigma(Q)\|\nabla \GG\|_{L^\infty(E_Q)}.
\]
This gives that $|E_Q|>0$ and the second estimate in \eqref{eq:GG}. Moreover, for any $Y\in E_Q$ 
\[
 \delta(Y)\le |Y-x_Q|
\le
|Y-X_{k_Q}|+ |X_{k_Q}-x_Q|
<
2^{-N-1}\ell(Q)+ \frac54\,r(\widetilde{B}_Q)
<
\frac34\Xi r_Q+\frac54\,r(\widetilde{B}_Q)
=
2r(\widetilde{B}_Q).
\]
Consequently, $Y\in 2\widetilde{B}_Q\cap\Omega$ and $2^{-N+1}\ell(Q)\le \delta(Y)\lesssim \ell(Q)$. Note also that if $X\in E_Q$, then 
\[
|X-Y|<2^{-N}\ell(Q)\le \delta(Y)/2.
\]
As a result, 
\begin{multline}\label{AWfvawvqaevq}
	\ell(Q) \mu(Q) 
	\lesssim_N
	\iint_{E_Q} |\nabla \GG| \, dX
	\lesssim_N
	\ell(Q)^{\frac{n+1}{2}}\Big(\iint_{B(Y,\delta(Y)/2)} |\nabla \GG|^2 \, dX \Big)^\frac12
	\\
	\lesssim
	\ell(Q)^{\frac{n-1}{2}}\Big(\iint_{B(Y,3\delta(Y)/4)} |\GG|^2 \, dX \Big)^\frac12
	\approx
	\sigma(Q) \GG(Y),
\end{multline}
where we have used Caccioppoli's and Harnack's inequalities. This readily leads to the desired estimate completing the proof. 
\end{proof}

\begin{remark}\label{remark:gradient:Norefinment} 
We need to make the following observation which  will be used below to obtain Theorem \ref{thm:UR}. In the previous proof the fact that $(\B, \G, \bbF)$ is the $2^N$-refinement of $(\B_0, \G_0, \bbF_0) $ has been only used at the very beginning to make sure that $\ell(Q)\lesssim 2^{-N}\ell(Q_\S)$. This means that one can state a version of Lemma~\ref{lem:gradient} for the original corona decomposition $(\B_0, \G_0, \bbF_0)$ as follows: Fixed $\S\in\bbF_0$ and defining $\mu$ and $\GG$ as in \eqref{eq:wG-normalize}, there exists $N\gg 1$, depending on the allowable parameters and on the implicit constants in \eqref{corona-w:CKS-S}, such that if $Q \in \S\in\bbF$ with $\ell(Q)\le 2^{-N}\,\ell(\Top(S))$, then \eqref{eq:GG} holds. Details are left to the interested reader.
\end{remark}

We are now ready to continue with the proof. Apply Lemma~\ref{lem:gradient} with $N$ 
large enough. For $Q\in\S$, we have $|E_Q|>0$, thus there exists $Y_Q\in E_Q$. On account of the normalization introduced in \eqref{eq:wG-normalize}, we then obtain 
\begin{multline*}
\sup_{\substack{ X \in 2\widetilde{B}_{Q}\cap\Omega \\ \delta(X) \ge 2^{-N}\ell(Q)}}\frac{G_L(X_{\S}, X)}{\delta(X)} 
\ge
\frac{G_L(X_{\S}, Y_Q)}{\delta(Y_Q)} 
=
\frac1{\sigma(Q_\S)}\frac{\GG(Y_Q)}{\delta(Y_Q)} 
\\
\gtrsim
\frac1{\sigma(Q_\S)}\frac{\mu(Q)}{\sigma(Q)} 
=
\frac{\omega_L^{X_{\S}}(Q)}{\sigma(Q)}
\approx
\frac{\omega_L^{X_{\S}}(Q_\S)}{\sigma(Q_\S)},
\end{multline*}
We next establish the converse inequality. Fix $X \in 2\widetilde{B}_{Q}\cap\Omega$ with $\delta(X) \gtrsim \ell(Q)$. 
Note that by \eqref{eq:BQ} 
\[
\delta(X)
\le 
2\,\Xi\,r_Q
\le
2\,\Xi\,\ell(Q)
\lesssim
2^{-N}\,\Xi\,\ell(Q_\S) 
\approx
2^{-N}\,\Xi\,\delta(X_\S)
<
2^{-1}\,\delta(X_\S),
\]
provided $N$ is large enough, and 
\begin{equation*}
	2\delta(X)
	<
	\delta(X_\S)
	\le
	|X_\S-X|+\delta(X).
\end{equation*}
Hence, $|X_\S-X|>\delta(X)$. Invoking then \eqref{eq:CFMS:1}, 
\begin{equation}\label{eq:Gd}
	\frac{G_L(X_{\S}, X)}{\delta(X)}
	\lesssim \frac{\omega_L^{X_\S}(\Delta(\widehat{x}, 2\,\delta(X)))}{\delta(X)^n},
\end{equation}
where $\widehat{x} \in \pom$ is such that $|X-\widehat{x}| = \delta(X)$. 
For any $y \in \Delta(\widehat{x}, 2\,\delta(X))$ we observe that 
\begin{equation*}
	|y-x_Q| 
	\leq 
	|y-\widehat{x}| + |\widehat{x} - X| + |X-x_Q| 
	\le 
	3\delta(X) + |X-x_Q| 
	\leq 
	4|X-x_{Q}| 
	\le 
	8\Xi r_{Q},  
\end{equation*}
hence $\Delta(\widehat{x}, 2\,\delta(X))\subset 8\widetilde{\Delta}_{Q}$. Also, if $z \in 8\widetilde{\Delta}_{Q}$, recalling that we write $Q^{(N)}$ for the $N$-th dyadic ancestor of $Q$ (that is the unique dyadic cube containing $Q$ and with sidelength $2^{N}\ell(Q)$), we have
\begin{align*}
	|z-x_{Q^{(N)}}| 
	\leq 
	|z-x_{Q}| + |x_{Q}-x_{Q^{(N)}}| 
	\le 8\Xi r_{Q}+ \Xi\,r_{Q^{(N)}}
	< 2\,\Xi\,r_{Q^{(N)}},  
\end{align*}
provided $N \ge 3$, since $r_{Q}=2^{-N}r_{Q^{(N)}}$. Altogether, 
\begin{align}\label{eq:MDQ}
	\Delta(\widehat{x}, 2\delta(X)) \subset 8\widetilde{\Delta}_{Q} \subset 2\widetilde{\Delta}_{Q^{(N)}}. 
\end{align}
Note that $Q\in\S$ and $Q^{(N)}$ are $2^N$-close, hence $Q^{(N)}\in \S(N)$. Thus, we can invoke  \eqref{eq:Gd}, \eqref{eq:MDQ}, and \eqref{qfeawgvev} to conclude that
\begin{equation*}
	\frac{G_L(X_\S, X)}{\delta(X)}
	\lesssim 
	\frac{\omega_L^{X_\S}(2\widetilde{\Delta}_{Q^{(N)}})}{\delta(X)^n}
	\approx 
	\frac{\omega_L^{X_\S}(2\widetilde{\Delta}_{Q^{(N)}})}{\sigma(2\widetilde{\Delta}_{Q^{(N)}})}
	\lesssim 
	\frac{\omega_L^{X_\S}(Q_\S)}{\sigma(Q_\S)}.  
\end{equation*}
This completes the proof of the current implication.\qed

\subsection{Proof of \texorpdfstring{$\eqref{list:GL} \Longrightarrow \eqref{list:CME}$}{(c) implies (d)}: \texorpdfstring{$G_L$}{the Green function} is comparable to the distance to the boundary in the corona sense  implies that \texorpdfstring{$L$}{L} satisfies partial/weak Carleson measure estimates}\label{sec:G-CME}
Let $u \in L^{\infty}(\Omega)$ be a non-trivial weak solution of $Lu=-\div(A\nabla u)=0$ in $\Omega$. By homogeneity, we may assume that $\|u\|_{L^{\infty}(\Omega)}=1$. Fix also $\tau\in (0,\frac12)$, the parameter which appears in the partial/weak Carleson measure estimate (see Definition~\ref{def:Carleson}).  Assume that $(\B_0, \G_0, \bbF_0)$ is the assumed corona decomposition associated with $G_L$. Let $N\ge 0$ be large enough to be chosen momentarily (depending on $\tau$). We may proceed as in Remark~\ref{remark:corona-G:CFMS:N-close} and pick $(\B, \G, \bbF)$ the  $2^N$-refinement of $(\B_0, \G_0, \bbF_0)$. This implies that for every $\S\in\bbF$ there exist $Q_\S\in\D$ and $X_\S \in \Omega$ so that 
\begin{equation}\label{se5w4g4g}
\Top(\S)\subset \Top(\S)^{(N)}\subset Q_\S,
\qquad  
	\delta(X_{\S}) \ge 4\Xi \ell(Q_\S), \qquad \dist(X_{\S}, Q_\S) \lesssim \ell(Q_\S),
	\end{equation}
and
\begin{equation}\label{prwgbsbs}
	\sup_{\substack{X \in 2\,\widetilde{B}_Q\cap\Omega \\ \delta(X) \ge c_0\,\ell(Q)}} 
	\frac{G_L(X_{\S}, X)}{\delta(X)} \approx \frac{\omega_L^{X_{\S}}(Q_\S)}{\sigma(Q_\S)}=:\Lambda_\S, \quad\forall Q \in \S(N),
\end{equation}
for some $c_0\in (0,\frac12)$. Writing $G_{\S}:=G_L(X_{\S}, \cdot)$, by \eqref{prwgbsbs} and Harnack's inequality, for each $Q\in\S(N)$ there exists $P_Q \in  2\,\widetilde{B}_Q\cap\Omega$ with $\delta(P_Q) \ge c_0 \ell(Q)$ such that
\begin{align}\label{eq:Lambda-S}
\Lambda_{\S} \lesssim \frac{G_{\S}(P_Q)}{\delta(P_Q)} \approx_\tau \frac{G_{\S}(X)}{\delta(X)}, 
\quad\forall X \in V_Q := B(P_Q, (1-\tau)\delta(P_Q)). 
\end{align} 
Set 
\[
\alpha_Q := \iint_{V_Q} |\nabla u(X)|^2 \delta(X) \, dX, \quad\, Q \in \D, 
\]
and note that by Caccioppoli's inequality
\begin{multline}\label{qafvavwe}
	\alpha_Q
\lesssim
\delta(P_Q)\,  \iint_{V_Q} |\nabla u(X)|^2 dX
\lesssim_\tau
\delta(P_Q)^{-1} \iint_{B(P_Q, (1-\tau/2)\delta(P_Q))} |u(X)|^2 dX
\\
\lesssim
\delta(P_Q)^{n} 
\approx
\ell(Q)^n
\approx
\sigma(Q).
\end{multline}

We first claim that 
\begin{equation}\label{claim:CME}
\sup_{\S\in\bbF, Q\in\S} \frac1{\sigma(Q)}\sum_{Q' \in \S \cap \D_{Q}} \alpha_{Q'}
	\lesssim
	1.
\end{equation}
Assuming this momentarily let us continue as follows. Introduce some notation: if $Q_0\in\G$ we define $\S_{Q_0}$ to be the unique $\S\in\bbF$ so that $Q_0\in\S$, otherwise, if $Q_0\in\B$ then we set  $\S_{Q_0}=\emptyset$. We first see that 
\begin{equation}\label{xcndrnbr}
\D_{Q_0}\cap\G=\Big(\bigcup_{\S\in\bbF: \Top(\S)\subsetneq Q_0} \S\Big) \bigcup \big(\S_{Q_0}\cap \D_{Q_0}\big).
\end{equation}
Indeed, let $Q\in\D_{Q_0}\cap\G$, then $Q\in \S$ for some $\S\in\bbF$. In particular $Q\subset\Top(\S)\cap Q_0$. If $\Top(\S)\subsetneq Q_0$, then $Q$ in the first set of the right hand side of the last expression. Otherwise, if $Q_0\subset\Top(\S)$, then $\S\ni Q\subset Q_0\subset\Top(\S)$ and the semi-coherency of $\S$ implies that $Q_0\in\S\subset\G$, hence, necessarily $\S=\S_{Q_0}$ and $Q\in \S_{Q_0}\cap \D_{Q_0}$. Once \eqref{xcndrnbr} has been shown, we see that  \eqref{qafvavwe} and \eqref{claim:CME}  yield
\begin{align}\label{eq:bad-good}
	\sum_{Q \in \D_{Q_0}} \alpha_Q
	&=
	\sum_{Q \in \B \cap \D_{Q_0}} \alpha_Q 
	+
	\sum_{Q \in \G \cap \D_{Q_0}} \alpha_Q 
	\\
	&\le
	\sum_{Q \in \B \cap \D_{Q_0}} \alpha_Q 
	+
	\sum_{\S\in \bbF: \Top(\S)\subsetneq Q_0}
	\sum_{Q \in \S } \alpha_Q
	+
	\sum_{Q\in \S_{Q_0}\cap \D_{Q_0}} \alpha_Q
	\\
	&\lesssim
	\sum_{Q \in \B \cap \D_{Q_0}}\sigma(Q)
	+
	\sum_{\S\in \bbF: \Top(\S)\subsetneq Q_0} \sigma(\Top(\S))
	+
	\sigma(Q_0)
	\\
	&\le
	\sum_{Q \in \B \cap \D_{Q_0}} \sigma(Q)
	+
	\sum_{Q \in \Top(\bbF) \cap \D_{Q_0}} \sigma(Q)
	+
	\sigma(Q_0)
	\\
	&\lesssim
	\sigma(Q_0),
	\end{align}
where in the last estimate we have used that the families  $\B$ and $\Top(\bbF)$ satisfy a Carleson packing condition. Since $Q_0\in\D$ is arbitrary this allows us to obtain our desired estimate modulo our claim \eqref{claim:CME} which we prove next.

To see \eqref{claim:CME} we fix $\S\in\bbF$ and $Q_0\in\S$, in particular  $\S\cap\D_{Q_0}\neq\emptyset$. Set $\S':=\S\cap\D_{Q_0}$ which is clearly semi-coherent with $\Top(\S')=Q_0$. Introduce $\F$, the family of maximal dyadic cubes (hence, pairwise disjoint) in $\D_{Q_0}\setminus \S'$. We claim that $\D_{\F, Q_0}=\S'$. Indeed, if $Q\in \D_{\F, Q_0}\setminus \S'$ then $Q\subset Q'$ for some $Q'\in \F$, a contradiction.  On the other hand, if $Q\in\S'\setminus \D_{\F, Q_0}$, then $Q\subset Q'\subset Q_0$ for some $Q'\in\F$. Noting that $Q\in\S'$ with $\S'$ being semi-coherent, we conclude that $Q'\in\S'$ which leads again to a contradiction. To continue we  observe that for every $Q\in\S'\subset \S\subset\S(N)$ we have by \eqref{eq:Lambda-S}
\[
\delta(X)\lesssim_\tau \Lambda_\S^{-1} G_\S(X), \qquad\forall\,X\in V_Q.
\]
This and Lemma~\ref{lemma:VQ-overlap}  allow us to deduce that 
\begin{multline}\label{eq:QSQ}
\sum_{Q \in \S' } \alpha_Q 
=
\sum_{Q \in \D_{\F,Q_0}} \alpha_Q 
= 
\sum_{Q \in \D_{\F,Q_0}} \iint_{V_Q} |\nabla u|^2 \, \delta \, dX
\\
\lesssim 
\Lambda_{\S}^{-1} \sum_{Q \in \D_{\F, Q_0}} \iint_{V_Q} |\nabla u|^2 \, G_{\S} \, dX 
\lesssim 
\Lambda_{\S}^{-1} \iint_{\Omega_{\F, Q_0}^\vartheta} |\nabla u|^2 \, G_{\S} \, dX, 
\end{multline}
where $\vartheta$ depend son the allowable constants but it is independent of $N$.
Observe that by \eqref{eq:kappa} and \eqref{eq:BQ} one has 
\[
\Omega_{\F, Q_0}^\vartheta\subset T_{Q_0}^{\vartheta,**}\subset \frac12 B_{Q_0}^*\cap\Omega \subset \{X\in\Omega:\delta(X)\le \kappa_0\Xi \ell(Q_0)\}.
\]
On the other hand, from  \eqref{se5w4g4g}, and the fact that $Q_0\in \S$ we can readily see that 
\[
\delta(X_\S)\approx \ell(Q_S)\gtrsim 2^N \ell(Q_0)\gg \kappa_0\Xi \ell(Q_0),
\]
by choosing $N$ large enough (depending eventually on $\tau$). This means that $X_\S$ is far away from $B_{Q_0}^*$ and, in  what comes after we never get to worry about the position of the pole of the Green function, $X_\S$, since it is always far from where the integrations take place. We warn the reader that, henceforth, will make use of this observation repeatedly without explicitly mentioning it.  

To proceed, let $I\in \W_Q^\vartheta$ with $Q\in\D_{\F, Q_0}=\S\cap\D_{Q_0}$. Caccioppoli's and Harnack's inequalities yield
\begin{multline}\label{aeevave}
	\iint_{I^{**}}  \big(|\nabla G_{\S}|+ |\nabla u|\,G_{\S}\big)\, dX 
	\\
	\lesssim
	\ell(I)^{\frac{n-1}{2}} \Big(\iint_{I^{***}} |G_{\S}|^2\,dX\Big)^{\frac12} 
	+
	\ell(I)^{-1}\Big(\iint_{I^{***}} |u|^2\,dX\Big)^{\frac12}\Big(\iint_{I^{**}} |G_{\S}|^2\,dX\Big)^{\frac12}
	\\
	\lesssim
	\ell(I)^n\,G_\S(X(I)).
\end{multline}
Let $x_I\in\pom$ be such that $|X(I)-x_I|=\delta(X(I))$ and pick $Q_I\in\D$ satisfying $Q_I\ni x_I$ and $\delta(X(I))\le \ell(Q_I) < 2\delta(X(I))$. Note that \eqref{eq:Q-DQ}--\eqref{eq:BQ}  give
\[
|X(I)-x_{Q_I}|
\le
|X(I)-x_I|+|x_I-x_{Q_I}|
<
\delta(X(I))+\Xi\,r_{Q_I}
\le 
\ell(Q_I)+\Xi\,r_{Q_I}
\le 
2\Xi\,r_{Q_I},
\]
hence $X(I)\in 2\widetilde{B}_{Q_I}$. Besides,
\[
\ell(Q_I)\approx \delta(X(I))\approx \ell(I)\approx_\vartheta \ell(Q)
\]
and
\[
\dist(Q,Q_I)
\le
\dist(Q,I)+\diam(I)+|X(I)-x_I|
\lesssim_\vartheta 
\ell(Q)+\ell(Q_I).
\]
By choosing (and fixing) $N$ large enough (depending on $\vartheta$) we therefore obtain that $Q_I$ and $Q\in\S$ are $2^N$-close, hence $Q_I\in \S(N)$. Using that $c_0<\frac12$, we can next invoke \eqref{prwgbsbs} to obtain that
\begin{equation}\label{fgergerg}
\frac{G_\S(X(I))}{\delta(X(I))} 
\le
\sup_{\substack{X \in 2\,\widetilde{B}_{Q_I}\cap\Omega \\ \delta(X) \ge \frac12\,\ell(Q_I)}} 
\frac{G_L(X_{\S}, X)}{\delta(X)} 
\le
\sup_{\substack{X \in 2\,\widetilde{B}_{Q_I}\cap\Omega \\ \delta(X) \ge c_0\,\ell(Q_I)}} 
\frac{G_L(X_{\S}, X)}{\delta(X)} 
\approx
\Lambda_\S.
\end{equation}
This together with \eqref{aeevave} leads to 
\begin{equation}\label{agergaegvbe}
	\iint_{I^{**}}  \big(|\nabla G_{\S}|+ |\nabla u|\,G_{\S}\big)\, dX 
\lesssim
\Lambda_\S\,\ell(I)^{n+1},
\qquad\text{for every $I\in \W_Q^\vartheta$ with $Q\in\D_{\F, Q_0}$}.
\end{equation}

To continue, for every $M \geq 1$, we consider the pairwise disjoint collection $\F_M$  given by the family of maximal cubes of the collection $\F$ augmented by adding all the cubes $Q \in \D_{Q_0}$ such that $\ell(Q) \leq 2^{-M} \ell(Q_0)$. In particular, $Q \in \D_{\F_M, Q_0}$ if and only if $Q \in \D_{\F, Q_0}$ and $\ell(Q)>2^{-M}\ell(Q_0)$.  Moreover, $\D_{\F_M, Q_0} \subset \D_{\F_{M'}, Q_0}$ for all $M \leq M'$, and hence $\Omega_{\F_M, Q_0}^\vartheta \subset \Omega_{\F_{M'}, Q_0}^\vartheta \subset \Omega_{\F, Q_0}^\vartheta$.  Then the monotone convergence theorem implies
\begin{align}\label{eq:KN-lim}
\iint_{\Omega_{\F, Q_0}^\vartheta} |\nabla u|^2 \, G_{\S} \, dX
=\lim_{M \to \infty} \iint_{\Omega_{\F_M, Q_0}^\vartheta} |\nabla u|^2 \, G_{\S} \, dX 
=: \lim_{M \to \infty}	\mathcal{K}_M.
\end{align}
We claim that 
\begin{align}\label{eq:claim}
\mathcal{K}_M \lesssim \Lambda_{\S} \, \sigma(Q_0),  
\end{align}
where the implicit constant is independent of $M$, $\S$, and $Q_0$. Assuming that \eqref{eq:claim} holds momentarily, we obtain at once that \eqref{eq:QSQ}, \eqref{eq:KN-lim} and \eqref{eq:claim} give
\[
\sum_{Q \in \S' } \alpha_Q 
\lesssim
\sigma(Q_0).
\] 
This justifies as desired \eqref{claim:CME} and we eventually shown the desired partial CME estimates. 

The rest of this section is devoted to proving \eqref{eq:claim} and we borrow some ideas from \cite{CHMT}. Pick $\Psi_M$ from Lemma~\ref{lem:approx} and use  Leibniz's rule to arrive at 
\begin{multline}\label{eq:Leib}
A \nabla u \cdot \nabla u \  G_{\S} \Psi_M^2 
=A \nabla u \cdot \nabla (u G_{\S} \Psi_M^2) 
- \frac12 A \nabla (u^2 \Psi_M^2) \cdot \nabla G_{\S} 
\\ 	
+ \frac12 A \nabla(\Psi_M^2) \cdot \nabla G_{\S} \ u^2
-\frac12 A \nabla (u^2) \cdot \nabla(\Psi_M^2) G_{\S}. 
\end{multline}
Note that $u \in W^{1,2}_{\loc}(\Omega)\cap L^\infty(\Omega)$, $G_{\S} \in W^{1,2}_{\loc}(\Omega \setminus \{X_{\S}\})$, and that $\overline{\Omega_{\F_M, Q_0}^{\vartheta,**}}$ is a compact subset of $\Omega$ away from $X_{\S}$. Hence, $u \in W^{1,2}(\Omega_{\F_M, Q_0}^{\vartheta,**})$ and $u G_{\S} \Psi_M^2 \in W_0^{1,2}(\Omega_{\F_M, Q_0}^{\vartheta,**})$. These, together with the fact that $Lu=0$ in the weak sense in $\Omega$, lead to
\begin{align}\label{eq:AA-1}
\iint_{\Omega} A \nabla u \cdot \nabla(u G_{\S} \Psi_M^2) dX
=\iint_{\Omega_{\F_M, Q_0}^{\vartheta,**}} A \nabla u \cdot \nabla(u G_{\S} \Psi_M^2) dX=0.
\end{align}
On the other hand, $G_{\S} \in W^{1,2}(\Omega_{\F_M, Q_0}^{\vartheta,**})$ and $L^{\top} G_{\S}=0$ in the weak sense in $\Omega \setminus\{X_{\S}\}$. Thanks to the fact that $u^2 \Psi_M^2 \in W^{1,2}_0(\Omega_{\F_M, Q_0}^{**})$, we then obtain
\begin{align}\label{eq:AA-2}
\iint_{\Omega} A \nabla (u^2 \Psi_M^2) \cdot \nabla G_{\S} \, dX
=\iint_{\Omega_{\F_M, Q_0}^{\vartheta,**}} A^{\top} \nabla G_{\S} \cdot \nabla(u^2 \Psi_M^2)\, dX=0.
\end{align}
By Lemma~\ref{lem:approx}, the ellipticity of $A$, \eqref{eq:Leib}--\eqref{eq:AA-2}, the fact that $\|u\|_{L^{\infty}(\Omega)}=1$, and \eqref{agergaegvbe} we arrive at
\begin{multline}\label{eq:KN}
\mathcal{K}_M 
\lesssim \iint_{\Omega} A \nabla u \cdot \nabla u \, G_{\S} \Psi_M^2\, dX
\lesssim \iint_{\Omega} \big(|\nabla G_{\S}| +|\nabla u|\, \, G_{\S}\big) \,|\nabla\Psi_M|\,  dX
\\
\lesssim
\sum_{I \in \W_M^{\vartheta,\Sigma}} \ell(I)^{-1}\iint_{I^{**}}  \big(|\nabla G_{\S}|+ |\nabla u|\,G_{\S}\big)\, dX 
\lesssim
\Lambda_\S\sum_{I \in \W_M^{\vartheta,\Sigma}} \ell(I)^{n}.
\end{multline}
where we have used that if $I \in \W_M^{\vartheta,\Sigma}$ then $I\in \W_Q^\vartheta$ with $Q\in\D_{\F_M, Q_0}\subset \D_{\F, Q_0}$. We use again Lemma~\ref{lem:approx} to observe that 
\begin{equation}\label{34qt3tgg3}
	\sum_{I \in \W_M^{\vartheta,\Sigma}} \ell(I)^{n}
\approx
\sum_{I \in \W_M^{\vartheta,\Sigma}} 
\sigma(\widehat{Q}_I)
\lesssim \sigma\bigg(\bigcup_{I \in \W_M^{\vartheta,\Sigma}} \widehat{Q}_I \bigg)
\le \sigma(C \Delta_{Q_0}) 
\approx 
\sigma(Q_0), 
\end{equation}
where in the next-to-last inequality we have used that $\widehat{Q}_I \subset C\Delta_{Q_0}$ for every $I\in \W_M^{\vartheta,\Sigma}$. Indeed by  Lemma~\ref{lem:approx} we have that if $x\in\widehat{Q}_I$ where $I\in \W_M^{\vartheta,\Sigma}$ and we let $Q\in\D_{\F_M,Q_0}$ be so that $I\in\W_{Q}^\vartheta$, then
\begin{multline*}
	|x-x_{Q_0}|
\le
\diam(\widehat{Q}_I)+\dist(\widehat{Q}_I,I)+\diam(I)+\dist(I,Q)+ \diam(Q_0)
\\
\lesssim
\ell(I)+\ell(Q)+\ell(Q_0)
\approx
\ell(Q)+\ell(Q_0)
\lesssim
\ell(Q_0).
\end{multline*}
Collecting \eqref{eq:KN} and \eqref{34qt3tgg3} we conclude as desired \eqref{eq:claim}, and the proof is then complete.\qed

\subsection{Proof of \texorpdfstring{$\eqref{list:CME} \Longrightarrow \eqref{list:wL-strong}$}{(d) implies (a)}: \texorpdfstring{$L$}{L} satisfies partial/weak Carleson measure estimates implies that \texorpdfstring{$\omega_L$}{the elliptic measure} admits a strong corona decomposition}\label{sec:CME-wL}

We introduce some notation, given $N\ge 1$, for any $Q \in \D$, we let $Q(N) \in \D_Q$ be the unique dyadic cube such that $Q(N)\ni x_Q$ and $\ell(Q(N)) =2^{-N} \ell(Q)$.

We begin with some auxiliary result which will be iterated to construct the desired corona decomposition:

\begin{proposition}\label{prop:corona:basic-Q}
Let $0<\tau<\tau_0$ (cf.~Lemma~\ref{lemma:AGMT}) be the parameter implicitly assumed in \eqref{list:CME}. There exists $N_\tau$ (depending on $n$, ADR, ellipticity, and $\tau$) such that for every $Q\in\D$ and for each $N\gg N_\tau$, if we set $Y_Q=P_{Q(N_\tau)}$  (where $P_{Q(N_\tau)}$ is the point associated with $Q(N_\tau)$ in \eqref{list:CME}),  one can then find possibly empty families of pairwise disjoint cubes $\F_Q=\F_Q^+\sqcup \F_Q^-\subset\D_Q\setminus\{Q\}$,  a Borel function $f_Q$, and $u_Q\in W^{1,2}_{\loc} (\Omega)$ with $L u_Q=0$ in the weak sense in $\Omega$, such that the following hold: 
\begin{align}\label{eq:YQ-pos}
\delta(Y_Q)\approx\ell(Q)\approx\dist(Y_Q,Q);
\qquad \omega_L^{Y_Q}(Q)\approx 1;
\end{align}
\begin{align}\label{eq:NN}
2^{-N}\frac{\omega_L^{Y_Q}(Q)}{\sigma(Q)} 
\lesssim 
\frac{\omega_L^{Y_Q}(Q')}{\sigma(Q')} 
\le \bigg(\fint_{Q'} (\mathcal{M} \omega_L^{Y_Q})^{\frac12} \, d\sigma \bigg)^2 
\lesssim 
2^{2N} \frac{\omega_L^{Y_Q}(Q)}{\sigma(Q)}, 
\qquad \forall\,Q' \in \D_{\F_Q, Q}.
\end{align}
Moreover, if we set $F_Q^{\pm}:=\bigcup_{Q' \in \F_Q^\pm} Q'$, 
then
\begin{align}\label{eq:FQ2}
\sigma(F_Q^+) \le 
2^{-N}\,\sigma(Q);
\end{align}
\begin{align}\label{eq:uQfQ}
	0\le f_Q\lesssim  \mathbf{1}_{Q\setminus F_Q^-}, \qquad
	u_Q(X):=\int_{Q\setminus F_Q^-} f_Q\,d\omega_L^X, \quad X\in\Omega, 
\end{align}
and
\begin{align}\label{eq:uQ-lower} 
	\sigma(Q)\lesssim \iint_{B(Y_Q, (1-\tau)\delta(Y_Q))} |\nabla u_Q(X)|^2\,\delta(X).
\end{align}
In the previous estimates the implicit constants depend on $n$, ADR, ellipticity, and $\tau$, but they do not depend on $N$.
\end{proposition}

\begin{proof}

Write $C_0>1$ for the constant in Bourgain's estimate (see \eqref{Bourgain}). Let  $N_\tau$ be large enough to be chosen momentarily so that $2^{N_\tau} \gg C_0$.  Let $Y_Q:=P_{Q(N_\tau)}$ be the point associated with $Q(N_\tau)$ in \eqref{list:CME}. Observe that since $x_Q\in Q(N_\tau)$ we have
\begin{equation}\label{afawefwf}
	|Y_Q-x_Q|
\le
\dist(Y_Q, Q(N_\tau))+ \diam(Q(N_\tau))
\approx
\ell(Q(N_\tau))
=
2^{-N_\tau}\ell(Q)
\ll \ell(Q).
\end{equation}
Hence, by taking $N_\tau$ large enough, Bourgain's estimate (see Lemma~\ref{lem:Bourgain-CFMS}) and \eqref{eq:Q-DQ} imply
\begin{align}\label{eq:wpQQ}
C_0^{-1}\le \omega_L^{Y_Q}(Q) \le 1.
\end{align}  
Note that by \eqref{afawefwf} we have $Y_Q\in\frac12 B_Q\cap\Omega$ (cf.~\eqref{eq:BQ}) and also $\delta(Y_Q)=\delta(P_{Q(N_\tau)}) \approx \ell(Q(N_\tau))=2^{-N_\tau}\,\ell(Q)$. Hence, in the context of Lemma~\ref{lemma:AGMT}, applied to $\epsilon=2^{-N_\tau}\ll  \epsilon_\tau$ so that we have that  $\delta(Y_Q)\approx \epsilon\,\ell(Q)$.

Set $\w:=\sigma(Q) \omega_L^{Y_{Q}}$ and note that 
\begin{align}\label{eq:wseq}
C_0^{-1} \le \frac{\w(Q)}{\sigma(Q)} \le 
\frac{\w(\pom)}{\sigma(Q)} 
\le 1.
\end{align}
Write $\mathcal{M}$ for the Hardy-Littlewood maximal operator on $\pom$ (with respect to $\sigma$) and observe that by the Hardy-Littlewood theorem and \eqref{eq:wseq} one has
\[
 \|\mathcal{M} \w\|_{L^{1,\infty}(\pom, \sigma)} 
 \lesssim
 \omega(\pom)
 \lesssim \sigma(Q).
\]
This and Kolmogorov's inequality imply
\begin{align}\label{AWfawcvarc}
	\bigg(\fint_{Q} (\mathcal{M} \w)^{\frac12} \, d\sigma \bigg)^2 
\lesssim
\sigma(Q)^{-1}\, \|\mathcal{M} \w\|_{L^{1,\infty}(\pom, \sigma)} 
\lesssim 1. 
\end{align}
Let $N\gg N_\tau$ be large enough to be chosen. Subdivide dyadically $Q$ and stop the first time that one of the following two conditions occur
\begin{align}\label{eq:stopping}
	\frac{\w(Q')}{\sigma(Q')} < 2^{-N}, \qquad 
	\bigg(\fint_{Q'} (\mathcal{M} \w)^{\frac12} \, d\sigma \bigg)^2 >2^{2N}. 
\end{align}
This stopping time family is denoted by $\F_Q$ and it is pairwise disjoint. Assuming that $N$ is sufficiently large, \eqref{eq:wseq} and \eqref{AWfawcvarc} clearly give that $\F_Q\subset\D_Q\setminus \{Q\}$. Let  $\F_Q^-$ be the subcollection of cubes in $\F_Q$ satisfying the first condition in \eqref{eq:stopping} and set $\F_Q^+=\F_Q\setminus\F_Q^{-}$, that is, the cubes that satisfy the second condition in \eqref{eq:stopping}, but not the first one. By construction $\F_Q:=\F_Q^+ \sqcup \F_Q^-$ and 
\begin{align*}
	2^{-N} \le \frac{\w(Q')}{\sigma(Q')} 
	\le \bigg(\fint_{Q'} (\mathcal{M} \w)^{\frac12} \, d\sigma \bigg)^2 
	\leq 2^{2N}, \qquad \forall\,Q' \in \D_{\F_Q, Q}. 
\end{align*}
This readily implies \eqref{eq:NN} using \eqref{eq:wpQQ} and  $\w:=\sigma(Q) \omega_L^{Y_{Q}}$. 
If $\F_Q^+\neq\emptyset$ we use that the cubes in $\F_Q^+$ are pairwise disjoint and satisfy the second condition in \eqref{eq:stopping}, \eqref{AWfawcvarc}, and \eqref{eq:wseq}, to arrive at \eqref{eq:FQ2}:
\begin{multline*}
\sigma(F_Q^+)=\sigma \bigg(\bigcup_{Q' \in \F_Q^+} Q' \bigg)
= \sum_{Q' \in \F_Q^+} \sigma(Q') 
\le 2^{-N} \sum_{Q' \in \F_Q^+} \int_{Q'} (\mathcal{M} \w)^{\frac12} \, d\sigma
\\
\le 2^{-N} \int_{Q} (\mathcal{M} \w)^{\frac12} \, d\sigma
\lesssim 
2^{-N}  \sigma(Q), 
\end{multline*}
provided $N$ is sufficiently large. This estimate clearly holds if  $\F_Q^+=\emptyset$. Hence we have shown \eqref{eq:FQ2}.

To continue if $\F_Q^-\neq\emptyset$, use that the cubes in $\F_Q^-$ are pairwise disjoint and  satisfy the first  condition in \eqref{eq:stopping}, and \eqref{eq:wseq}, to see that 
\begin{align*}
\w(F_Q^-)
=
\w \bigg(\bigcup_{Q' \in \F_Q^-} Q' \bigg)
= \sum_{Q' \in \F_Q^-} \w(Q') 
\le 
2^{-N} \sum_{Q' \in \F_Q^-} \sigma(Q') 
\le
2^{-N} \sigma(Q)
\le C_0 2^{-N} \w(Q).
\end{align*}
Again this estimate clearly holds in the case  $\F_Q^-=\emptyset$. Thus, in either scenario,
\[
\omega_L^{Y_Q}(Q\setminus F_Q^-) \ge (1-C_0\,2^{-N}) \omega_L^{Y_Q}(Q). 
\]
As observed above, $Y_Q\in\frac12 B_Q\cap\Omega$ and $\delta(Y_Q)\approx \epsilon\,\ell(Q)$. In view of Lemma~\ref{lemma:AGMT}, if one further assumes that $N$ is large enough so that $2^{-N}C_0< \epsilon$, we can find a Borel function $f_Q$ and $u_Q\in W^{1,2}_{\loc}(\Omega)$, with $Lu_Q=0$ in the weak sense in $\Omega$, such that \eqref{eq:uQfQ} and \eqref{eq:uQ-lower} holds where the implicit constants depend on $n$, ADR, ellipticity, and $\tau$, but they are independent of $N$.  
This completes the proof. 
\end{proof}

Having Proposition~\ref{prop:corona:basic-Q} at our disposal, we are going to iterate that construction to obtain the desired corona decomposition. With this goal in mind we fix an arbitrary $Q^0 \in \D$. Set $\F_0:=\{Q^0\}$, $\F_0^+:=\{Q^0\}$, and $\F_0^-:=\emptyset$. Let $\F_1 := \bigsqcup_{Q \in \F_0} \F_Q$ and $\F_1^{\pm} := \bigsqcup_{Q \in \F_0} \F^{\pm}_Q$ denote the first generation cubes. In general, we may define recursively 
\begin{align*}
\F_{k+1} := \bigsqcup_{Q \in \F_k} \F_Q \quad\text{ and }\quad 
\F_{k+1}^{\pm} := \bigsqcup_{Q \in \F_k} \F_Q^{\pm}, \qquad k \ge 0. 
\end{align*}
We also set 
\begin{equation*}
	\F := \bigsqcup_{k=0}^{\infty} \F_k 
	\quad\text{ and }\quad 
	\F^{\pm} := \bigsqcup_{k=0}^{\infty} \F_k^{\pm}. 
\end{equation*} 
With all these we construct a semi-coherent corona decomposition relative to $Q_0$ as follows. Observe first that 
\[
\D_{Q^0} 
= 
\D_{\F_{Q^0}, Q^0} \bigsqcup \bigg(\bigsqcup_{Q \in \F_{Q^0}} \D_Q\bigg)
=
\Big(\bigsqcup_{Q \in \F_0}\D_{\F_Q, Q}\Big) \bigsqcup \bigg(\bigsqcup_{Q \in \F_1} \D_Q\bigg)
.
\]
In turn, 
\[
\bigsqcup_{Q \in \F_1}\D_Q
= 
\bigsqcup_{Q \in \F_1}\bigg(\D_{\F_Q, Q} \bigsqcup \Big(\bigsqcup_{Q' \in \F_Q} \D_{Q'}\Big)\bigg)
=
\Big(\bigsqcup_{Q \in \F_1}\D_{\F_Q, Q}\Big) \bigsqcup \Big(\bigsqcup_{Q' \in \F_2} \D_{Q'}\Big).
\]
Iterating this procedure we eventually obtain 
\begin{align}\label{eq:SG}
	\D_{Q^0} = \bigsqcup_{k=0}^{\infty} \bigg(\bigsqcup_{Q \in \F_k} \D_{\F_Q, Q} \bigg) 
	= \bigsqcup_{Q \in \F} \D_{\F_Q, Q}.
\end{align}
We then set $\B_{Q^0}:=\emptyset$, $\bbF_{Q^0}=\{\D_{\F_Q, Q}\}_{Q\in\F}$, and $\G_{Q^0}=\bigsqcup_{\S\in\bbF_{Q^0}}\S$. Note that trees are of the form $\D_{\F_Q, Q}$ with $Q\in\F$, hence they are clearly semi-coherent with $\Top(\S)=Q$, hence $\Top(\bbF_{Q^0})=\F$.  Thus, in order to see that $(\B_{Q^0},\G_{Q^0},\bbF_{Q^0})$ is a semi-coherent corona decomposition relative to $Q^0$ we are left with showing that the family $\Top(\bbF_{Q_0})=\F$ satisfy a packing condition. This is the content of the following result:

\begin{proposition}\label{prop:packing-corona-w}
Under the previous considerations, and taking $N$ large enough  (depending on $n$, ADR, ellipticity, and $\tau$)  the family $\Top(\bbF_{Q_0})=\F$ satisfies the Carleson condition
\begin{align}\label{erwfggweg}
	\sup_{Q \in \D_{Q^0}}	\frac1{ \sigma(Q)}\sum_{Q'\in \F\cap\D_{Q}} \sigma(Q') 
\lesssim 1,
\end{align}
with an implicit constant which depends on $n$, ADR, ellipticity, and $\tau$.  As  a consequence, $(\B_{Q^0},\G_{Q^0},\bbF_{Q^0})$ is a semi-coherent corona decomposition relative to $Q^0$.
\end{proposition}

Assume this result momentarily and take the corresponding large enough $N$. Pick an $\S\in\bbF_{Q^0}$, that is, $\S=\D_{\F_{Q_0}, Q_0}$ for some $Q_0\in\F$. Set $X_{\S}:=Y_{Q_0}$ so that by  \eqref{eq:YQ-pos}
\[
\delta(X_{\S}) \approx \ell(\Top(\S)) \approx \dist(X_{\S}, \Top(\S)). 
\]
Invoking \eqref{eq:NN} and recalling that $\Top(\S)=Q_0$, we get 
\begin{equation}\label{eq:wsTop}
	2^{-N}\frac{\omega_L^{X_{\S}}(\Top(\S))}{\sigma(\Top(\S))}
	\lesssim \frac{\omega_L^{X_{\S}}(Q)}{\sigma(Q)} 
	\le \bigg(\fint_Q (\mathcal{M} \omega_L^{X_{\S}})^{\frac12} \, d\sigma \bigg)^2 
	\lesssim 
	2^{2N}
	\frac{\omega_L^{X_{\S}}(\Top(\S))}{\sigma(\Top(\S))}, \quad\forall Q \in \S.  
\end{equation}
All these, Proposition~\ref{prop:packing-corona-w}, the fact that $Q^0$ is arbitrary, and Proposition~\ref{pro:global} give that $\omega_L$ admits a semi-coherent corona decomposition with $Q_\S=\Top (\S)$. Note that with the help of Remark~\ref{remark:corona-G:CFMS:coherent} we can refine transform this corona decomposition so that the new corona is coherent. This completes the proof of the current implication modulo proving Proposition~\ref{prop:packing-corona-w}.\qed

\begin{remark}\label{remark:Bourgain-proof}
We would like to observe that in the previous proof we have obtained a semi-coherent corona decomposition such that $Q_\S=\Top (\S)$ and with the additional property that $\omega_L^{X_\S}(Q_\S)\approx 1$, see \eqref{eq:wpQQ}. Invoking Remark~\ref{remark:corona-G:CFMS:coherent} we can transform this into a coherent corona decomposition as in (i) in Definition~\ref{def:corona} with the additional property that $\omega_L^{X_\S}(Q_\S)\approx 1$. A careful examination of  the proofs of $\eqref{list:wL-strong} \Longrightarrow \eqref{list:wL}$, $\eqref{list:wL} \Longrightarrow \eqref{list:GL}$, and $\eqref{list:GL} \Longrightarrow \eqref{list:CME}$ reveals that this extra property can be inherited in any of the implications, that is, both in  (ii) and (iii) in Definition~\ref{def:corona} we obtain the additional the property $\omega_L^{X_\S}(Q_\S)\approx 1$, and in in the partial/weak Carleson measure estimates the extra property $\omega_L^{P_Q}(Q)\approx 1$. All these together mean that in each of the conditions \eqref{list:wL-strong}, \eqref{list:wL}, \eqref{list:GL} we may assume that the corona decomposition has the extra property  $\omega_L^{X_\S}(Q_\S)\approx 1$ for every $\S\in\bbF$. Analogously, in \eqref{list:CME} we may additionally assume that  $\omega_L^{P_Q}(Q)\approx 1$. We observe however that adding this extra condition requires to possibly use a different decomposition in \eqref{list:wL-strong}, \eqref{list:wL}, \eqref{list:GL}, or a different collection of points $\{P_Q\}_{Q\in\D}$ in \eqref{list:CME}. 
\end{remark}

\begin{proof}[Proof of Proposition~\ref{prop:packing-corona-w}]
We borrow some ideas from \cite{GMT, AGMT}. We start by arranging the cubes in $\F$ into some trees. To set the stage let $Q\in\F$ be an arbitrary cube. Write $\Scal_Q$ for the (possible empty) family of maximal cubes in $\D_Q^*\cap\F^+$, where, here and elsewhere, $\D_Q^*:=\D_Q\setminus\{Q\}$. Note that by construction $\Scal_Q\subset \D_{Q}^*\cap \F^+$ is a pairwise disjoint family so that  $(\D_{ \Scal_Q,Q}\setminus\{Q\})\cap\F \subset\F^-$. 

We next iterate the previous selection procedure. Write $\Scal_0=\{Q^0\}$ and define recursively $\Scal_{k+1}=\bigsqcup_{Q\in\Scal_k} \Scal_Q$ for $k\ge 0$. We then set $\Scal=\bigsqcup_{k=0}^\infty \Scal_k$. Recalling that $Q^0\in\F^+$ it is easy to see that $\Scal=\F^+\cap\D_{Q^0}$. On the other hand,
\[
\D_{Q^0} 
= 
\big(\D_{\Scal_{Q^0}, Q^0} \big)\bigsqcup \bigg(\bigsqcup_{Q \in \Scal_{Q^0}} \D_Q\bigg)
=
\Big(\bigsqcup_{Q \in \Scal_0}\D_{\Scal_Q, Q}\Big) \bigsqcup \bigg(\bigsqcup_{Q \in \Scal_1} \D_Q\bigg)
.
\]
In turn, 
\[
\bigsqcup_{Q \in \Scal_1}\D_Q
= 
\bigsqcup_{Q \in \Scal_1}\bigg(\D_{\Scal_Q, Q} \bigsqcup \Big(\bigsqcup_{Q' \in \Scal_Q} \D_{Q'}\Big)\bigg)
=
\Big(\bigsqcup_{Q \in \Scal_1}\D_{\Scal_Q, Q}\Big) \bigsqcup \Big(\bigsqcup_{Q' \in \Scal_2} \D_{Q'}\Big).
\]
Iterating this procedure we eventually obtain 
\begin{align}\label{eq:SGadsas}
	\D_{Q^0} = \bigsqcup_{k=0}^{\infty} \bigg(\bigsqcup_{Q \in \Scal_k} \D_{\Scal_Q, Q} \bigg) 
	= \bigsqcup_{Q \in \Scal} \D_{\Scal_Q, Q}.
\end{align}
We first show that
\begin{align}\label{Wefqwcvf}
\sup_{Q_0\in\Scal} \sup_{Q_0'\in \D_{\Scal_{Q_0}, Q_0}\cap\F} \frac1{\sigma(Q_0')} \sum_{Q\in \D_{\Scal_{Q_0}, Q_0'}\cap \F}\sigma(Q)\lesssim 1,
\end{align}
with a constant that is independent of $N$. Fix then $Q_0\in\Scal$ and $Q_0'\in \D_{\Scal_{Q_0}, Q_0}\cap \F$. For each $Q\in\F$ we set $V_Q=B(Y_Q,(1-\tau)\delta(Y_Q))$.
%
Recalling the notation introduced in Proposition~\ref{prop:corona:basic-Q}, we claim that the family $\{Q\setminus F_Q^-\}_{Q\in \D_{\Scal_{Q_0}, Q_0}\cap \F}$ is pairwise disjoint. Suppose otherwise that there are two distinct cubes $Q, Q'\in \D_{\Scal_{Q_0}, Q_0}\cap\F$ with $Q\setminus F_Q^-$ meeting $Q'\setminus F_{Q'}^-$. By relabeling if needed, we may assume that $Q'\subsetneq Q$. Let $Q''\in\F $ be the maximal cube so that $Q'\subset Q''\subsetneq Q$, and note that since $Q, Q'\in \D_{\Scal_{Q_0}, Q_0}$ we necessarily have that $Q''\in (\D_{\Scal_{Q_0}, Q_0}\setminus\{Q_0\})\cap \F\subset\F^-$ and $Q''\in\F_{Q}^-$. Hence, $Q'\subset Q''\subset F_{Q}^-$ and $Q'$ cannot meet $Q\setminus F_{Q}^-$, which is a contradiction.

For any given $K\gg 1$ we take an arbitrary family $\{Q_k\}_{k=1}^K \subset \D_{\Scal_{Q_0}, Q_0'}\cap \F$. Invoking Proposition~\ref{prop:corona:basic-Q} we can find $f_{Q_k}$ and $u_{Q_k}$ for each $1\le k\le K$ so that \eqref{eq:uQfQ} holds. Write $\{r_j(\cdot)\}_{j\in\N}$ for the Rademacher system in $[0,1)$ and for every $t\in[0,1)$ let us set
\[
f_t:=\sum_{k=1}^{K} r_k(t)  f_{Q_k}\,\mathbf{1}_{Q_k\setminus F_{Q_k}^-}
\]
and
\[
u_t(X)
:=
\int_{\pom} f_t \,d\omega_L^X
=
\sum_{k=1}^{K} r_k(t)  \int_{Q_k\setminus F_{Q_k}^-} f_{Q_k} \,d\omega_L^X
=
\sum_{k=1}^{K} r_k(t) u_{Q_k}(X), \quad X\in\Omega.
\]
Note that $f_t$ is Borel measurable and by  \eqref{eq:uQfQ} and the fact that the family $\{Q\setminus F_Q^-\}_{Q\in \D_{\Scal_{Q_0}, Q_0}\cap \F}$ is pairwise disjoint we conclude that for every $y\in\pom$
\[
|f_t(y)|
\le 
\sum_{k=1}^{K} |r_j(t)|\,  |f_{Q_k}(y)|\,\mathbf{1}_{Q_k\setminus F_{Q_k}^-}(y)
\lesssim
\sum_{k=1}^{K} \mathbf{1}_{Q_k\setminus F_{Q_k}^-}(y)
\le
1,
\]
Hence, $\|u_t\|_{L^\infty(\Omega)}\lesssim 1$, uniformly on $t$.
As such,  recalling Proposition~\ref{prop:corona:basic-Q} and that $Y_Q=P_{Q(N_\tau)}$, we can invoke $\eqref{list:CME}$ to obtain
\begin{multline*}
\sum_{k=1}^K \iint_{V_{Q_k}} |\nabla u_t(X)|^2\,\delta(X)\,dx
=
\sum_{k=1}^K \iint_{B(P_{Q_k(N_\tau)},(1-\tau)\delta(P_{Q_k(N\tau)}) )} |\nabla u_t(X)|^2\,\delta(X)\,dx
\\
\le
\sum_{Q\in\D_{Q_0'}} \iint_{B(P_Q,(1-\tau)\delta(P_Q))} |\nabla u_t(X)|^2\,\delta(X)\,dx
\lesssim
\sigma(Q_0'),
\end{multline*}
with implicit constants that are uniform on $t\in[0,1)$. Thus, by \eqref{eq:uQ-lower} and the orthogonality of the Rademacher system we arrive at 
\begin{align*}
	\sum_{k=1}^K \sigma(Q_k)
	&\lesssim
	\sum_{k=1}^K \iint_{V_{Q_k}} |\nabla u_{Q_k}(X)|^2\,\delta(X)\,dX
	\\
	&
	\le 
	\iint_{\bigcup\limits_{k=1}^{K} V_{Q_k}} \sum_{k=1}^{K} |\nabla u_{Q_k}(X)|^2 \delta(X)\, dX
	\\
	&
	=
	\sum_{j=1}^{n+1} \iint_{\bigcup\limits_{k=1}^{K} V_{Q_k}} \sum_{k=1}^{K} |\partial_j u_{Q_k}(X)|^2 \delta(X)\, dX
	\\
	&
	=
	\sum_{j=1}^{n+1} \iint_{\bigcup\limits_{k=1}^{K} V_{Q_k}}  
	\bigg( \int_0^1 \bigg| \sum_{k=1}^{K} r_k(t) \partial_j u_{Q_k}(X) \bigg|^2 dt \bigg) \delta(X)\, dX
	\\
	&
	=
	\int_0^1 
	\bigg( \iint_{\bigcup\limits_{k=1}^{K} V_{Q_k}}  |\nabla u_{t}(X)|^2\,\delta(X)\,dX\bigg) \,dt
	\\
	&
\le 
	\int_0^1 
	\bigg( \sum_{k=1}^K\iint_{V_{Q_k}}  |\nabla u_{t}(X)|^2\,\delta(X)\,dX\bigg) \,dt
	\\
	&
	\lesssim
	\sigma(Q_0').
\end{align*}
Recalling that $K\gg 1$ and the family $\{Q_k\}_{k=1}^K \subset \D_{\Scal_{Q_0}, Q_0'}\cap \F$ is arbitrary we can easily see that
\[
\sum_{Q\in \D_{\Scal_{Q_0}, Q_0'}\cap \F} \sigma(Q)\lesssim \sigma(Q_0'),
\]
hence \eqref{Wefqwcvf} holds. 

To proceed we next show that
\begin{align}\label{e4wgvsae}
\sup_{Q_0\in\Scal} \frac1{\sigma(Q_0)} \sum_{Q\in \Scal\cap \D_{Q_0}}\sigma(Q)\le 2,
\end{align}
provided $N$ is taken large enough. 

For starters, fix $Q_0\in \Scal$ and let $k_0\ge 0$ be such that $Q_0\in\Scal_{k_0}$. Then,
\begin{align*}
\sum_{Q\in \Scal\cap \D_{Q_0}}\sigma(Q)
=
\sum_{k=k_0}^\infty \sum_{Q\in \Scal_k\cap \D_{Q_0}}\sigma(Q)
=:
\sum_{k=k_0}^\infty \Sigma_k,
\end{align*}
and we estimate each term in turn. Note first that 
\[
\Sigma_{k_0}
=
\sum_{Q\in \Scal_{k_0}\cap \D_{Q_0}}\sigma(Q)
=
\sigma(Q_0).
\]
Additionally, for each $k\ge k_0+1$ we have
\[
\Sigma_k
=
\sum_{Q\in \Scal_k\cap \D_{Q_0}}\sigma(Q)
=
\sum_{Q\in \Scal_{k-1}}\sum_{Q'\in \Scal_Q\cap \D_{Q_0}}\sigma(Q')
=
\sum_{Q\in \Scal_{k-1}\cap \D_{Q_0}}\sum_{Q'\in \Scal_Q}\sigma(Q').
\]
We claim that for any fixed $Q\in \Scal_{k-1}\cap \D_{Q_0}$ there holds
\begin{align}\label{Wqqfvwa}
	\Scal_Q=\bigsqcup_{Q'\in \D_{\Scal_{Q},Q} \cap \F} \F_{Q'}^+.
\end{align}
That the union is made of pairwise disjoint collections is clear from the iterative construction of the family $\F$. On the other hand, 
if $Q'\in \Scal_Q\subset\D_Q^*\cap \F^+$, let $Q''\in\F$ be so that $Q'\in\F_{Q''}^+\subset \D_{Q''}^*$. Note that since $Q'\subsetneq Q\in\Scal_{k-1}\subset \F^+$ we must have $Q''\subset Q$. Hence, $Q''\in\D_{\Scal_{Q},Q}$ and this shows the left-to-right inclusion. 

Conversely, let $Q'\in \D_{\Scal_{Q},Q} \cap \F$ and $Q''\in\F_{Q'}^+$. Since $Q''\subsetneq Q'\subset Q$ and, as observed before $(\D_{\Scal_Q,Q}\setminus\{Q\}) \cap \F\subset\F^-$, we obtain that  $Q''\notin\D_{\Scal_Q,Q}$. That is, $Q''\subset Q'''\in \Scal_Q$. Since $Q''\subset Q'\cap Q'''$ we must have $Q'''\subsetneq Q'$ ---otherwise, $Q'\subset Q'''$ and, as a consequence, $Q'\notin \D_{\Scal_{Q},Q}$, which is a contradiction. Using that 
$\F^+\ni Q'''\subsetneq Q'\subsetneq Q$ and that  $Q''\in\F_{Q'}^+$, we arrive at $Q'''=Q''$, hence $Q''\in\Scal_Q$ as desired. 

Using then \eqref{Wqqfvwa},  \eqref{eq:FQ2}, and \eqref{Wefqwcvf} (whose implicit constant does not depend on $N$) we obtain
\begin{multline*}
	\Sigma_k
=
\sum_{Q\in \Scal_{k-1}\cap \D_{Q_0}}\sum_{Q'\in \D_{\Scal_{Q},Q} \cap \F}\sum_{Q''\in\F_{Q'}^+} \sigma(Q'')
=
\sum_{Q\in \Scal_{k-1}\cap \D_{Q_0}}\sum_{Q'\in \D_{\Scal_{Q},Q} \cap \F}\sigma(F_{Q'}^+)
\\
\le
2^{-N}\sum_{Q\in \Scal_{k-1}\cap \D_{Q_0}}\sum_{Q'\in \D_{\Scal_{Q},Q} \cap \F}\sigma(Q')
\lesssim
2^{-N}\sum_{Q\in \Scal_{k-1}\cap \D_{Q_0}} \sigma(Q)
=
2^{-N}\Sigma_{k-1}
<\frac12 \Sigma_{k-1},
\end{multline*}
provided $N$ is taken large enough. Iterating this we conclude as desired 
\begin{align*}
\sum_{Q\in \Scal\cap \D_{Q_0}}\sigma(Q)
=
\sum_{k=k_0}^\infty \Sigma_k
\le
\Sigma_{k_0} \sum_{k=k_0}^\infty 2^{-(k-k_0)}
=
2\Sigma_{k_0}
=
2\sigma(Q_0).
\end{align*}

With \eqref{Wefqwcvf} and \eqref{e4wgvsae} we are now ready to obtain \eqref{erwfggweg}. Fix then $Q_0\in\D_{Q^0}$ and assume first that $Q_0\in\F$. By \eqref{eq:SGadsas} we can find $Q_0'\in\Scal$ such that $Q_0\in\D_{\Scal_{Q_0'}, Q_0'}$. We note that if $Q\in\Scal$ with $Q_0'\subsetneq Q$ then $\D_{Q_0'}\cap \D_{\Scal_{Q}, Q}=\emptyset$, otherwise there is  $Q'\in \D_{Q_0'}\cap \D_{\Scal_{Q}, Q}$ and since $Q'\subset Q_0'\subsetneq Q$ we conclude that $Q_0'\in \D_{\Scal_{Q}, Q}$, and this contradicts \eqref{eq:SGadsas} and the fact that  $Q_0'\in\D_{\Scal_{Q_0'}, Q_0'}$. Thus,   using again \eqref{eq:SGadsas} we have
\begin{multline*}
\sum_{Q\in \F\cap\D_{Q_0}} \sigma(Q) 
=
\sum_{Q\in\Scal} \sum_{Q'\in \D_{\Scal_Q,Q\cap Q_0}\cap\F} \sigma(Q') 
=
\sum_{Q\in\Scal\cap \D_{Q_0'}} \sum_{Q'\in \D_{\Scal_Q,Q\cap Q_0}\cap\F} \sigma(Q') 
\\
\le 
\sum_{Q\in\Scal\cap \D_{Q_0'}: Q_0\subset Q} \sum_{Q'\in \D_{\Scal_Q,Q_0}\cap\F} \sigma(Q') 
+
\sum_{Q\in\Scal\cap \D_{Q_0}} \sum_{Q'\in \D_{\Scal_Q,Q}\cap\F} \sigma(Q') 
\\
=
\sum_{Q'\in \D_{\Scal_{Q_0'},Q_0}\cap\F} \sigma(Q')
+
\sum_{Q\in\Scal\cap \D_{Q_0}} \sum_{Q'\in \D_{\Scal_Q,Q}\cap\F} \sigma(Q') 
, 
\end{multline*}
where in the last equality we have used that if $Q\in\Scal\cap \D_{Q_0'}$ with  $Q_0\subset Q$ and there is $Q'\in \D_{\Scal_Q,Q_0}\cap\F$  we have
$Q'\subset Q_0\subset Q$ with $Q'\in \D_{\Scal_Q,Q_0}\subset \D_{\Scal_Q,Q}$, hence $Q_0\in \D_{\Scal_Q,Q}\cap\F$ and by \eqref{eq:SGadsas}  we conclude that $Q=Q_0'$. 
Let $\Scal^{Q_0}$ be the collection of maximal cubes (hence, pairwise disjoint) in $\Scal\cap\D_{Q_0}$. Then, invoking \eqref{Wefqwcvf} and \eqref{e4wgvsae} we easily see that 
\begin{multline*}
	\sum_{Q\in \F\cap\D_{Q_0}} \sigma(Q) 
\lesssim	
\sigma(Q_0) +
	\sum_{Q\in\Scal\cap \D_{Q_0}} \sigma(Q) 
=
\sigma(Q_0) +
\sum_{Q\in \Scal^{Q_0} } \sum_{Q'\in\Scal\cap \D_{Q}} \sigma(Q') 
\\
\le
\sigma(Q_0) +
2\,\sum_{Q\in \Scal^{Q_0} } \sigma(Q) 
=
\sigma(Q_0) +
2\,\sigma\bigg(\bigsqcup_{Q\in \Scal^{Q_0}} Q\bigg)
\le
3\sigma(Q_0).
\end{multline*}
Recalling that $Q_0\in\D_{Q^0}\cap\F$ is arbitrary we have shown that 
\begin{align}\label{fsgetgte}
	\sup_{Q\in \D_{Q^0}\cap\F}\frac1{\sigma(Q)}\sum_{Q'\in \F\cap\D_{Q}} \sigma(Q') \lesssim 1, 
\end{align}
and we are left with the consider in the sup all cubes $Q\in\D_{Q^0}$. To see this we fix $Q_0\in\D_{Q^0}$ and let $\F^{Q_0}$ be the collection of maximal cubes (hence, pairwise disjoint) in $\F\cap\D_{Q_0}$. It is straightforward to see that \eqref{fsgetgte} yields
\[
\sum_{Q\in \F\cap\D_{Q_0}} \sigma(Q)
=
\sum_{Q\in \F^{Q_0}} \sum_{Q'\in \F\cap\D_{Q}} \sigma(Q')
\lesssim
\sum_{Q\in \F^{Q_0}} \sigma(Q)
=
\sigma\bigg(\bigsqcup_{Q\in \F^{Q_0}} Q\bigg)
\le
\sigma(Q_0).
\]
This completes the proof. 
\end{proof}

\section{Proof of Theorems \ref{thm:LL} and \ref{thm:LLT}}\label{sec:perturbation}

Theorems \ref{thm:LL} and \ref{thm:LLT} will follow from the following general result. First we introduce some notation. Give a matrix $D=(d_{i,j})_{1\le i,j\le n+1}\in \Lip_{\loc}(\Omega)$ we write $\div_C D$ to denote its divergence acting on columns, that is,
\[
\div_C D:=
\big( \div(d_{\cdot, j})\big)_{1\le j \le n+1} =
\Big(\sum_{i=1}^{n+1} \partial_i d_{i,j}\Big)_{1\le j \le n+1}.
\]

\begin{theorem}\label{thm:AAAD}
Let $\Omega \subset \R^{n+1}$, $n \ge 2$, be an open set with ADR boundary satisfying the corkscrew condition. Let $L_0u=-\div(A_0 \nabla u)$ and $L_1u=-\div(A_1 \nabla u)$ be real (not necessarily symmetric) elliptic operators. Assume that $A_0-A_1=A+D$ where $A, D \in L^{\infty}(\Omega)$ are real matrices verifying the following conditions: 
\begin{list}{\rm (\theenumi)}{\usecounter{enumi}\leftmargin=1cm \labelwidth=1cm \itemsep=0.1cm \topsep=.2cm \renewcommand{\theenumi}{\roman{enumi}}}

\item  The function $a(X):=\sup_{Y \in B(X, \delta(X)/2)} |A(Y)|$, $X \in \Omega$ satisfies the Carleson measure estimate 
\begin{equation}\label{eq:Car-a}
\sup_{\substack{x \in \partial \Omega \\ 0<r<\diam(\partial \Omega)} }
\frac{1}{\sigma(B(x, r) \cap \partial \Omega)} \iint_{B(x, r) \cap \Omega} 
\frac{a(X)^2}{\delta(X)} dX<\infty.
\end{equation} 
 
\item $D \in \Lip_{\loc}(\Omega)$ is antisymmetric and its divergence acting on columns $\div_C D$ satisfies the Carleson measure estimate
\begin{equation}\label{eq:Car-D}
\sup_{\substack{x \in \partial \Omega \\ 0<r<\diam(\partial \Omega)} }
\frac{1}{\sigma(B(x, r) \cap \partial \Omega)} \iint_{B(x, r) \cap \Omega} |\div_C D(X)|^2 \delta(X) dX<\infty. 
\end{equation} 
\end{list} 
Then, $\omega_{L_0}$ admits a (strong) corona decomposition (equivalently, $G_{L_0}$ is comparable to the distance to the boundary in the corona sense or $L_0$ satisfies partial/weak Carleson measure estimates) if and only if $\omega_{L_1}$ admits a (strong) corona decomposition (equivalently, $G_{L_1}$ is comparable to the distance to the boundary in the corona sense or $L_1$ satisfies partial/weak Carleson measure estimates).  
\end{theorem}

\begin{proof}
By symmetry and Theorem~\ref{thm:corona}, it suffices to show that if $G_{L_0}$ is comparable to the distance to the boundary in the corona sense then $L_1$ satisfies partial/weak Carleson measure estimates. We are going to use the ideas from Section  \ref{sec:G-CME}. We fix $u \in W^{1,2}_{\loc}(\Omega) \cap L^{\infty}(\Omega)$ a non-trivial weak solution of $L_1u=0$ in $\Omega$. By homogeneity we may assume that $\|u\|_{L^{\infty}(\Omega)}=1$. Fix also $\tau\in (0,\frac12)$, the parameter which appears in the partial/weak Carleson measure estimate (see Definition~\ref{def:Carleson}).  Assume that $(\B_0, \G_0, \bbF_0)$ is the assumed corona decomposition associated with $G_{L_0}$. Given $N\ge 0$, large enough to be chosen momentarily (depending on $\tau$), we may proceed as in Remark~\ref{remark:corona-G:CFMS:N-close} and pick $(\B, \G, \bbF)$ the  $2^N$-refinement of $(\B_0, \G_0, \bbF_0)$. This implies that for every $\S\in\bbF$ there exist $Q_\S\in\D$ and $X_\S \in \Omega$ so that \eqref{se5w4g4g} and \eqref{prwgbsbs} hold with $L=L_0$.  Writing $G_{\S}:=G_{L_0}(X_{\S}, \cdot)$, by \eqref{prwgbsbs} and Harnack's inequality, for each $Q\in\S(N)$ there exists $P_Q \in  2\,\widetilde{B}_Q\cap\Omega$ with $\delta(P_Q) \ge c_0 \ell(Q)$ such that
\begin{align}\label{eq:Lambda-S:*}
	\Lambda_{\S} \lesssim \frac{G_{\S}(P_Q)}{\delta(P_Q)} \approx_\tau \frac{G_{\S}(X)}{\delta(X)}, 
	\quad\forall X \in V_Q := B(P_Q, (1-\tau)\delta(P_Q)). 
\end{align} 
Set 
\[
\alpha_Q := \iint_{V_Q} |\nabla u(X)|^2 \delta(X) \, dX, \quad\, Q \in \D, 
\]
which, much as before, satisfies \eqref{qafvavwe}. Following the argument in Section \ref{sec:G-CME} it is then easy to see that we can reduce matters to to proving the following estimate: 
\begin{align}\label{eq:JN-bound}
\mathcal{J}_M:=\iint_{\Omega_{\F_M, Q_0}^\vartheta} |\nabla u|^2 \, G_{\S} \, dX 
\lesssim \Lambda_\S \, \sigma(Q_0). 
\end{align}
with implicit constant independent of $M\ge 1$, $\S$, and $Q_0$. We warn the reader that the difference of the present case with Section \ref{sec:G-CME} is that before, $u$ and $G_\S$ were associated with the same operator $L$, while now $u$ is a associated with $L_1$ and $G_{\S}$ is associated with $L_0$. 

To show \eqref{eq:JN-bound}, we pick $\Psi_M$ from Lemma~\ref{lem:approx} and use Leibniz's rule to get 
\begin{multline}\label{eq:Leib*}
A_1 \nabla u \cdot \nabla u \  G_{\S} \Psi_M^2 
=A_1 \nabla u \cdot \nabla (u G_{\S} \Psi_M^2) 
- \frac12 A_0 \nabla (u^2 \Psi_M^2) \cdot \nabla G_{\S} 
\\ 	
+ \frac12 A_0 \nabla(\Psi_M^2) \cdot \nabla G_{\S} \ u^2
-\frac12 A_0 \nabla (u^2) \cdot \nabla(\Psi_M^2) G_{\S}
+\frac12 \mathcal{E} \nabla(u^2) \cdot \nabla(G_{\S} \Psi_M^2), 
\end{multline}
where $\mathcal{E}(X):=A_0(X) - A_1(X)$. Note that $u \in W^{1,2}_{\loc}(\Omega)\cap L^\infty(\Omega)$, $G_{\S} \in W^{1,2}_{\loc}(\Omega \setminus \{X_{\S}\})$, and that $\overline{\Omega_{\F_M, Q_0}^{\vartheta, **}}$ is a compact subset of $\Omega$ away from $X_{\S}$. Hence, $u \in W^{1,2}(\Omega_{\F_M, Q_0}^{\vartheta, **})$ and $u G_{\S} \Psi_M^2 \in W_0^{1,2}(\Omega_{\F_M, Q_0}^{\vartheta, **})$. Since $L_1u=0$ in the weak sense in $\Omega$, these lead to
\begin{align}\label{eq:AA-1*}
\iint_{\Omega} A_1 \nabla u \cdot \nabla(u G_{\S} \Psi_M^2) dX
=\iint_{\Omega_{\F_M, Q_0}^{\vartheta, **}} A_1 \nabla u \cdot \nabla(u G_{\S} \Psi_M^2) dX=0.
\end{align}
On the other hand, $G_{\S} \in W^{1,2}(\Omega_{\F_M, Q_0}^{\vartheta, **})$ and $L_0^{\top} G_{\S}=0$ in the weak sense in $\Omega \setminus\{X_{\S}\}$. Thanks to the fact that $u^2 \Psi_M^2 \in W^{1,2}_0(\Omega_{\F_M, Q_0}^{\vartheta, **})$, we then obtain
\begin{align}\label{eq:AA-2*}
\iint_{\Omega} A_0 \nabla (u^2 \Psi_M^2) \cdot \nabla G_{\S} \, dX
=\iint_{\Omega_{\F_M, Q_0}^{\vartheta, *}} A_0^{\top} \nabla G_{\S} \cdot \nabla(u^2 \Psi_M^2)\, dX=0.
\end{align}
By Lemma~\ref{lem:approx}, the ellipticity of $A$, \eqref{eq:Leib*}--\eqref{eq:AA-2*}, and the fact that $\|u\|_{L^{\infty}(\Omega)}=1$, we then arrive at
\begin{align}\label{eq:JN-1}
\widetilde{\mathcal{J}}_M 
&:= \iint_{\Omega} |\nabla u|^2 G_{\S} \Psi_M^2 \, dX 
\lesssim \iint_{\Omega} A_1 \nabla u \cdot \nabla u \ G_{\S} \Psi_M^2\, dX
\nonumber\\
&\lesssim \iint_{\Omega} \big(|\nabla G_{\S}| + |\nabla u| G_{\S}\big) |\nabla\Psi_M|\, dX
+ \bigg|\iint_{\Omega} \mathcal{E} \nabla(u^2) \cdot \nabla(G_{\S} \Psi_M^2) \, dX\bigg|
\nonumber \\  
&\lesssim
\sigma(Q_0)+ 
\bigg|\iint_{\Omega} \mathcal{E} \nabla(u^2) \cdot \nabla(G_{\S} \Psi_M^2) \, dX\bigg|, 
\end{align}
where the last estimate follows as in \eqref{eq:KN} and \eqref{34qt3tgg3} with the help of \eqref{agergaegvbe}, and where the implicit constant does not depend on $M$.  To bound the last term we use that  $\mathcal{E}=A_0-A_1=A+D$ to get 
\begin{multline}\label{eq:JN-3} 
\bigg|\iint_{\Omega} \mathcal{E} \nabla(u^2) \cdot \nabla(G_{\S} \Psi_M^2) \, dX\bigg|
\\
\le \bigg|\iint_{\Omega} A \nabla(u^2) \cdot \nabla(G_{\S} \Psi_M^2) \, dX\bigg| 
+ \bigg|\iint_{\Omega} D \nabla(u^2) \cdot \nabla(G_{\S} \Psi_M^2) \, dX\bigg|
=:
\widetilde{\mathcal{J}}_M^1 
+\widetilde{\mathcal{J}}_M ^2
.   
\end{multline}
To control $\widetilde{\mathcal{J}}_M^1 $, we write 
\begin{align}\label{eq:JN-412} 
\widetilde{\mathcal{J}}_M^1
\lesssim \iint_{\Omega} |A| |\nabla u|  |\nabla G_{\S}| \Psi_M^2 \, dX 
+ \iint_{\Omega}  |\nabla u|  |\nabla \Psi_M| \Psi_M G_{\S} \, dX
=:
\widetilde{\mathcal{J}}_M^{1,1} + \widetilde{\mathcal{J}}_M^{1,2}.   
\end{align}
Considering the first term, we observe that  $\sup_{I^{**}} |A| \le \inf_{I^*} a$ for any $I \in \W$ 
since $I^{**} \subset \{Y \in \Omega: |Y-X|<\delta(X)/2\}$ for each $X \in I^*$. By Caccioppoli's and Harnack's inequality,  we obtain 
\begin{align}\label{eq:JN-41} 
\widetilde{\mathcal{J}}_M^{1,1} 
& \lesssim \sum_{I \in \W_M^\vartheta} \Big(\sup_{I^{**}} |A| \Big)  
\bigg(\iint_{I^{**}} |\nabla u|^2 \Psi_M^2 \, dX \bigg)^{\frac12} 
\bigg(\iint_{I^{**}} |\nabla G_{\S}|^2 \, dX \bigg)^{\frac12}
\nonumber\\
& \lesssim \sum_{I \in \W_M^\vartheta} \ell(I)^{-1} \Big(\sup_{I^{**}} |A| \Big)  
\bigg(\iint_{I^{**}} |\nabla u|^2 \Psi_M^2 \, dX \bigg)^{\frac12} 
\bigg(\iint_{I^{***}} |G_{\S}|^2 \, dX \bigg)^{\frac12}
\nonumber\\
& \lesssim \sum_{I \in \W_M^\vartheta} \Big(\sup_{I^{**}} |A|^2 \ell(I)^{n-1} G_{\S}(X_I)^2 \Big)^{\frac12}   
\bigg(\iint_{I^{**}} |\nabla u|^2 \Psi_M^2 \, dX \bigg)^{\frac12} 
\nonumber\\
& \lesssim \Lambda_{\S}^{\frac12} \sum_{I \in \W_M^\vartheta} \bigg(\iint_{I^*} \frac{a(X)^2}{\delta(X)} dX \bigg)^{\frac12}   
\bigg(\iint_{I^{**}} |\nabla u|^2 G_{\S} \Psi_M^2 \, dX \bigg)^{\frac12} 
\nonumber\\
& \lesssim \Lambda_{\S}^{\frac12} \bigg(\iint_{B^*_{Q_0} \cap \Omega} \frac{a(X)^2}{\delta(X)} dX \bigg)^{\frac12}   
\bigg(\iint_{\Omega} |\nabla u|^2 G_{\S} \Psi_M^2 \, dX \bigg)^{\frac12} 
\nonumber\\
&  
\lesssim \big(\widetilde{\mathcal{J}}_M \, \Lambda_{\S} \, \sigma(Q_0)\big)^{\frac12},  
\end{align}
where we have used \eqref{fgergerg} in the fourth inequality and \eqref{eq:Car-a} in the last one. On the other hand, \eqref{fgergerg}  and \eqref{34qt3tgg3} imply
\begin{align}\label{eq:JN-42} 
\widetilde{\mathcal{J}}_M^{1,2}
&\le
\bigg(\iint_{\Omega}  |\nabla \Psi_M|^2 G_{\S}\, dX\bigg)^{\frac12}
 \bigg(\iint_{\Omega}  |\nabla u|^2 G_{\S} \Psi_M^2 \, dX\bigg)^{\frac12} 
\nonumber\\
&
\lesssim
\bigg(\sum_{I \in \W_M^\vartheta} \ell(I)^{n-1} G_{\S}(X(I))\bigg)^{\frac12}
\bigg(\iint_{\Omega}  |\nabla u|^2 G_{\S} \Psi_M^2 \, dX\bigg)^{\frac12} 
\nonumber\\
&
\lesssim \Lambda_{\S}^{\frac12}
\bigg(\sum_{I \in \W_M^\vartheta} \ell(I)^{n}\bigg)^{\frac12}
\bigg(\iint_{\Omega}  |\nabla u|^2 G_{\S} \Psi_M^2 \, dX\bigg)^{\frac12} 
\nonumber\\
&
\lesssim \big(\widetilde{\mathcal{J}}_M\, \Lambda_{\S} \, \sigma(Q_0)\big)^{\frac12}
. 
\end{align}
Moreover, it follows from \cite[Lemma 4.1]{CHMT} (whose proof works to the present scenario with no change), \eqref{fgergerg},  and \eqref{eq:Car-D} that 
\begin{align}\label{eq:JN-5} 
\widetilde{\mathcal{J}}_M^2 &= \bigg|\iint_{\Omega} \div_C D \cdot \nabla(u^2)  G_{\S} \Psi_M^2 \, dX\bigg| 
\nonumber\\
&\lesssim \bigg(\iint_{\Omega} |\div_C D|^2 G_{\S} \Psi_M^2 \, dX\bigg)^{\frac12} 
\bigg(\iint_{\Omega} |\nabla u|^2  G_{\S} \Psi_M^2 \, dX\bigg)^{\frac12} 
\nonumber\\
&\lesssim \widetilde{\mathcal{J}}_M^{\frac12} \bigg(\sum_{I \in \W_M^\vartheta} G_{\S}(X(I)) 
\iint_{I^{**}} |\div_C D|^2 dX \bigg)^{\frac12} 
\nonumber\\
&\lesssim \Lambda_{\S}^{\frac12} \widetilde{\mathcal{J}}_M^{\frac12} \bigg( 
\iint_{B^*_{Q_0} \cap \Omega} |\div_C D|^2 \delta(X) \, dX \bigg)^{\frac12} 
\nonumber\\
&
\lesssim \big(\widetilde{\mathcal{J}}_M\, \Lambda_{\S} \, \sigma(Q_0)\big)^{\frac12}. 
\end{align}

Now gathering \eqref{eq:JN-1}--\eqref{eq:JN-5} and using Young's inequality, we obtain
\begin{align}\label{eq:JNJN}
\widetilde{\mathcal{J}}_M  
\lesssim \Lambda_{\S} \, \sigma(Q_0)  + \big(\widetilde{\mathcal{J}}_M \, \Lambda_{\S} \, \sigma(Q_0)\big)^{\frac12} 
\le \frac12\, \widetilde{\mathcal{J}}_M + C_0\, \Lambda_{\S} \, \sigma(Q_0),
\end{align}
where the implicit constant and $C_0$ are independent of $M$. Observe that $\widetilde{\mathcal{J}}_M<\infty$ since $\supp \Psi_M \subset \overline{\Omega_{\F_M, Q_0}^{\vartheta,*}}$, which is a compact subset of $\Omega$ and $u \in W^{1,2}_{\loc}(\Omega)$. Thus, we hide the last term in \eqref{eq:JNJN} and obtain 
\begin{align*}
\mathcal{J}_M  \lesssim \widetilde{\mathcal{J}}_M \lesssim \Lambda_{\S} \, \sigma(Q_0).  
\end{align*}
This shows \eqref{eq:JN-bound} and completes the proof.  
\end{proof}

We next see how Theorems \ref{thm:LL} and \ref{thm:LLT} are deduced from Theorem \ref{thm:AAAD}. 

\begin{proof}[Proof of Theorem \ref{thm:LL}]
Let $L_0$ and $L_1$ be the elliptic operators given in Theorem \ref{thm:LL}. Setting $A=A_0-A_1$ and $D=0$ in Theorem \ref{thm:AAAD}, we see that \eqref{eq:Car-a} coincides with the Carleson condition \eqref{carleson-perturb}  in Theorem \ref{thm:LL}, and \eqref{eq:Car-D} holds automatically. Therefore, Theorem \ref{thm:LL} immediately follows from Theorem \ref{thm:AAAD}.  
\end{proof}

\begin{proof}[Proof of Theorem \ref{thm:LLT}]
Let $A$ be the matrix from Theorem \ref{thm:LLT}. If we take $A_0=A$, $A_1=A^{\top}$, $\widetilde{A}=0$ and 
$D=A-A^{\top}$ in Theorem \ref{thm:AAAD}, then $A_0-A_1=\widetilde{A}+D$ with $D \in \Lip_{\loc}(\Omega)$ antisymmetric, \eqref{eq:Car-a} holds trivially and  \eqref{eq:Car-D} agrees with the Carleson measure estimate \eqref{carleson-div} in Theorem \ref{thm:LLT}. Thus, Theorem \ref{thm:AAAD} implies that $\omega_L$ admits a corona decomposition if and only so does $\omega_{L^{\top}}$.  

Similarly, if we pick $A_0=A$, $A_1=(A+A^{\top})/2$, $\widetilde{A}=0$ and $D=(A-A^{\top})/2$, then $\omega_L$ admits a corona decomposition if and only if so does. 
\end{proof}

\section{Proof of Theorem \ref{thm:UR}}\label{section:UR}
In this section we prove Theorem~\ref{thm:UR}. Assume first that $\pom$ is uniformly rectifiable (here we do not need to impose the corkscrew condition). Invoking \cite[Theorem~1.1]{HMM} with $E=\pom$ we obtain that any bounded harmonic function in $\ree\setminus E$ satisfies (full) Carleson measure estimates. Hence, clearly bounded harmonic functions in $\Omega$ satisfy (full and hence partial) Carleson measure estimates in $\Omega$. This together with Theorem~\ref{thm:corona}, implies that harmonic measure, $\omega_{-\mathcal{L}}$, admits a corona a decomposition. The same can be done with the operators in \eqref{thm:UR:a}, because they are a subclass of the Kenig-Pipher operators (see \cite[Theorem~7.5]{HMM2}), or similarly to those in \eqref{thm:UR:b} much as in \cite[Corollary~10.3]{HMMTZ}.

\subsection{Proof of Theorem \ref{thm:UR} for the Laplacian}
The argument to show that UR follows from the corona decomposition associated with the harmonic measure  is somehow implicit in \cite[Section 5]{HLMN} and we sketch the main changes. For starters we need some preliminaries. We first recall \cite[Definitions~2.17 and 2.19]{HLMN}: 
\begin{definition}
	Let $E \subset \R^{n+1}$ be an ADR set. 
	
	\begin{list}{\rm (\theenumi)}{\usecounter{enumi}\leftmargin=1cm \labelwidth=1cm \itemsep=0.2cm \topsep=.2cm \renewcommand{\theenumi}{\roman{enumi}}}
		
		\item Given $\varepsilon>0$, we say that $Q \in \D(E)$ satisfies the $\varepsilon$-local weak half-space approximation condition ($\varepsilon$-local WHSA with parameter $K_0$) if there is a half-space $H=H(Q)$, a hyperplane $P=P(Q)=\partial H$, and a fixed positive number $K_0$ satisfying
		\begin{list}{\rm (\theenumii)}{\usecounter{enumii}\leftmargin=1cm \labelwidth=1cm \itemsep=0.2cm \topsep=.2cm \renewcommand{\theenumi}{\alph{enumii}}}
			\item  $\dist(X, E) \leq \varepsilon \ell(Q)$, for every $X \in P \cap B_Q^{**}(\varepsilon)$, 
			where $B_Q^{**}(\varepsilon)=B(x_Q, \varepsilon^{-2}\ell(Q))$.

			\item$\dist(Q, P) \leq K_0^{3/2} \ell(Q)$. 
			
			\item $H \cap B_Q^{**}(\varepsilon) \cap E= \emptyset$.  
		\end{list}
		
		\item We say that $E $ satisfies the weak half-space approximation property (WHSA) if for some pair of positive constants $\varepsilon_0$ and $K_0$, and for every $0<\varepsilon<\varepsilon_0$, there is a constant $C_{\varepsilon}$ such that the set $\mathcal{B}$ of bad cubes in $\D(E)$, for which the $\varepsilon$-local WHSA condition with parameter $K_0$ fails, satisfies the packing condition 
		\begin{align*}
		\sum_{Q \subset Q_0: Q \in \mathcal{B}} \sigma(Q) 
		\leq C_{\varepsilon} \sigma(Q_0),\qquad\forall Q_0 \in \D(E).
		\end{align*}
	\end{list}
\end{definition} 

Recall next the definition of the Whitney regions $U_Q^\vartheta$ where $\vartheta\ge \vartheta_0$ is some parameter large enough to be chosen momentarily. We observe that these Whitney regions may have more than one connected component, but that the number of distinct components is uniformly bounded, depending only upon $\vartheta$ and dimension.  We enumerate the connected components of $U_Q^\vartheta$ as $\{U_Q^{\vartheta,i}\}_i$. As in \cite[Definition~2.26]{HLMN}  we enlarge the Whitney regions as follows. For small enough $\varepsilon>0$, we write $X\approx_{\varepsilon,Q} Y$ if $X$ may be connected to $Y$ by a chain of at most $\varepsilon^{-1}$ balls of the form $B(Y_k,\delta(Y_k)/2)$, with  	$\varepsilon^3\ell(Q)\leq\delta(Y_k)\leq \varepsilon^{-3}\ell(Q)$. We then set
\begin{equation}\label{eq2.3a}
\widetilde{U}^{\vartheta,i}_Q
:= 
\big\{X \in\Omega:\, X\approx_{\varepsilon,Q} Y \text{ for some } Y\in U^{\vartheta,i}_Q\big\} \,.
\end{equation}
Given $M\gg 1$ large enough to be chosen and some small enough $\varepsilon>0$, for $Y\in\Omega$, we set 
\[
B_Y:=\overline{B(Y, (1-\varepsilon^{2M/\gamma})\delta(Y))},
\]
where $\gamma\in (0,1)$ is the exponent appearing in Lemma~\ref{lem:Bourgain-CFMS}. We augment  $\widetilde{U}^{\vartheta,i}_Q$ (cf.~\eqref{eq2.3a}) as follows. Set
\begin{equation}\label{Ustar}
\W^{\vartheta,i,*}_{Q}:= \Big\{I\in \W: I^* \text{ meets $B_Y$ for some } Y \in \bigcup_{X\in\widetilde{U}^{\vartheta,i}_Q}
B_X\Big\}
\end{equation}
and 
\begin{equation}\label{Ustar2}
U^{\vartheta,i,*}_{Q} := \bigcup_{I\in\W^{\vartheta,i,*}_{Q}} I^{**}, \qquad U^{*}_{Q} := \bigcup_i U^{\vartheta,i,*}_{Q},
\end{equation}
where we recall that $I^{**}=(1+2\lambda)I$.
By construction,
\[
\widetilde{U}^{\vartheta,i}_Q 
\subset\,
\bigcup_{X\in\widetilde{U}^{\vartheta,i}_Q} B_X\,
\subset
\bigcup_{Y\in\bigcup\limits_{X\in \widetilde{U}^{\vartheta,i}_Q} B_X} B_Y
\subset\, U^{\vartheta,i,*}_{Q}\,,
\]
and for all $Y\in U_{Q}^{\vartheta,i,*}$, we have that $\delta(Y)\approx \ell(Q)$ (depending of course on $\varepsilon$).
Moreover, also by construction, for every
$I\in\W^{\vartheta,i,*}_{Q}$ 
\[
\varepsilon^{s}\,\ell(Q)\lesssim \ell(I) \lesssim \varepsilon^{-3}\,\ell(Q),
\qquad
\dist(I,Q)\lesssim \varepsilon^{-4} \,\ell(Q)\,,
\]
where $s =s(M,\gamma)>0$.

We are now ready to establish that  if $\omega_{-\mathcal{L}}$ admits a semi-coherent corona decomposition,  then $\pom$ is uniformly rectifiable. In view of Theorem~\ref{thm:corona} and Remark~\ref{remark:Bourgain-proof} we may assume that $\omega_{-\mathcal{L}}$ admits a semi-coherent corona decomposition $(\B, \G, \bbF)$  with the property that $Q_\S=\Top(\S)$ and $\omega_{-\mathcal{L}}^{X_\S}(Q_\S)\approx 1$ for each $\S\in\bbF$. Our goal is then to obtain that $\pom$ is uniformly rectifiable, and by \cite[Proposition~1.17]{HLMN} it suffices to see that $\pom$ it satisfies the WHSA property. Let $K_0$ be large enough to be chosen and let $\varepsilon_0 \ll K_0^{-6}$. Fix an arbitrary $\varepsilon \in (0, \varepsilon_0)$, and let $\mathcal{B}$ be the collection of cubes in $\D=\D(\pom)$ for which the $\varepsilon$-local WHSA condition with parameter $K_0$ fails. Then, we are reduced to showing that $\mathcal{B}$ satisfies a Carleson packing condition. With this goal in mind we take an arbitrary cube $Q_0 \in \D$ and write
\begin{equation*}
\sum_{Q \in \mathcal{B} \cap \D_{Q_0}} \sigma(Q)
=
\sum_{Q \in \B \cap \mathcal{B} \cap \D_{Q_0}} \sigma(Q) + 
\sum_{Q \in \G \cap \mathcal{B} \cap \D_{Q_0}} \sigma(Q).
\end{equation*}
Since the bad cubes $\B$ satisfy a packing condition, we have 
\begin{equation*}
\sum_{Q \in \B \cap \mathcal{B} \cap \D_{Q_0}} \sigma(Q)
\le 
\sum_{Q \in \B \cap \D_{Q_0}} \sigma(Q)
\lesssim 
\sigma(Q_0).
\end{equation*}
Much as we did in \eqref{xcndrnbr}--\eqref{eq:bad-good}, and recalling that $Q_\S=\Top(\S)$ one can then see that it suffices to prove 
\begin{align}\label{eq:SB-pack}
\sum_{Q \in \S \cap \mathcal{B} \cap \D_{Q_0}} \sigma(Q) \lesssim \sigma(Q_\S \cap Q_0), \qquad\forall \S \in \bbF.
\end{align}
To this end, we take an arbitrary $\S \in \bbF$ and perform a stopping time argument to extract $\F_\S$, a family of dyadic subcubes of $Q_\S$ which are maximal (hence pairwise-disjoint) with respect to the property $Q\not\in\S$. The semi-coherency of $\S$ clearly implies that $\S  =\D_{\F_\S, Q_\S}$. Set 
\begin{equation} \label{eq:wG-normalize:0}
\mu := \sigma(Q_\S) \omega_{-\mathcal{L}}^{X_\S} 
\quad\text{ and }\quad 
\GG := \sigma(Q_\S) G_{-\mathcal{L}}(X_\S, \cdot), 
\end{equation}
and note that 
\[
\frac{\mu(Q_\S)}{\sigma(Q_\S)}
=
\omega_{-\mathcal{L}}^{X_\S}(Q_\S)\approx 1.
\]
Thus, \eqref{corona-w-strong:Aver-S} becomes
\begin{equation}\label{corona-w-strong:Aver-S:1}
1
\lesssim \frac{\mu(Q)}{\sigma(Q)} 
\lesssim 
\bigg(\fint_Q (\mathcal{M} \mu)^{\frac12} d\sigma \bigg)^2
\lesssim 
1 , \qquad\forall Q \in \D_{\F_\S, Q_\S}. 
\end{equation}
Apply Remark~\ref{remark:gradient:Norefinment}, fixing $N$ large enough. Observing that in this case $\nabla \GG$ is continuous (away from $X_{\S}$), we obtain for every $Q\in\D_{\F_\S, Q_\S}$ with $\ell(Q)\le 2^{-N}\ell(Q_\S)$, that 
 there exists $Y_Q\in 2\widetilde{B}_Q\cap\Omega$ with $\delta(Y_Q)\ge 2^{-N}\ell(Q)$ for which  
\begin{equation}\label{eq:GG*}
|\nabla\GG(Y_Q) |\gtrsim 1 \qquad\text{and}\qquad \frac{\GG(Y_Q)}{\delta(Y_Q)} \gtrsim 1.
\end{equation}
Note also that if $Q\in\S$ is so that $\ell(Q)\le 2^{-N'}\ell(Q_\S)$, with $N'$ large enough, we have
\begin{multline*}
\delta(X_\S)
\le
|X_\S-Y_Q|+\delta(Y_Q)
\le
|X_\S-Y_Q|+ C\,\ell(Q)
\le
|X_\S-Y_Q|+ C\,2^{-N'}\ell(Q_\S)
\\
\le|X_\S-Y_Q|+ C'\,2^{-N'}\delta(X_\S)
<
|X_\S-Y_Q|+\frac12\delta(X_\S),
\end{multline*}
thus
\[
|X_\S-Y_Q|
\ge 
\frac12\delta(X_\S)
\approx
\ell(Q_\S)
\ge
2^{N'}\ell(Q)
\gtrsim
2^{N'}\delta(Y_Q)
>
\frac12 \delta(Y_Q).
\]
We can then invoke \eqref{eq:wG-normalize:0}, \eqref{eq:CFMS:1} and \eqref{corona-w-strong:Aver-S:1} to obtain
\begin{equation}\label{234t2qfgq3}
\frac{\GG(Y_Q)}{\delta(Y_Q)}
\lesssim
\frac{\mu(\Delta(x_Q, C\delta(Y_Q)))}{\delta(Y_Q)^n}
\lesssim
\frac{\mu(\Delta(x_Q, C\ell(Q)))}{\sigma(\Delta(x_Q, C\ell(Q)))}
\le \bigg(\fint_Q (\mathcal{M} \mu)^{\frac12} d\sigma \bigg)^2
\lesssim 
1. 
\end{equation}
This, \eqref{eq:GG*}, the mean value property for harmonic functions, and Caccioppoli's and Harnack's inequalities yield
\begin{equation}\label{eq:GG*:final}
1 
\lesssim 	
|\nabla\GG(Y_Q) |
\lesssim
\frac{\GG(Y_Q)}{\delta(Y_Q)}
\lesssim 
1,
\qquad \text{for each $Q\in\D_{\F_\S, Q_\S}$ with $\ell(Q)\le 2^{-(N+ N')}\ell(Q_\S)$}.
\end{equation}

Since $Y_Q\in\Omega$, there exists $I_Q\in\W$ so that $Y_Q\in I_Q$. Note that $\ell(I_Q)\approx\delta(Y_Q)\approx_N \ell(Q)$ and $\dist(I_Q, Q)\le |Y_Q-x_Q|\lesssim \ell(Q)$. As a result, by taking $\vartheta\ge 10(N+N')$ large enough  we clearly have that $I_Q\in \W_{Q}^\vartheta$, whence $Y_Q\in I_Q\subset U_Q^\vartheta$. From this point  $\vartheta$ remains fixed and we set $K_0=2^\vartheta$. To simplify the notation, in what follows we drop the superindex $\vartheta$ and simply write $U_Q$ with connected components $\{U_Q^{\vartheta,i}\}_i$. Since $Y_Q\in U_Q$, then it belongs to some $U_Q^i$, by relabeling if needed we assume that $Y_Q\in U_Q^0$. 

Note that by construction (see \eqref{Ustar} and \eqref{Ustar2}), there is a Harnack path connecting any pair of points in $U^{0,*}_{Q}$ (depending on $\varepsilon$),
thus, by Harnack's inequality and \eqref{eq:GG*:final}, and recalling that $\varepsilon\ll K_0^{-6}$
\begin{equation}\label{aaaa1}
C^{-1} \delta(Y)  \leq \GG(Y)   \leq C \delta(Y), \qquad \forall\, Y\in U^{0,*}_{Q}, \quad \ell(Q)\le \varepsilon \ell(Q_\S),
\end{equation}
with $C=C(K_0,\varepsilon,M)$ (depending also on the allowable parameters).  Moreover, much as in \eqref{234t2qfgq3},
\begin{equation}\label{aaaa2}
\GG(Y)   \leq C \delta(Y)\,, \qquad \forall\, Y\in U^{*}_{Q}, \quad \ell(Q)\le \varepsilon \ell(Q_\S),
\end{equation}
by  \eqref{eq:wG-normalize:0}, \eqref{eq:CFMS:1},  and   \eqref{corona-w-strong:Aver-S:1}, where again $C=C(K_0,\varepsilon,M)$.

We are now ready to establish \eqref{eq:SB-pack}. We may assume that $\S\cap\D_{Q_0}\neq\emptyset$, in which case the semi-coherency of $\S$ gives $R_0:=Q_\S \cap Q_0\in\S$ and $\S\cap\D_{Q_0}=\D_{\F_\S, R_0}$. At this point we proceed as in \cite[Section 5]{HLMN} consider three cases: 
\begin{list}{$\bullet$}{\leftmargin=1cm \labelwidth=1cm \itemsep=0.1cm \topsep=.2cm}

	\item \textbf{Case 0:} $Q \in \D_{\F_\S, R_0}$ with $\ell(Q)>\varepsilon^{10} \ell(R_0)$. \vspace{0.2cm} 
	
	\item \textbf{Case 1:} $Q \in \D_{\F_\S, R_0}$ with $\ell(Q) \le \varepsilon^{10} \ell(R_0)$ and 
	\begin{align}\label{eq:osc-large}
	\sup_{X \in \widetilde{U}_Q^i} \sup_{Y \in B_X} |\nabla \GG(Y)-\nabla \GG(Y_Q)| > \varepsilon^{2M}. 
	\end{align}
	
	\item \textbf{Case 2:} $Q \in \D_{\F_\S, R_0}$ with $\ell(Q) \le \varepsilon^{10} \ell(R_0)$ and 
	\begin{align}\label{eq:osc-small}
	\sup_{X \in \widetilde{U}_Q^i} \sup_{Y \in B_X} |\nabla \GG(Y)-\nabla \GG(Y_Q)| \le \varepsilon^{2M}. 
	\end{align}
\end{list}
Note that one trivially has
\begin{align}\label{eq:Case-0}
\sum_{\substack{Q \in \D_{\F_\S, R_0} \\ \text{\textbf{Case 0} holds}}} \sigma(Q) 
\le \sum_{\substack{Q \in \D_{R_0} \\ \varepsilon^{10} \ell(R_0) < \ell(Q) \le \ell(R_0)}} \sigma(Q) 
\lesssim (\log \varepsilon^{-1}) \, \sigma(R_0). 
\end{align}
For the cubes in \textbf{Case 1} we can use \eqref{aaaa1} and \eqref{aaaa2} so that  \cite[Lemma~5.8]{HLMN} yields
\begin{align}\label{eq:Case-1}
\sum_{\substack{Q \in \D_{\F_\S, R_0} \\ \text{\textbf{Case 1} holds}}} \sigma(Q) 
\le C(\varepsilon, K_0, M) \, \sigma(R_0). 
\end{align}
Similarly, \cite[Lemma~5.10]{HLMN} implies that if $M$ is large enough then all cubes in \textbf{Case 2} satisfy the $\varepsilon$-local WHSA with parameter $K_0$. All these together eventually imply \eqref{eq:SB-pack}:
\begin{align*}
\sum_{Q \in \S \cap \mathcal{B} \cap \D_{Q_0}} \sigma(Q) 
=\sum_{Q \in \mathcal{B} \cap \D_{\F_\S, R_0}} \sigma(Q) 
= \sum_{\substack{Q \in \mathcal{B} \cap \D_{\F_\S, R_0} \\ \text{\textbf{Case 0} holds}}} \sigma(Q) 
+ \sum_{\substack{Q \in \mathcal{B} \cap \D_{\F_\S, R_0} \\ \text{\textbf{Case 1} holds}}} \sigma(Q) 
\lesssim \sigma(R_0).
\end{align*}
This completes the proof.
\qed 

\begin{remark}\label{remark:4-corner-CME}
	In Theorem~\ref{thm:UR}  the corkscrew condition cannot be removed. Much as in \cite[Example 3, Appendix A]{AHMMT}, consider  $\Omega := \bigcup_{k\ge 0} \Omega_k\subset \re^2$, where for each $k\ge 0$, we let $\Omega_k$ be the $k$-th stage of Garnett's 4-corners construction (see, e.g., \cite[Chapter 1]{DS2}), positioned inside the unit square whose lower left corner is at the point $(2k,0)$ on the $x$-axis. As observed in  \cite[Example 3, Appendix A]{AHMMT} one can easily see that $\Omega$ fails to satisfy the interior corkscrew condition and its boundary is ADR but not UR. However, bounded harmonic functions satisfy (full) Carleson measure estimates. To see this we fix $u$ harmonic in $\Omega$ with $\|u\|_{L^\infty(\Omega)}\le 1$, and let $x_0\in\pom$, $0<r_0<\infty$. Observing that $\Omega$ is comprised of a countable number of open squares, we let $\{R_j\}_j$ be the collection of those such cubes meeting $B_0=B(x_0,r_0)$. One can find  $y_j\in B_0\cap \partial R_j\subset\pom$. Setting $\rho_j:=\min(r_0,\diam(\partial R_j))$, we note that $B_0\cap R_j\subset B(y_j,2\rho_j)\cap R_j$ and 
	\begin{align*}
	\iint_{B_0\cap\Omega} |\nabla u|^2\,\delta\,dX
	=
	\sum_j \iint_{B_0\cap R_j} |\nabla u|^2\,\delta\,dX
	\le
	\sum_j \iint_{B(y_j,2\rho_j)\cap R_j} |\nabla u|^2\,\delta\,dX.
	\end{align*}
	Recall that $u$ is harmonic in $\Omega$ and bounded by 1, hence it satisfies the same properties in $R_j$. Additionally, since $R_j$ is a square (a Lipschitz domain), bounded harmonic functions satisfy Carleson measure estimates, with implicit constant that do not depend on the particular $j$ (the Carleson measure estimates property is scale-invariant and in that regard the $R_j$'s are all equivalent to the unit square). Moreover, the $R_j$'s all have ADR boundaries with uniform constant (again ADR is a scale-invariant property). Hence, we can use that $B(y_j,\rho_j)\subset 2 B_0$ to obtain
	\begin{multline*}
	\iint_{B_0\cap\Omega} |\nabla u|^2\,\delta\,dX
	\lesssim
	\sum_j \rho_j
	\lesssim  \sum_j \H^1(B(y_j, \rho_j)\cap\partial R_j)
	\\
	\le
	\sum_j \H^1( 2 B_0\cap\partial R_j)
	\le
	\H^1( 2 B_0\cap\pom)
	\lesssim
	r_0, 
	\end{multline*}
	where we have employed that sets $\{\partial R_j\}_j$ are pairwise disjoint and contained in $\pom$, and that $\pom$ is also ADR.  We have then shown that bounded harmonic functions satisfy Carleson measure estimates in $\Omega$.  One can produce a similar construction in 3 dimensions. Details are left to the interested reader.
\end{remark}

\subsection{Proof of Theorem \ref{thm:UR} for the  operators in \eqref{thm:UR:a}}
The argument is very similar to the one for the Laplacian but it requires some changes. 
First of all by Theorem~\ref{thm:LLT} we can reduce matters to the case on which $L$ is symmetric. Writing $K:=\big\||\nabla A|\,\delta(\cdot)\big\|_{L^\infty(\Omega)}<\infty$ we have by \cite[Lemma~4.39]{HMT1} that  $|\nabla u(X)|\delta(X)\lesssim_K u(X)$ for every $X\in\Omega$ and for every $0\le u\in W_{\loc}^{1,2}(\Omega)$ so that $Lu=0$ in $\Omega$ in the weak sense. Additionally, if $0\le u\in W^{1,2}(6I)$ is so that $Lu=0$ in $6I$, with $I$ being an $(n+1)$-dimensional cube with $6I\subset\Omega$ then
$\|\nabla^2 u\|_{L^2(I)}\lesssim_K \ell(I)^{-1} \|\nabla u\|_{L^2(2I)}$ (\cite[Lemma~4.39]{HMT1}) and 
\begin{equation}\label{wvvwev}
|\nabla u(X)-\nabla u(Y)|\lesssim_K \left(\frac{|X-Y|}{\ell(I)}\right)^\frac12 \sup_{2I} |\nabla u|,
\qquad X,Y\in I.
\end{equation}
The latter estimate can be derived from \cite[Theorem~5.19]{GM} applied in the unit cube and by translation and rescaling much as in \cite[Lemma~4.39]{HMT1}. With this estimate we can proceed as in \cite[Section~5]{HLMN} and much as before there are three cases. The cubes in \textbf{Case 0} are trivial as before. For \textbf{Case 2} one has to inspect the proof of \cite[Lemma~5.10]{HLMN} 
and see that the argument follows \textit{mutatis mutandis} in  account of the previous estimates. 
For the cubes in \textbf{Case 1} we need to make some changes. Rather than using the argument in \cite[Section~5A]{HLMN} we initially follow \cite[Section~5B]{HLMN}, 
using \eqref{wvvwev} in place of \cite[(3.38)]{HLMN}, up to the \cite[(5.35)]{HLMN} with $p=2$ and $F\gamma\equiv 1$, that is,
\[
\sum_{\substack{Q \in \D_{\F_\S, R_0} \\ \text{\textbf{Case 1} holds}}} \sigma(Q) 
\lesssim 
\iint_{\Omega_{\F_\S, R_0}^*} |\nabla^2 \GG|^2\,\GG\,dX. 
\] 
From this point we integrate by parts as in \cite{HMT1}  (see also Section~\ref{sec:G-CME}) using that we have assumed that $L$ is symmetric. With that one can obtain the desired estimate for the cubes in \textbf{Case 1} and the proof is then complete. Further details are left to the reader. \qed

\subsection{Proof of Theorem \ref{thm:UR} for the  operators in \eqref{thm:UR:b}}
Much as in \cite[Corollary~10.3]{HMMTZ} one can regularize $A$ so that the new matrix $\widetilde{A}$ is one of the operators considered in \eqref{thm:UR:a} and, moreover, $\widetilde{A}$ is a Fefferman-Kenig-Pipher perturbation of $A$. Thus, Theorem~\ref{thm:LL} and the fact that we have already taken care of the operators in \eqref{thm:UR:a} give the desired equivalences. \qed

\subsection{Proof of Theorem \ref{thm:UR} for the  operators in \eqref{thm:UR:c}}
This follows at once from Theorem~\ref{thm:LL} and the previous cases. \qed

\section{Proof of Theorem \ref{thm:connected}}\label{sec:1-sided-CAD}

We introduce some notation. Let  $\Omega \subset \ree$, $n \ge 2$, be an open set with ADR boundary and satisfying the corkscrew condition. We say that   $\omega_L\in A^{\rm weak}_\infty(\sigma)$, if there are a positive constant $C>1$  and an exponent $q>1$, such that for every surface ball $\Delta=\Delta(x,r)$,
with $x\in \pom$ and $0<r<\diam(\pom)$, there exists $X_{\Delta}\in  B(x,r)\cap\Omega$, with $\dist(X_\Delta,\pom)\ge C^{-1} r$, satisfying
$\omega_L^{X_\Delta}\ll \sigma$ in  $2\Delta$, 	and $k_L^{X_\Delta}:=\frac{d\omega_L^{X_\Delta}}{d\sigma}$ verifies
\begin{equation}\label{eqn:main-SI}
\int_ {2\Delta} 
k_L^{X_\Delta}(y)^q\,d\sigma(y)
\le
C\,   \sigma(\Delta)^{1-q}.
\end{equation}

We say that an open set $\Omega$ satisfies the Harnack chain condition if there is a uniform constant $C$ such that for every 
	$\rho>0$, $\Lambda \geq 1$, and every pair of points $X, X' \in \Omega$ with $\min\{\delta(X), \delta(X')\} \geq \rho$ 
	and $|X-X'| < \Lambda \rho$, there is a chain of open balls $B_1, \ldots, B_N \subset \Omega$, $N \leq C(\Lambda)$, 
	with $X \in B_1$, $X' \in B_N$, $B_k \cap B_{k+1} \neq \emptyset$, 
	$C^{-1} \diam(B_k) \leq \dist(B_k, \partial \Omega) \leq C \diam(B_k)$. 
	We observe that the Harnack chain condition is a scale-invariant version of path connectedness.  

We say that $\Omega$ is a $1$-sided NTA (non-tangentially accessible) domain if  $\Omega$ satisfies both the corkscrew and Harnack chain conditions. Furthermore, we say that $\Omega$ is an NTA domain if it is a $1$-sided NTA 	domain and if, in addition, $\R^{n+1} \setminus \overline{\Omega}$ satisfies the corkscrew condition.  	If a $1$-sided NTA domain, or an NTA domain, has an ADR boundary, then it is called a 1-sided CAD (chord-arc domain) or a CAD, respectively. 

In 1-sided CAD domains the elliptic measure is doubling (see e.g. \cite{HMT2}), hence $\omega_L\in A^{\rm weak}_\infty(\sigma)$ becomes $\omega_L\in A_\infty(\sigma)$, condition that can be equivalently written as follows:  there exist constants $0<\alpha, \beta<1$ such that given an arbitrary surface ball $\Delta_0 :=B_0 \cap \pom$, with $B_0 :=B(x_0, r_0)$, $x_0 \in \pom$, $0<r_0<\diam(\pom)$, and for every surface ball $\Delta :=B \cap \pom$ centered at $\pom$ with $B \subset B_0$, and for every Borel set $F \subset \Delta$, we have that 
	\begin{align*}
	\frac{\omega_L^{X_{\Delta_0}}(F)}{\omega_L^{X_{\Delta_0}}(\Delta)} \le \alpha 
	\quad\Longrightarrow\quad \frac{\sigma(F)}{\sigma(\Delta)} \le \beta. 
	\end{align*}

\subsection{Proof of Theorem \ref{thm:connected}, part \eqref{1-sidedCAD:i}}\label{sec:corona-Ainfty:i} 

As discussed in \cite[Section~4]{HLMN}, by slightly changing the constant $C$ in the definition of $\omega_L\in A^{\rm weak}_\infty(\sigma)$ we may assume that $\omega_L^{X_\Delta}(\Delta)\ge C_1^{-1}$ for some constants $C_1>1$ on $n$, ADR, and ellipticity. In turn, it was also shown that there exists $\beta, \eta \in (0, 1)$ so that for every $Q\in\D=\D(\pom)$  one can find $X_{Q}\in  B_{Q}\cap\Omega$ with $\delta(X_{Q})\gtrsim \ell(Q)$ for which $\omega_L^{X_Q}(Q)\ge C_1^{-1}$ and 
\begin{equation}\label{erwfaev}
	\sigma(A) \geq (1-\eta) \sigma(Q) \implies \omega_L^{X_{Q}}(A) \geq \beta \omega_L^{X_Q}(Q),
	\qquad 
	\text{for any Borel set } A\subset Q.
\end{equation}

We need the following auxiliary result: 

\begin{lemma}[{\cite[Lemma~4.12]{HLMN}}]\label{lem:stopping}
	Let $Q \in \D$ and let $\mu$ be a Borel measure on $\pom$. Assume that there exist $K_1 \ge 1$ and $\beta, \eta \in (0, 1)$ such that 
	\begin{equation*}
		1 \le \frac{\mu(Q)}{\sigma(Q)} \le \frac{\mu(\pom)}{\sigma(Q)} \le  K_1,  
	\end{equation*} 
	and for any Borel set $A \subset Q$, 
	\begin{equation*}
		\sigma(A) \geq (1-\eta) \sigma(Q) \quad\implies\quad \mu(A) \geq \beta \mu(Q).
	\end{equation*}
	Then there exists a pairwise disjoint family $\F_Q = \{Q_j\}_j \subset \D_{Q}\setminus\{Q\}$ such that
	\begin{equation} \label{eq:stop-1}
		\sigma \Big(Q \setminus \bigcup_{Q_j \in \F_Q} Q_j \Big) \geq (1-\alpha) \sigma(Q), 
	\end{equation}
	and 
	\begin{equation} \label{eq:stop-2}
		\frac \beta 2 \leq \frac{\mu(Q')}{\sigma(Q')} 
		\leq \Big(\fint_{Q'} (\mathcal{M} \mu)^{\frac12} d\sigma\Big)^2 \le K_2, \qquad\forall \, Q' \in \D_{\F_Q, Q}, 
	\end{equation}
	where $\alpha \in (0, 1)$ and $K_2>1$ depend only on $n$, $\beta$, $\eta$, $K_1$, and the ADR constant. 
\end{lemma}

We are now ready to prove that $\omega_L \in A_{\infty}^{\text{weak}}(\sigma)$ implies that $\omega_L$ admits a strong corona decomposition. 	By Proposition \ref{pro:global}, it suffices to obtain a strong corona decomposition for $\omega_L$ on $\D_{Q_0}$ any cube $Q_0 \in \D(\pom)$. Fix then $Q_0 \in \D(\pom)$ and we will construct the corona decomposition by iterating Lemma~\ref{lem:stopping}. The $0$-th generation cubes are constructed as follows.  As observed above $\omega_L^{X_{Q_0}}(Q_0) \geq C_1^{-1}$ for some $C_1 \ge 1$. Write $\mu := C_1 \sigma(Q_0) \omega_L^{X_{Q_0}}$ so that 
	\begin{equation}\label{eq:mu_over_sigma}
		1 
		\le 
		C_1\,\omega_L^{X_{Q_0}}(Q_0)
		=
		\frac{\mu(Q_0)}{\sigma(Q_0)} 
		\le
		\frac{\mu(\pom)}{\sigma(Q_0)} 
		=
		C_1\,\omega_L^{X_{Q_0}}(\pom)
		\leq 
		C_1.
	\end{equation}
	Using this and \eqref{erwfaev} we can invoke Lemma~\ref{lem:stopping} and find a pairwise family $\F_{Q_0}\subset \D_{Q_0}\setminus\{Q_0\}$ such that
	\eqref{eq:stop-1} and \eqref{eq:stop-2} hold with $Q=Q_0$.  Let $\S_{Q_0}:=\D_{\F_{Q_0}, Q_0}$, which is semi-coherent with $\Top(\S_{Q_0})=Q_0$. Set $Q_{\S_{Q_0}}:=\Top(\S_{Q_0})=Q_0$ and $X_{\S_{Q_0}}:=X_{Q_0}$ so that \eqref{corona-w-strong:CKS-S} holds for $\S_{Q_0}$. Observe that 
	\eqref{eq:stop-2} and the fact that  $C_1^{-1}\le \omega_L^{X_{Q_0}}(Q_0) \le 1$ give for every  $Q \in \S_{Q_0}$ 
	\begin{multline*}
		\frac{\beta}{2C_1}\, \frac{\omega_L^{X_{\S_{Q_0}}}(Q_{\S_{Q_0}})}{\sigma(Q_{\S_{Q_0}})}
		\le
		\frac{\beta}{2C_1\sigma(Q_{0})} 
		\le 
		\frac{\omega_L^{X_{\S_{Q_0}}}(Q)}{\sigma(Q)}
		\\
		\leq 
		\Big(\fint_{Q} (\mathcal{M} \omega_L^{X_{\S_{Q_0}}})^{\frac12} d\sigma\Big)^2 
		\le 
		\frac{K_2}{C_1\sigma(Q_{0})} 
		\le
		K_2\, \frac{\omega_L^{X_{\S_{Q_0}}}(Q_{\S_{Q_0}})}{\sigma(Q_{\S_{Q_0}})},
	\end{multline*}
	that is, \eqref{corona-w-strong:Aver-S} holds for $\S_{Q_0}$. Defining $\bbF^{0}:=\{\S_{Q_0}\}$, $\G^0:=\S_{Q_0}$, and $\F^0=\F_{Q_0}$ one can clearly obtain 
	\[
	\D_{Q_0}
	=
	\G^0 \bigsqcup \Big(\bigsqcup_{Q\in \F^0} \D_{Q}\Big)
	\]
	and, by  \eqref{eq:stop-1}, 
	\begin{align*}
		\sigma \bigg(\bigsqcup_{Q \in \F^0} Q \bigg) 
		\le 
		\alpha \sigma(Q_0). 
	\end{align*}
	
	We now iterate, repeating this process for any $Q'\in \F^0$. We then obtain a pairwise family $\F_{Q'}\subset\D_{Q'}\setminus\{Q'\}$. Let $\S_{Q'}:=\D_{\F_{Q'}, Q'}$ which is semi-coherent with $\Top(\S_{Q'})=Q'$. Set $Q_{\S_{Q'}}:=\Top(\S_{Q'})=Q'$ and $X_{\S_{Q'}}=X_{Q'}$ so that \eqref{corona-w-strong:CKS-S} and \eqref{corona-w-strong:Aver-S} hold for $\S_{Q'}$. Defining $\bbF^{1}:=\{\S_{Q'}:  Q'\in \F^0\}$, $\G^1:=\bigsqcup_{\S\in\bbF^1}\S$, and $\F^1:=\{\F_{Q'}: Q'\in\F^0\}$ we easily see that 
	\[
	\D_{Q_0}
	=
	\G^0 \bigsqcup \Big(\bigsqcup_{Q\in \F^0} \D_{Q}\Big)
	=
	\G^0 \bigsqcup \G^1 \bigsqcup \Big(\bigsqcup_{Q\in \F^1} \D_{Q}\Big)
	\]
	and 
	\begin{align*}
		\sigma \bigg(\bigsqcup_{Q \in \F^1} Q \bigg) 
		=
		\sum_{Q'\in \F^0} \sigma \bigg(\bigsqcup_{Q \in \F_{Q'}} Q \bigg) 
		\le 
		\alpha
		\sum_{Q'\in \F^0} \sigma (Q')
		=
		\alpha
		\sigma \bigg(\bigsqcup_{Q' \in \F^0} Q' \bigg) 
		\le 
		\alpha^2 \sigma(Q_0). 
	\end{align*}
	We now iterate this argument with the cubes of $\F^1$ and so forth so on. We then define $\bbF=\bigsqcup_{j=0}^\infty \bbF^j$, $\G=\bigsqcup_{j=0}^\infty \G^j$, and $\F=\bigsqcup_{j=0}^\infty \F^j$. Note that by construction $\G=\bigsqcup_{\S\in \bbF} \S$, where each $\S$ is semi-coherent. For each $\S\in\bbF$ we have that $Q_\S=\Top(\S)$, $X_{\S}=X_{Q_\S}$, and that \eqref{corona-w-strong:CKS-S} and \eqref{corona-w-strong:Aver-S} hold. It is also clear from the construction that $\Top(\bbF)=\F\sqcup\{Q_0\}$. We next show that $\D\setminus\G=\emptyset$. Assume otherwise that the exists $Q'\in \D\setminus\G=\emptyset$. Since $0<\alpha<1$, we can find $k\ge 0$ so that $\alpha^k\sigma(Q_0)<\sigma(Q')$. Note that 
	\[
	\D_{Q_0}
	=
	\Big(\bigsqcup_{j=0}^k \G^j\Big)\bigsqcup \Big(\bigsqcup_{Q\in \F^k} \D_{Q}\Big),
	\]
	hence $Q'\in\D_Q$ for some $Q\in \F^k$. This gives a contradiction:
	\[
	\alpha^k\sigma(Q_0)
	<
	\sigma(Q')
	\le 
	\sigma \bigg(\bigsqcup_{Q \in \F^k} Q \bigg) 
	\le
	\alpha^{k+1}\sigma(Q_0)<
	\alpha^{k}\sigma(Q_0).
	\]
	We next show that the cubes in $\Top(\bbF)=\F\sqcup\{Q_0\}$ satisfy a Carleson packing condition. Assuming this momentarily, and if we set $\B=\emptyset$ we have then shown that $(\B,\G, \bbF)$ is a semi-coherent corona decomposition so that  \eqref{corona-w-strong:CKS-S} and \eqref{corona-w-strong:Aver-S} hold, which completes the proof.

	Let us then show that $\Top(\bbF)=\F\sqcup\{Q_0\}$ satisfy a Carleson packing condition. Take an arbitrary $Q_0'\in \Top(\bbF)$, that is, $Q_0'\in \F^{j_0}$ for some $j_0$. Let $j\ge j_0$ and note that $\F^{j+1}=\{\F_{Q'}:Q'\in \F^{j}\}$, hence
	\begin{multline*}
		\Sigma_{j+1}
		:=
		\sum_{Q\in\F^{j+1}\cap\D_{Q_0'}} \sigma(Q)
		=
		\sum_{Q'\in\F^j} \sum_{Q\in\F_{Q'}\cap\D_{Q_0'}} \sigma(Q)
		\\
		=
		\sum_{Q'\in\F^j\cap\D_{Q_0'}} \sigma \bigg(\bigsqcup_{Q \in \F_{Q'}} Q \bigg) 
		\le
		\alpha \sum_{Q'\in\F^j\cap\D_{Q_0'}} \sigma(Q')
		=
		\alpha \Sigma_{j}.
	\end{multline*}
	Iterating this we obtain $\Sigma_{j}\le \alpha^{j-j_0}\,\Sigma_{j_0}=\alpha^{j-j_0}\,\sigma(Q_0')$, for every $j\ge j_0$. As a result,
	\begin{align}\label{qvavcv}
		\sum_{Q\in \Top(\bbF)\cap \D_{Q_0'}} \sigma(Q) 
		=
		\sum_{j=j_0}^\infty \Sigma_j
		\le
		\sigma(Q_0')\sum_{j=j_0}^\infty \alpha^{j-j_0}
		=
		(1-\alpha)^{-1}\,\sigma(Q_0'),
	\end{align}
	which is the desired estimate for any arbitrary cube $Q'_0\in  \Top(\bbF)$.

	Consider next the general case $Q\in\D_{Q_0}$ and let $\{Q_k\}_k$ be the collection of maximal cubes (hence pairwise disjoint) in $\Top(\bbF)$ contained in $Q$. Then, using \eqref{qvavcv} with $Q_k$
	\begin{multline*}
		\sum_{Q'\in \Top(\bbF)\cap \D_{Q}} \sigma(Q') 
		=
		\sum_k \sum_{Q'\in \Top(\bbF)\cap \D_{Q_k}} \sigma(Q') 
		\le
		(1-\alpha)^{-1}\,\sum_k \sigma(Q_k)
		\\
		=
		(1-\alpha)^{-1}\,\sigma\Big(\bigsqcup_k Q_k\Big)
		\le
		(1-\alpha)^{-1}\sigma(Q),
	\end{multline*}
	and this shows that $\Top(\bbF)=\F\sqcup\{Q_0\}$ satisfy a Carleson packing condition, as desired. \qed

\subsection{Proof of Theorem \ref{thm:connected}, part \eqref{1-sidedCAD:ii}} \label{sec:corona-Ainfty:ii}
We assume that $\Omega$ is a 1-sided CAD. Part \eqref{1-sidedCAD:i} shows that $\eqref{list:con-Ainfty}\Longrightarrow\eqref{list:con-corona}$,  since $\omega_L$ is doubling in 1-sided CAD. Below we show that $\eqref{list:con-corona}\Longrightarrow\eqref{list:con-Ainfty}$. Assuming this momentarily, \cite[Theorem~1.1]{CHMT} gives that   $\eqref{list:con-Ainfty}\Longleftrightarrow \eqref{list:con-full-CME}$ and Theorem \ref{thm:corona} yields the other equivalences.

We introduce some notation. Let $E\subset\ree$ be a an ADR set and let $\D=\D(E)$ be its associated family of dyadic cubes. Given sequence of non-negative numbers $\{\alpha_Q\}_{Q \in \D}$, we define the ``measure'' $\m$ (acting on collection of dyadic cubes) by 
\begin{equation}\label{eq:m-def}
	\m(\D') := \sum_{Q \in \D'} \alpha_Q,  \qquad \text{for } \D' \subset \D.
\end{equation} 
For a fixed $Q_0 \in \D$, we say that $\m$ is a discrete ``Carleson measure'' on $\D_{Q_0}$ 
(with respect to $\sigma$), and write $\m \in \C(Q_0)$, if  
\begin{equation}
	\|\m\|_{\C(Q_0)} :=\sup _{Q \in \D_{Q_0}} \frac{\m(\D_Q)}{\sigma(Q)} < \infty. 
\end{equation} 
We also set 
\begin{equation}
	\|\m\|_{\C} :=\sup _{Q \in \D} \frac{\m(\D_Q)}{\sigma(Q)} < \infty 
\end{equation} 
to denote the  global Carleson norm on $\D$. Given a family $\F \subset \D$ of pairwise disjoint cubes, we define the restriction of $\m$ to the sawtooth $\D_{\F}$ by 
\begin{equation}
	\m_{\F} (\D') :=\m(\D' \cap \D_{\F}) = \sum_{Q \in \D' \setminus (\bigcup\limits_{Q_j \in \F} \D_{Q_j})} \alpha_{Q}. 
\end{equation} 
For a pairwise disjoint family $\F=\{Q_j\} \subset \D$ and a non-negative Borel measure $\mu$ on $\pom$ we define the projection  measure $\P_{\F} \mu$ as 
\begin{align*}
	\P_{\F} \mu(A) := \mu \Big(A \setminus \bigcup_{Q_j \in \F} Q_j \Big) 
	+ \sum_{Q_j \in \F} \frac{\sigma(A \cap Q_j)}{\sigma(Q_j)} \mu(Q_j),  
\end{align*}
for any Borel set $A \subset \pom$. 

\begin{lemma}[{\cite[Lemma 8.5]{HM}}]\label{lem:extrapolation}
	Suppose that $E\subset\ree$ is an ADR set. Fix $Q^0 \in \D(E)$, let $\sigma$ and $\w$ be a pair of dyadically doubling Borel measures on $Q^0$, and let $\m$ be a discrete Carleson measure with respect to $\sigma$ with
	\begin{align}\label{eq:m-CQ}
		\|\m\|_{\C(Q^0)} \leq M_0<\infty. 
	\end{align} 
	Suppose that there esixts $\gamma>0$ such that for every $Q_0 \in \D_{Q^0}$ and every family of pairwise disjoint dyadic cubes $\F=\{Q_j\} \subset \D_{Q_0}$ verifying 
	\begin{align}\label{eq:mF-small}
		\|\m_{\F}\|_{\C(Q_0)} \leq \gamma,  
	\end{align} 
	we have that $\P_{\F} \w$ satisfies the following property:
	\begin{equation}\label{eq:PF-Ainfty}
		\forall \varepsilon \in(0,1),\ \exists\, C_{\varepsilon}>1 \text { such that } 
		\bigg(F \subset Q_0, \frac{\sigma(F)}{\sigma(Q_0)} \geq \varepsilon \, \Longrightarrow \, 
		\frac{\P_{\F} \, \omega(F)}{\P_{\F} \, \omega(Q_0)} \geq \frac{1}{C_{\varepsilon}}\bigg). 
	\end{equation} 
	Then, there exist $\alpha, \beta \in (0,1)$ such that for every $Q_0 \in \D_{Q^0}$,  
	\begin{equation}
		F \subset Q_0, \quad \frac{\sigma(F)}{\sigma(Q_0)} \geq \alpha \quad \Longrightarrow \quad 
		\frac{\omega(F)}{\omega(Q_0)} \geq \beta, 
	\end{equation} 
	that is, $\w \in A_{\infty}^{\rm{dyadic}}(Q^0,\sigma)$.
\end{lemma}

We are now ready to prove that $\eqref{list:con-corona}\Longrightarrow\eqref{list:con-Ainfty}$. A more restrictive version of this result appears in \cite[Section~4.2]{GMT} using different methods.  Let $(\B,\G,\bbF)$ be the assumed corona decomposition associated with $\omega_L$. Set for any $Q \in \D$ 
	\begin{equation}\label{eq:aQ}
		\alpha_Q := 
		\begin{cases}
			\sigma(Q), & Q \in \B \cup \Top(\bbF), 
			\\
			0, & \text{otherwise}. 
		\end{cases}
	\end{equation}
	We define the associated discrete measure $\m$ as in \eqref{eq:m-def}. Fix $Q^0 \in \D$ and write $\w:=\omega_L^{X_{Q^0}}$. Note that both $\w$ and $\sigma$ are dyadically doubling on $Q^0$, since $\Omega$ is a 1-sided CAD (see for instance \cite{HMT2}). In view of the Carleson packing condition for bad and top cubes, we have 
	\[
	\|\m\|_{\C(Q^0)} \le \|\m\|_{\C} \le M_0<\infty. 
	\]  
	Therefore, Lemma \ref{lem:extrapolation} implies that $\omega_L^{X_{Q^0}} \in A^{{\rm dyadic}}_{\infty}(Q^0)$ provided we show that  \eqref{eq:mF-small} for some small $\gamma \in (0, 1)$ to be found  implies \eqref{eq:PF-Ainfty}. Since $Q^0$ is arbitrary and the same constants $\alpha$ and $\beta$ are valid for all $Q^0$, the fact that the elliptic measure is doubling in the present scenario easily yields that $\omega\in A_\infty(\sigma)$, details are left to the interested reader.

	To complete the proof, it suffices to show that \eqref{eq:mF-small}  implies \eqref{eq:PF-Ainfty}. To this end, fix $Q_0 \in \D_{Q^0}$ and a family of pairwise disjoint dyadic subcubes $\F=\{Q_j\} \subset \D_{Q_0}$ satisfying \eqref{eq:mF-small} with $\gamma \in (0, 1)$ to be chosen momentarily. Let $F \subset Q_0$ be an arbitrary  Borel set.
	
	The case $\F=\{Q_0\}$ is trivial since 
	$\frac{\P_{\F} \, \w(F)}{\P_{\F} \, \w(Q_0)} = \frac{\sigma(F)}{\sigma(Q_0)}$, and \eqref{eq:PF-Ainfty} clearly holds. Thus we may assume that $\F \subset \D_{Q_0} \setminus \{Q_0\}$. Write $E_0:=Q_0\setminus (\bigcup_{Q_j\in\F} Q_j)$. We claim that there exists $\S_0 \in \G$ such that $Q_0 \subsetneq \Top(\S_0)$, $\D_{\F, Q_0} \subset \S_0$, and 
	\begin{align}\label{eq:wQQ}
		\frac{\w(Q)}{\sigma(Q)} \approx \frac{\w(Q^0 \cap \Top(\S_0))}{\sigma(Q^0 \cap \Top(\S_0))} =: \Lambda_0, \quad \forall Q \in \D_{\F, Q_0}. 
	\end{align}
	Assuming \eqref{eq:wQQ} momentarily, we conclude \eqref{eq:PF-Ainfty} as follows. By definition, 
	\begin{align}\label{eq:P1-P2}
		\P_{\F} \w(F) = \w (F\cap E_0)
		+ \sum_{Q_j \in \F} \frac{\w(Q_j)}{\sigma(Q_j)} \sigma(F \cap Q_j) : = \mathcal{I} + \mathcal{II}.  
	\end{align}
	We first deal with the second term. If we write $\widehat{Q}_j$ for the dyadic parent of $Q_j\in\F $, it is clear than 
	$\widehat{Q}_j \in \D_{\F, Q_0}$. Since both $\w$ and $\sigma$ are dyadically doubling, it follows from \eqref{eq:wQQ} that 
	\begin{align}\label{eq:P2}
		\mathcal{II} \approx \sum_{Q_j \in \F} \frac{\w(\widehat{Q}_j)}{\sigma(\widehat{Q}_j)} \sigma(F \cap Q_j) 
		\approx \Lambda_0 \sum_{Q_j \in \F} \sigma(F \cap Q_j) 
		= \Lambda_0 \, \sigma (F\setminus E_0) . 
	\end{align}
	To bound the first term, let $\eta>0$ be an arbitrary number. Since $\w$ is a regular Borel measure, one can find an open set $U_{\eta} \supset F\cap E_0$ such that $\w(U_{\eta} \setminus (F\cap E_0))<\eta$. Let $x \in F\cap E_0 \subset U_{\eta}$. Then $\Delta(x, r_x) \subset U_{\eta}$ for some $0<r_x<\infty$. Pick $Q_x \in \D_{Q_0}$ containing $x$ such that $\diam(Q_x)<r_x$ and $\ell(Q_x)<\ell(Q_0)$. This and the fact $x \in Q_0$ give $Q_x \subset \Delta(x, r_x) \cap Q_0 \subset U_{\eta} \cap Q_0$. Let $Q_x^{\text{max}} \in \D_{Q_0}$ be the maximal cube with $Q_x^{\text{max}} \supset Q_x$ such that $Q_x^{\text{max}} \subset U_{\eta} \cap Q_0$. Denote by $\widetilde{\F}$ be the collection of the maximal cubes $Q_x^{\text{max}}$ for $x \in F\cap E_0$. Note that $\widetilde{\F}$ is pairwise disjoint, 
	\begin{align}\label{eq:FFF}
		F\cap E_0 \subset \bigcup_{Q \in \widetilde{\F}} Q \subset U_{\eta} \cap Q_0, \quad\text{ and }\quad  
		F\cap E_0 \cap Q \neq \emptyset, \quad \forall Q \in \widetilde{\F}.
	\end{align}
	On the other hand, 
	\begin{equation}\label{eq:QFF}
		\widetilde{\F}\subset \D_{\F, Q_0}. 
	\end{equation}
	Otherwise, there are $Q \in \widetilde{\F}$ and $Q_j \in \F$ such that $Q \subset Q_j$. Then $F\cap E_0 \cap Q \subset E_0 \cap Q_j=\emptyset$, which contradicts \eqref{eq:FFF}. Then, we use \eqref{eq:wQQ}, \eqref{eq:FFF} and \eqref{eq:QFF} to obtain 
	\begin{multline*}
		\Lambda_0 \sigma(F\cap E_0) \le \sum_{Q \in \widetilde{\F}} \Lambda_0 \sigma(Q) 
		\approx \sum_{Q \in \widetilde{\F}} \w(Q) 
		\le \w(U_{\eta} \cap Q_0)
		\\
		\le \w(U_{\eta} \setminus (F\cap E_0))+\w(F\cap E_0)
		\le \eta + \w(F\cap E_0),  
	\end{multline*}
	where the implicit constant is independent of $\eta$. Letting $\eta \to 0$, we arrive at
	\begin{align}\label{eq:P1}
		\mathcal{I}=\w(F\cap E_0) \gtrsim \Lambda_0 \, \sigma(F\cap E_0). 
	\end{align}
	Now collecting \eqref{eq:P1-P2}, \eqref{eq:P2}, and \eqref{eq:P1} we get 
	\begin{align*}
		\P_{\F} \w(F) \gtrsim \Lambda_0 \, \sigma(F).  
	\end{align*}
	This and the fact that $\P_{\F} \w(Q_0)=\w(Q_0) \approx \Lambda_0 \sigma(Q_0)$ easily yield
	\begin{align}
		\frac{\P_{\F} \, \w(F)}{\P_{\F} \, \w(Q_0)} \gtrsim \frac{\sigma(F)}{\sigma(Q_0)},  
	\end{align}
	which immediately gives \eqref{eq:PF-Ainfty}. 
	
	It remains to show \eqref{eq:wQQ}. Note that if $Q \in \D_{\F, Q_0}$, then    
	\[
	\alpha_{Q} \le \sum_{Q' \in \D_{\F, Q}} \alpha_{Q'} 
	\le \|\m_{\F}\|_{\C(Q_0)} \sigma(Q) \le \gamma \, \sigma(Q). 
	\]
	By \eqref{eq:aQ}, if we simply take $\gamma=\frac12$ definition, we conclude that  $Q \not\in \B \cup \Top(\bbF)$. That is, our choice of $\gamma$ guarantees that $\D_{\F, Q_0}\subset \D\setminus (\B \cup \Top(\bbF))$. This and the semi-coherency of the regimes in $\bbF$ immediately imply that there exists a unique $\S_0 \in \G$ such that $\D_{\F, Q_0}\subset \S_0$ and  $Q_0 \subsetneq \Top(\S_0)$. 
	Note that $Q^0\cap \Top(\S_0)\in \S_0$ since $\S_0$ is semi-coherent and $\S_0\ni Q_0\subset Q^0\cap \Top(\S_0) \subset \Top(\S_0)$.
	Thus, \eqref{corona-w:Aver-S} implies
	\begin{align}\label{fcevca}
		\frac{\omega_L^{X_{\S_0}}(Q)}{\sigma(Q)} 
		\approx 
		\frac{\omega_L^{X_{\S_0}}(Q_{\S_0})}{\sigma(Q_{\S_0})} 
		\approx 
		\frac{\omega_L^{X_{\S_0}}(Q^0\cap \Top(\S_0))}{\sigma(Q^0\cap\Top(\S_0))} 
		, \qquad \forall Q \in \D_{\F, Q_0}.  
	\end{align}
	Note that $Q_0 \subsetneq \Top(\S_0)\cap Q^0\subset Q_{\S_0}\cap Q^0$, hence $Q_{\S_0}\cap Q^0\neq\emptyset$. 
	If $Q^0 \subset Q_{\S_0}$, then the change of pole formula gives
	\[
	\frac{d\omega_L^{X_{Q^0}}}{d\omega_L^{X_{\S_0}}} (y)
	\approx
	\frac1{\omega_L^{X_{\S_0}}(Q^0)},
	\qquad\text{for $\omega_L^{X_{\S_0}}$-a.e.~$y\in Q^0$}.
	\]
	As a result, if $Q, Q'\subset Q^0$ we obtain
	\[
	\frac{\omega(Q)}{\omega(Q')}
	=
	\frac{\omega_L^{X_{Q^0}}(Q)}{\omega_L^{X_{Q^0}}(Q')}
	\approx
	\frac{\omega_L^{X_{\S_0}}(Q)}{\omega_L^{X_{\S_0}}(Q')}. 
	\]
	Analogously,  if 
	$Q_{\S_0}\subset Q^0$, then the change of pole formula gives
	\[
	\frac{d\omega_L^{X_{\S_0}}}{d\omega_L^{X_{Q^0}}} (y)
	\approx
	\frac1{\omega_L^{X_{Q^0}}(Q_{\S_0})},
	\qquad\text{for $\omega_L^{X_{Q^0}}$-a.e.~$y\in Q_{\S_0}$}. 
	\]
	Thus,  if $Q, Q'\subset Q_{\S_0}$, then 
	\[
	\frac{\omega_L^{X_{\S_0}}(Q)}{\omega_L^{X_{\S_0}}(Q')}
	\approx
	\frac{\omega_L^{X_{Q^0}}(Q)}{\omega_L^{X_{Q^0}}(Q')}
	=
	\frac{\omega(Q)}{\omega(Q')}. 
	\]
	This means that in either scenario if $Q, Q'\subset Q_{\S_0}\cap Q^0$ we obtain
	\[
	\frac{\omega(Q)}{\omega(Q')}
	\approx
	\frac{\omega_L^{X_{\S_0}}(Q)}{\omega_L^{X_{\S_0}}(Q')}.
	\]
	Recalling that $\D_{\F, Q_0}\subset \S_0\cap \D_{Q_0} \subset\D_{Q_{\S_0}}\cap \D_{Q^0}$, the previous estimate and  \eqref{fcevca} give as desired \eqref{eq:wQQ}:
\begin{align*}
\frac{\w(Q)}{\w(Q^0\cap \Top(\S_0))} 
\approx 	
\frac{\omega_L^{X_{\S_0}}(Q)}{\w^{X_{\S_0}}(Q^0\cap \Top(\S_0))} 
\approx
\frac{\sigma(Q)}{\sigma(Q^0\cap\Top(\S_0))} 
, \qquad \forall Q \in \D_{\F, Q_0}.  
\end{align*}
	The proof is then complete. \qed
	
\subsection{Proof of Theorem \ref{thm:connected}, part \eqref{1-sidedCAD:iii}} \label{sec:corona-Ainfty:iii}
Assuming that $\Omega$ is a 1-sided CAD, we just need to apply \cite[Theorem~1.2]{AHMNT} for the Laplacian, \cite[Theorem~1.6]{HMMTZ} for the Kenig-Piper operators, \cite[Corollary ~10.3]{HMMTZ} for the analog class using oscillations, and \cite[Theorem~1.4]{CHMT} for the perturbations. Details are left to the interested reader. \qed

\end{document}